\documentclass[english]{article}
\usepackage[T1]{fontenc}
\usepackage[latin9]{inputenc}
\usepackage{geometry}
\usepackage{color}
\usepackage{babel}
\usepackage{mathrsfs}
\usepackage{amsmath}
\usepackage{amsthm}
\usepackage{amssymb}
\usepackage{stmaryrd}
\usepackage{stackrel}
\usepackage{graphicx}
\usepackage[all]{xy}
\usepackage[unicode=true,pdfusetitle,
 bookmarks=true,bookmarksnumbered=false,bookmarksopen=false,
 breaklinks=false,pdfborder={0 0 1},backref=false,colorlinks=true]
 {hyperref}

\makeatletter
\newcommand\thmsname{\protect\theoremname}
\newcommand\nm@thmtype{theorem}
\theoremstyle{plain}

\newenvironment{namedthm}[1][Undefined Theorem Name]{
  \ifx{#1}{Undefined Theorem Name}\renewcommand\nm@thmtype{theorem*}
  \else\renewcommand\thmsname{#1}\renewcommand\nm@thmtype{namedtheorem}
  \fi
  \begin{\nm@thmtype}}
  {\end{\nm@thmtype}}
\theoremstyle{plain}
\newtheorem{thm}{\protect\theoremname}[section]
\theoremstyle{remark}
\newtheorem{notation}[thm]{\protect\notationname}
\theoremstyle{remark}
\newtheorem*{acknowledgement*}{\protect\acknowledgementname}
\theoremstyle{definition}
\newtheorem{defn}[thm]{\protect\definitionname}
\theoremstyle{remark}
\newtheorem{rem}[thm]{\protect\remarkname}
\theoremstyle{plain}
\newtheorem{prop}[thm]{\protect\propositionname}
\theoremstyle{definition}
\newtheorem{condition}[thm]{\protect\conditionname}
\theoremstyle{plain}
\newtheorem{lem}[thm]{\protect\lemmaname}
\theoremstyle{definition}
\newtheorem{example}[thm]{\protect\examplename}
\theoremstyle{plain}
\newtheorem{cor}[thm]{\protect\corollaryname}
\newcommand{\lyxaddress}[1]{
	\par {\raggedright #1
	\vspace{1.4em}
	\noindent\par}
}

\usepackage[all]{xy}
\usepackage{mathpazo}

\makeatother

\providecommand{\acknowledgementname}{Acknowledgement}
\providecommand{\conditionname}{Condition}
\providecommand{\corollaryname}{Corollary}
\providecommand{\definitionname}{Definition}
\providecommand{\examplename}{Example}
\providecommand{\lemmaname}{Lemma}
\providecommand{\notationname}{Notation}
\providecommand{\propositionname}{Proposition}
\providecommand{\remarkname}{Remark}
\providecommand{\theoremname}{Theorem}

\begin{document}
\title{\textbf{\textsc{The Long-Moody construction and polynomial functors}}}
\author{\textsc{\textcolor{black}{Arthur Soulié}}}
\maketitle
\begin{abstract}
In 1994, Long and Moody gave a construction on representations of
braid groups which associates a representation of $\mathbf{B}_{n}$
with a representation of $\mathbf{B}_{n+1}$. In this paper, we prove
that this construction is functorial and can be extended: it inspires
endofunctors, called Long-Moody functors, between the category of
functors from Quillen's bracket construction associated with the braid
groupoid to a module category. Then we study the effect of Long-Moody
functors on strong polynomial functors: we prove that they increase
by one the degree of very strong polynomiality.
\end{abstract}
{\let\thefootnote\relax\footnotetext{Published in Annales de l'Institut Fourier, Volume 69 (2019) no. 4 p. 1799-1856.\\This work was partially supported by the ANR Project {\em ChroK}, {\tt ANR-16-CE40-0003}.}}

\section*{Introduction}

Linear representations of the Artin braid group on $n$ strands $\mathbf{B}_{n}$
is a rich subject which appears in diverse contexts in mathematics
(see for example \cite{BirmanBrendlesurvey} or \cite{Marin} for
an overview). Even if braid groups are of wild representation type,
any new result which allows us to gain a better understanding of them
is a useful contribution.

In $1994$, as a result of a collaboration with Moody in \cite{Long1},
Long gave a method to construct from a linear representation $\rho:\mathbf{B}_{n+1}\rightarrow GL\left(V\right)$
a new linear representation $\mathcal{LM}\left(\rho\right):\mathbf{B}_{n}\rightarrow GL\left(V^{\oplus n}\right)$
of $\mathbf{B}_{n}$ (see \cite[Theorem 2.1]{Long1}). Moreover, the
construction complicates in a sense the initial representation. For
example, applying it to a one dimensional representation of $\mathbf{B}_{n+1}$,
the construction gives a mild variant of the unreduced Burau representation
of $\mathbf{B}_{n}$. This method was in fact already implicitly present
in two previous articles of Long dated $1989$ (see \cite{Long2,Long3}).
In the article \cite{BigelowTian} dating from $2008$, Bigelow and
Tian consider the Long-Moody construction from a matricial point of
view. They give alternative and purely algebraic proofs of some results
of \cite{Long1}, and they slightly extend some of them. In a survey
on braid groups (see the Open Problem $7$ in \cite{BirmanBrendlesurvey}),
Birman and Brendle underline the fact that the Long-Moody construction
should be studied in greater detail.

Our work focuses on the study of the Long-Moody construction $\mathcal{LM}$
from a functorial point of view. More precisely, we consider the category
\textit{$\mathfrak{U}\boldsymbol{\beta}$} associated with braid groups.
This category is an example of a general construction due to Quillen
(see \cite{graysonQuillen}) on the braid groupoid $\boldsymbol{\beta}$.
In particular, the category \textit{$\mathfrak{U}\boldsymbol{\beta}$}
has natural numbers $\mathbb{N}$ as objects. For each natural number
$n$, the automorphism group $Aut_{\mathfrak{U}\boldsymbol{\beta}}\left(n\right)$
is the braid group $\mathbf{B}_{n}$. Let $\mathbb{K}\textrm{-}\mathfrak{Mod}$
be the category of $\mathbb{K}$-modules, with $\mathbb{K}$ a commutative
ring, and $\mathbf{Fct}\left(\mathfrak{U}\boldsymbol{\beta},\mathbb{K}\textrm{-}\mathfrak{Mod}\right)$
be the category of the functors from $\mathfrak{U}\boldsymbol{\beta}$
to $\mathbb{K}\textrm{-}\mathfrak{Mod}$. An object $M$ of $\mathbf{Fct}\left(\mathfrak{U}\boldsymbol{\beta},\mathbb{K}\textrm{-}\mathfrak{Mod}\right)$
gives by evaluation a family of representations of braid groups $\left\{ M_{n}:\mathbf{B}_{n}\rightarrow GL\left(M\left(n\right)\right)\right\} _{n\in\mathbb{N}}$,
which satisfies some compatibility properties (see Section \ref{subsec:Quillen's-bracket-construction}).
Randal-Williams and Wahl use the category $\mathfrak{U}\boldsymbol{\beta}$
in \cite{WahlRandal-Williams} to construct a general framework to
prove homological stability for braid groups with twisted coefficients.
Namely, they obtain the stability for twisted coefficients given by
objects of $\mathbf{Fct}\left(\mathfrak{U}\boldsymbol{\beta},\mathbb{K}\textrm{-}\mathfrak{Mod}\right)$.

In Proposition \ref{Thm:LMFunctor}, we prove that a version of the
Long-Moody construction is functorial. We fix two families of morphisms
$\left\{ a_{n}:\mathbf{B}_{n}\rightarrow Aut\left(\mathbf{F}_{n}\right)\right\} _{n\in\mathbb{N}}$
and $\left\{ \varsigma_{n}:\mathbf{F}_{n}\rightarrow\mathbf{B}_{n+1}\right\} _{n\in\mathbb{N}}$,
satisfying some coherence properties (see Section \ref{subsec:Braid-groups-and}).
Once this framework set, we show:
\begin{namedthm}[\textit{Theorem A }(Proposition \ref{Thm:LMFunctor}) ]
There is a functor $\mathbf{LM}_{a,\varsigma}:\mathbf{Fct}\left(\mathfrak{U}\boldsymbol{\beta},\mathbb{K}\textrm{-}\mathfrak{Mod}\right)\rightarrow\mathbf{Fct}\left(\mathfrak{U}\boldsymbol{\beta},\mathbb{K}\textrm{-}\mathfrak{Mod}\right)$,
called the Long-Moody functor with respect to coherent families of
morphisms $\left\{ a_{n}\right\} _{n\in\mathbb{N}}$ and $\left\{ \varsigma_{n}\right\} _{n\in\mathbb{N}}$,
which satisfies for $\sigma\in\mathbf{B}_{n}$ and $M\in Obj\left(\mathbf{Fct}\left(\mathfrak{U}\boldsymbol{\beta},\mathbb{K}\textrm{-}\mathfrak{Mod}\right)\right)$
\[
\mathbf{LM}_{a,\varsigma}\left(M\right)\left(\sigma\right)=\mathcal{LM}\left(M_{n}\right)\left(\sigma\right).
\]
\end{namedthm}
Among the objects in the category $\mathbf{Fct}\left(\mathfrak{U}\boldsymbol{\beta},\mathbb{K}\textrm{-}\mathfrak{Mod}\right)$
the strong polynomial functors play a key role. This notion extends
the classical one of polynomial functors, which were first defined
by Eilenberg and Mac Lane in \cite{EilenbergMacLane} for functors
on module categories, using cross effects. This definition can also
be applied to monoidal categories where the monoidal unit is a null
object. Djament and Vespa introduce in \cite{DV3} the definition
of strong polynomial functors for symmetric monoidal categories with
the monoidal unit being an initial object. Here, the category \textit{$\mathfrak{U}\boldsymbol{\beta}$}
is neither symmetric, nor braided, but pre-braided in the sense of
\cite{WahlRandal-Williams}. However, we show that the notion of strong
polynomial functor extends to the wider context of pre-braided monoidal
categories (see Definition \ref{def:defix}). We also introduce the
notion of very strong polynomial functor (see Definition \ref{def:defverystrong}).
Strong polynomial functors turn out inter alia to be very useful for
homological stability problems. For example, in \cite{WahlRandal-Williams},
Randal-Williams and Wahl prove their homological stability results
for twisted coefficients given by a specific kind of strong polynomial
functors, namely coefficient systems of finite degree (see \cite[Section 4.4]{WahlRandal-Williams}).

We investigate the effects of Long-Moody functors on very strong polynomial
functors. We establish the following theorem, under some mild additional
conditions (introduced in Section \ref{subsec:Additional-conditions})
on the families of morphisms \textit{$\left\{ a_{n}\right\} _{n\in\mathbb{N}}$
}and \textit{$\left\{ \varsigma_{n}\right\} _{n\in\mathbb{N}}$},
which are then said to be reliable.
\begin{namedthm}[\textit{Theorem B} (Corollary \ref{thm:Main result2}) ]
Let $M$ be a very strong polynomial functor of $\mathbf{Fct}\left(\mathfrak{U}\boldsymbol{\beta},\mathbb{K}\textrm{-}\mathfrak{Mod}\right)$
of degree $n$ and let $\left\{ a_{n}\right\} _{n\in\mathbb{N}}$
and $\left\{ \varsigma_{n}\right\} _{n\in\mathbb{N}}$ be coherent
reliable families of morphisms. Then, considering the Long-Moody functor
$\mathbf{LM}_{a,\varsigma}$ with respect to the morphisms $\left\{ a_{n}\right\} _{n\in\mathbb{N}}$
and $\left\{ \varsigma_{n}\right\} _{n\in\mathbb{N}}$, $\mathbf{LM}_{a,\varsigma}\left(M\right)$
is a very strong polynomial functor of degree $n+1$.
\end{namedthm}
Thus, iterating the Long-Moody functor on a very strong polynomial
functor of \textit{$\mathbf{Fct}\left(\mathfrak{U}\boldsymbol{\beta},\mathbb{K}\textrm{-}\mathfrak{Mod}\right)$}
of degree $d$, we generate polynomial functors of \textit{$\mathbf{Fct}\left(\mathfrak{U}\boldsymbol{\beta},\mathbb{K}\textrm{-}\mathfrak{Mod}\right)$},
of any degree bigger than $d$. For instance, Randal-Williams and
Wahl define in \cite[Example 4.3]{WahlRandal-Williams} a functor
$\mathfrak{Bur}_{t}:\mathfrak{U}\boldsymbol{\beta}\rightarrow\mathbb{C}\left[t^{\pm1}\right]\textrm{-}\mathfrak{Mod}$
encoding the unreduced Burau representations. Similarly, we introduce
a functor $\mathfrak{TYM}_{t}:\mathfrak{U}\boldsymbol{\beta}\rightarrow\mathbb{C}\left[t^{\pm1}\right]\textrm{-}\mathfrak{Mod}$
corresponding to the representations considered by Tong, Yang and
Ma in \cite{TYM}. These functors $\mathfrak{Bur}_{t}$ and $\mathfrak{TYM}_{t}$
are very strong polynomial of degree one (see Proposition \ref{prop:BurTYMverystrong}),
and moreover, we prove that the functor $\mathfrak{Bur}_{t}$ is equivalent
to a functor obtained by applying the Long-Moody construction. Thus,
the Long-Moody functors will provide new examples of twisted coefficients
corresponding to the framework of \cite{WahlRandal-Williams}.

This construction is extended in the forthcoming work \cite{soulieLMgeneralized}
for other families of groups, such as automorphism groups of free
groups, braid groups of surfaces, mapping class groups of orientable
and non-orientable surfaces or mapping class groups of 3-manifolds.
The results proved here for (very) strong polynomial functors will
also hold in the adapted categorical framework for these different
families of groups.

The paper is organized as follows. Following \cite{WahlRandal-Williams},
Section \ref{sec:Recollections-on-the} introduces the category $\mathfrak{U}\boldsymbol{\beta}$
and gives first examples of objects of \textit{$\mathbf{Fct}\left(\mathfrak{U}\boldsymbol{\beta},\mathbb{K}\textrm{-}\mathfrak{Mod}\right)$}.
Then, in Section \ref{sec:Functoriality-of-the}, we introduce the
Long-Moody functors, prove Theorem $A$ and give some of their properties.
In Section \ref{sec:Strong-polynomial-functors}, we review the notion
of strong polynomial functors and extend the framework of \cite{DV3}
to pre-braided monoidal categories. Finally, Section \ref{sec:The-Long-Moody-functoreffect}
is devoted to the proof of Theorem $B$ and to some other properties
of these functors. In particular, we tackle the Open Problem $7$
of \cite{BirmanBrendlesurvey}.
\begin{notation}
\label{nota:notationsgenerales}We will consider a commutative ring
$\mathbb{K}$ throughout this work. We denote by $\mathbb{K}\textrm{-}\mathfrak{Mod}$
the category of $\mathbb{K}$-modules. We denote by $\mathfrak{Gr}$
the category of groups. We take the convention that the set of natural
numbers $\mathbb{N}$ is the set of nonnegative integers $\left\{ 0,1,2,\ldots\right\} $.

Let $\mathfrak{Cat}$ denote the category of small categories. Let
$\mathfrak{C}$ be an object of $\mathfrak{Cat}$. We use the abbreviation
$Obj\left(\mathfrak{C}\right)$ to denote the objects of $\mathfrak{C}$.
For $\mathfrak{D}$ a category, we denote by $\mathbf{Fct}\left(\mathfrak{C},\mathfrak{D}\right)$
the category of functors from $\mathfrak{C}$ to $\mathfrak{D}$.
If $0$ is initial object in the category $\mathfrak{C}$, then we
denote by $\iota_{A}:0\rightarrow A$ the unique morphism from $0$
to $A$. The maximal subgroupoid $\mathscr{G}\mathfrak{r}\left(\mathfrak{C}\right)$
is the subcategory of $\mathfrak{C}$ which has the same objects as
$\mathfrak{C}$ and of which the morphisms are the isomorphisms of
$\mathfrak{C}$. We denote by $\mathscr{G}\mathfrak{r}:\mathfrak{Cat}\rightarrow\mathfrak{Cat}$
the functor which associates to a category its maximal subgroupoid.
\end{notation}

\begin{acknowledgement*}
The author wishes to thank most sincerely his PhD advisor Christine
Vespa, and Geoffrey Powell, for their careful reading, corrections,
valuable help and expert advice. He would also especially like to
thank Aurélien Djament, Nariya Kawazumi, Martin Palmer, Vladimir Verchinine
and Nathalie Wahl for the attention they have paid to his work, their
comments, suggestions and helpful discussions. Additionally, he would
like to thank the anonymous referee for his reading of this paper.
\end{acknowledgement*}
\tableofcontents{}

\section{The category $\mathfrak{U}\boldsymbol{\beta}$\label{sec:Recollections-on-the}}

The aim of this section is to describe the category $\mathfrak{U}\boldsymbol{\beta}$
associated with braid groups that is central to this paper. On the
one hand, we recall some notions and properties about Quillen's construction
from a monoidal groupoid and pre-braided monoidal categories introduced
by Randal-Williams and Wahl in \cite{WahlRandal-Williams}. On the
other hand, we introduce examples of functors over the category $\mathfrak{U}\boldsymbol{\beta}$.

We recall that the braid group on $n\geq2$ strands denoted by $\mathbf{B}_{n}$
is the group generated by $\sigma_{1}$, ..., $\sigma_{n-1}$ satisfying
the relations:
\begin{itemize}
\item $\forall i\in\left\{ 1,\ldots,n-2\right\} $, $\sigma_{i}\sigma_{i+1}\sigma_{i}=\sigma_{i+1}\sigma_{i}\sigma_{i+1}$;
\item $\forall i,j\in\left\{ 1,\ldots,n-1\right\} $ such that $\mid i-j\mid\geq2$,
$\sigma_{i}\sigma_{j}=\sigma_{j}\sigma_{i}$.
\end{itemize}
$\mathbf{B}_{0}$ and $\mathbf{B}_{1}$ both are the trivial group.
The family of braid groups is associated with the following groupoid.
\begin{defn}
\label{def:defbraidgroupoid}The braid groupoid $\boldsymbol{\beta}$
is the groupoid with objects the natural numbers $n\in\mathbb{N}$
and morphisms (for $n,m\in\mathbb{N}$):
\[
Hom_{\boldsymbol{\beta}}\left(n,m\right)=\begin{cases}
\mathbf{B}_{n} & \textrm{if \ensuremath{n=m}}\\
\emptyset & \textrm{if \ensuremath{n\neq m}}.
\end{cases}
\]
\end{defn}

\begin{rem}
The composition of morphisms $\circ$ in the groupoid $\boldsymbol{\beta}$
corresponds to the group operation of the braid groups. So we will
abuse the notation throughout this work, identifying $\sigma\circ\sigma'=\sigma\sigma'$
for all elements $\sigma$ and $\sigma'$ of $\mathbf{B}_{n}$ with
$n\in\mathbb{N}$ (with the convention that we read from the right
to the left for the group operation).
\end{rem}

\subsection{Quillen's bracket construction associated with the groupoid $\boldsymbol{\beta}$\label{subsec:Quillen's-bracket-construction}}

This section focuses on the presentation and the study of Quillen's
bracket construction $\mathfrak{U}\boldsymbol{\beta}$ (see \cite[p.219]{graysonQuillen})
on the braid groupoid $\boldsymbol{\beta}$. It associates to $\boldsymbol{\beta}$
a monoidal category whose unit is initial. The category $\mathfrak{U}\boldsymbol{\beta}$
has further properties: Quillen's bracket construction on $\boldsymbol{\beta}$
is a pre-braided monoidal category (see Section \ref{subsec:Pre-braided-monoidal-categories})
and $\boldsymbol{\beta}$ is its maximal subgroupoid. For an introduction
to (braided) strict monoidal categories, we refer to \cite[Chapter XI]{MacLane1}.
\begin{notation}
A strict monoidal category will be denoted by $\left(\mathfrak{C},\natural,0\right)$,
where $\mathfrak{C}$ is the category, $\natural$ is the monoidal
product and $0$ is the monoidal unit.
\end{notation}

\subsubsection{Generalities}

In \cite{WahlRandal-Williams}, Randal-Williams and Wahl study a construction
due to Quillen in \cite[p.219]{graysonQuillen}, for a monoidal category
$S$ acting on a category $X$ in the case $S=X=\mathfrak{G}$ where
$\mathfrak{G}$ is a groupoid. It is called Quillen's bracket construction.
Our study here is based on \cite[Section 1]{WahlRandal-Williams}
taking $\mathfrak{G}=\boldsymbol{\beta}$.
\begin{defn}
\cite[Chapter XI, Section 4]{MacLane1} \label{exa:defprodmonUbeta}A
monoidal product $\natural:\boldsymbol{\beta}\times\boldsymbol{\beta}\longrightarrow\boldsymbol{\beta}$
is defined by the usual addition for the objects and laying two braids
side by side for the morphisms. The object $0$ is the unit of this
monoidal product. The strict monoidal groupoid $\left(\boldsymbol{\beta},\natural,0\right)$
is braided, its braiding is denoted by $b_{-,-}^{\boldsymbol{\beta}}$.
Namely, the braiding is defined for all natural numbers $n$ and $m$
such that $n+m\geq2$ by:
\[
b_{n,m}^{\boldsymbol{\beta}}=\left(\sigma_{m}\circ\cdots\circ\sigma_{2}\circ\sigma_{1}\right)\circ\cdots\circ\left(\sigma_{n+m-2}\circ\cdots\circ\sigma_{n}\circ\sigma_{n-1}\right)\circ\left(\sigma_{n+m-1}\circ\cdots\circ\sigma_{n+1}\circ\sigma_{n}\right)
\]
where $\left\{ \sigma_{i}\right\} _{i\in\left\{ 1,\ldots,n+m-1\right\} }$
denote the Artin generators of the braid group $\mathbf{B}_{n+m}$.
\end{defn}

We consider the strict monoidal groupoid $\left(\boldsymbol{\beta},\natural,0\right)$
throughout this section.
\begin{defn}
\label{def:defUB}\cite[Section 1.1]{WahlRandal-Williams} Quillen's
bracket construction on the groupoid $\boldsymbol{\beta}$, denoted
by $\mathfrak{U\boldsymbol{\beta}}$, is the category defined by:

\begin{itemize}
\item Objects: $Obj\left(\mathfrak{U\boldsymbol{\beta}}\right)=Obj\left(\mathfrak{\boldsymbol{\beta}}\right)=\mathbb{N}$;
\item Morphisms: for $n$ and $n'$ two objects of $\boldsymbol{\beta}$,
the morphisms from $n$ to $n'$ in the category $\mathfrak{U\boldsymbol{\beta}}$
are given by:
\[
Hom_{\mathfrak{U\boldsymbol{\beta}}}\left(n,n'\right)=\underset{\mathfrak{\boldsymbol{\beta}}}{colim}\left[Hom_{\mathfrak{\boldsymbol{\beta}}}\left(-\natural n,n'\right)\right].
\]
In other words, a morphism from $n$ to $n'$ in the category $\mathfrak{U\boldsymbol{\beta}}$,
denoted by $\left[n'-n,f\right]:n\rightarrow n'$, is an equivalence
class of pairs $\left(n'-n,f\right)$ where $n'-n$ is an object of
$\mathfrak{\boldsymbol{\beta}}$, $f:\left(n'-n\right)\natural n\rightarrow n'$
is a morphism of $\boldsymbol{\beta}$, in other words an element
of $\mathbf{B}_{n'}$. The equivalence relation $\sim$ is defined
by $\left(n'-n,f\right)\sim\left(n'-n,f'\right)$ if and only if there
exists an automorphism $g\in Aut_{\boldsymbol{\beta}}\left(n'-n\right)$
such that the following diagram commutes.
\[
\xymatrix{\left(n'-n\right)\natural n\ar[d]_{g\natural id_{n}}\ar[r]^{\,\,\,\,\,\,\,\,\,\,\,\,\,\,\,\,f} & n'\\
\left(n'-n\right)\natural n\ar[ur]_{f'}
}
\]
\item For all objects $n$ of $\mathfrak{U\boldsymbol{\beta}}$, the identity
morphism in the category $\mathfrak{U\boldsymbol{\beta}}$ is given
by $\left[0,id_{n}\right]:n\rightarrow n$.
\item Let $\left[n'-n,f\right]:n\rightarrow n'$ and $\left[n''-n',g\right]:n'\rightarrow n''$
be two morphisms in the category $\mathfrak{U\boldsymbol{\beta}}$.
Then, the composition in the category $\mathfrak{U\boldsymbol{\beta}}$
is defined by:
\[
\left[n''-n',g\right]\circ\left[n'-n,f\right]=\left[n''-n,g\circ\left(id_{n'-n}\natural f\right)\right].
\]
\end{itemize}
\end{defn}

The relationship between the automorphisms of the groupoid $\boldsymbol{\beta}$
and those of its associated Quillen's construction $\mathfrak{U}\boldsymbol{\beta}$
is actually clear. First, let us recall the following notion.
\begin{defn}
\label{def:nozerodivisors}Let $\left(\mathfrak{G},\natural,0\right)$
be a strict monoidal category. It has no zero divisors if for all
objects $A$ and $B$ of $\mathfrak{G}$, $A\natural B\cong0$ if
and only if $A\cong B\cong0$.
\end{defn}

The braid groupoid $\left(\boldsymbol{\beta},\natural,0\right)$ has
no zero divisors. Moreover, by Definition \ref{def:defbraidgroupoid},
$Aut_{\boldsymbol{\beta}}(0)=\left\{ id_{0}\right\} $. Hence, we
deduce the following property from \cite[Proposition 1.7]{WahlRandal-Williams}.
\begin{prop}
The groupoid $\boldsymbol{\beta}$ is the maximal subgroupoid of $\mathfrak{U\boldsymbol{\beta}}$.
\end{prop}

In addition, $\mathfrak{U\boldsymbol{\beta}}$ has the additional
useful property.
\begin{prop}
\cite[Proposition 1.8 (i)]{WahlRandal-Williams} The unit $0$ of
the monoidal structure of the groupoid $\left(\mathfrak{\boldsymbol{\beta}},\natural,0\right)$
is an initial object in the category $\mathfrak{U\boldsymbol{\beta}}$.
\end{prop}

\begin{rem}
Let $n$ be a natural number and $\phi\in Aut_{\mathfrak{\boldsymbol{\beta}}}\left(n\right)$.
Then, as an element of $Hom_{\mathfrak{U\boldsymbol{\beta}}}\left(n,n\right)$,
we will abuse the notation $\phi=\left[0,\phi\right]$. This comes
from the canonical functor:
\begin{eqnarray*}
\boldsymbol{\beta} & \rightarrow & \mathfrak{U\boldsymbol{\beta}}\\
\phi\in Aut_{\mathfrak{\boldsymbol{\beta}}}\left(n\right) & \mapsto & \left[0,\phi\right].
\end{eqnarray*}
\end{rem}

Finally, we are interested in a way to extend an object of $\mathbf{Fct}\left(\boldsymbol{\beta},\mathbb{K}\textrm{-}\mathfrak{Mod}\right)$
to an object of $\mathbf{Fct}\left(\mathfrak{U}\boldsymbol{\beta},\mathbb{K}\textrm{-}\mathfrak{Mod}\right)$.
This amounts to studying the image of the restriction $\mathbf{Fct}\left(\mathfrak{U}\boldsymbol{\beta},\mathbb{K}\textrm{-}\mathfrak{Mod}\right)\rightarrow\mathbf{Fct}\left(\boldsymbol{\beta},\mathbb{K}\textrm{-}\mathfrak{Mod}\right)$.
\begin{prop}
\label{prop:criterionfamilymorphismsfunctor}Let $M$ be an object
of $\mathbf{Fct}\left(\boldsymbol{\beta},\mathbb{K}\textrm{-}\mathfrak{Mod}\right)$.
Assume that for all $n,n',n''\in\mathbb{N}$ such that $n''\geq n'\geq n$,
there exists an assignment $M\left(\left[n'-n,id_{n'}\right]\right):M\left(n\right)\rightarrow M\left(n'\right)$
such that:
\begin{equation}
M\left(\left[n''-n',id_{n''}\right]\right)\circ M\left(\left[n'-n,id_{n'}\right]\right)=M\left(\left[n''-n,id_{n''}\right]\right)\label{eq:criterion2}
\end{equation}
Then, we define a functor $M:\mathfrak{U}\boldsymbol{\beta}\rightarrow\mathbb{K}\textrm{-}\mathfrak{Mod}$
(assigning $M\left(\left[n'-n,\sigma\right]\right)=M\left(\sigma\right)\circ M\left(\left[n'-n,id_{n'}\right]\right)$
for all $\left[n'-n,\sigma\right]\in Hom_{\mathfrak{U}\boldsymbol{\beta}}\left(n,n'\right)$)
if and only if for all $n,n'\in\mathbb{N}$ such that $n'\geq n$:
\begin{equation}
M\left(\left[n'-n,id_{n'}\right]\right)\circ M\left(\sigma\right)=M\left(\psi\natural\sigma\right)\circ M\left(\left[n'-n,id_{n'}\right]\right)\label{eq:criterion}
\end{equation}
for all $\sigma\in\mathbf{B}_{n}$ and all $\psi\in\mathbf{B}_{n'-n}$.
\end{prop}

\begin{rem}
Note that for $n'=n$, $M\left(\left[n'-n,id_{n'}\right]\right)=Id_{M\left(n\right)}$.
\end{rem}

\begin{proof}
[Proof of Proposition 1.10]Let us assume that relation (\ref{eq:criterion})
is satisfied. We have to show that the assignment on morphisms is
well-defined with respect to $\mathfrak{U}\boldsymbol{\beta}$. First,
let us prove that our assignment conforms with the defining equivalence
relation of $\mathfrak{U}\boldsymbol{\beta}$ (see Definition \ref{def:defUB}).
For $n$ and $n'$ natural numbers such that $n'\geq n$, let us consider
$\left[n'-n,\sigma\right]$ and $\left[n'-n,\sigma'\right]$ in $Hom_{\mathfrak{U}\boldsymbol{\beta}}\left(n,n'\right)$
such that there exists $\psi\in\mathbf{B}_{n'-n}$ so that $\sigma'\circ\left(\psi\natural id_{n}\right)=\sigma$.
Since $M$ is a functor over $\boldsymbol{\beta}$, $M\left(\left[n'-n,\sigma\right]\right)=M\left(\sigma'\right)\circ\left(M\left(\psi\natural id_{n}\right)\circ M\left(\left[n'-n,id_{n'}\right]\right)\right)$.
According to the relation (\ref{eq:criterion}) and since $M$ satisfies
the identity axiom, we deduce that $M\left(\left[n'-n,\sigma\right]\right)=M\left(\sigma'\right)\circ M\left(\psi\natural id_{n}\right)\circ M\left(\left[n'-n,id_{n'}\right]\right)=M\left(\left[n'-n,\sigma'\right]\right)$.

Now, we have to check the composition axiom. Let $n$, $n'$ and $n''$
be natural numbers such that $n''\geq n'\geq n$, let $\left(\left[n'-n,\sigma\right]\right)$
and $\left(\left[n''-n',\sigma'\right]\right)$ be morphisms respectively
in $Hom_{\mathfrak{U}\boldsymbol{\beta}}\left(n,n'\right)$ and in
$Hom_{\mathfrak{U}\boldsymbol{\beta}}\left(n',n''\right)$. By our
assignment and composition in $\mathfrak{U}\boldsymbol{\beta}$ (see
Definition \ref{def:defUB}) we have that:
\[
M\left(\left[n''-n',\sigma'\right]\right)\circ M\left(\left[n'-n,\sigma\right]\right)=M\left(\sigma'\right)\circ\left(M\left(\left[n''-n',id_{n''}\right]\right)\circ M\left(\sigma\right)\right)\circ M\left(\left[n'-n,id_{n'}\right]\right).
\]
According to the relation (\ref{eq:criterion}), we deduce that:
\begin{eqnarray*}
M\left(\left[n''-n',\sigma'\right]\right)\circ M\left(\left[n'-n,\sigma\right]\right) & = & M\left(\sigma'\right)\circ\left(M\left(\left[n''-n',id_{n''}\right]\right)\circ M\left(\sigma\right)\right)\circ M\left(\left[n'-n,id_{n'}\right]\right).\\
 & = & M\left(\sigma'\right)\circ\left(M\left(id_{n''-n'}\natural\sigma\right)\circ M\left(\left[n''-n',id_{n''}\right]\right)\right)\circ M\left(\left[n'-n,id_{n'}\right]\right).
\end{eqnarray*}
Hence, it follows from relation (\ref{eq:criterion2}) that:
\[
M\left(\left[n''-n',\sigma'\right]\right)\circ M\left(\left[n'-n,\sigma\right]\right)=M\left(\sigma'\circ\left(id_{n''-n'}\natural\sigma\right)\right)\circ M\left(\left[n''-n,id_{n}\right]\right)=M\left(\left[n''-n',\sigma'\right]\circ\left[n'-n,\sigma\right]\right).
\]

Conversely, assume that the functor $M:\mathfrak{U}\boldsymbol{\beta}\rightarrow\mathbb{K}\textrm{-}\mathfrak{Mod}$
is well-defined. In particular, composition axiom in $\mathfrak{U}\boldsymbol{\beta}$
is satisfied and implies that for all $n,n'\in\mathbb{N}$ such that
$n'\geq n$, for all $\sigma\in\mathbf{B}_{n}$:
\[
M\left(\left[n'-n,id_{n'}\right]\right)\circ M\left(\sigma\right)=M\left(\left[n'-n,id_{n'-n}\natural\sigma\right]\right).
\]
It follows from the defining equivalence relation of $\mathfrak{U}\boldsymbol{\beta}$
(see Definition (\ref{def:defUB})) that for all $\psi\in\mathbf{B}_{n'-n}$:
\[
M\left(\left[n'-n,id_{n'}\right]\right)\circ M\left(\sigma\right)=M\left(\left[n'-n,\psi\natural\sigma\right]\right).
\]
We deduce from the composition axiom that relation (\ref{eq:criterion})
is satisfied.
\end{proof}
\begin{prop}
\label{prop:criterionnaturaltransfo}Let $M$ and $M'$ be objects
of $\mathbf{Fct}\left(\mathfrak{U}\boldsymbol{\beta},\mathbb{K}\textrm{-}\mathfrak{Mod}\right)$
and $\eta:M\rightarrow M'$ a natural transformation in the category
$\mathbf{Fct}\left(\boldsymbol{\beta},\mathbb{K}\textrm{-}\mathfrak{Mod}\right)$.
Then, $\eta$ is a natural transformation in the category $\mathbf{Fct}\left(\mathfrak{U}\boldsymbol{\beta},\mathbb{K}\textrm{-}\mathfrak{Mod}\right)$
if and only if for all $n,n'\in\mathbb{N}$ such that $n'\geq n$:
\begin{equation}
\eta_{n'}\circ M\left(\left[n'-n,id_{n'}\right]\right)=M'\left(\left[n'-n,id_{n'}\right]\right)\circ\eta_{n}.\label{eq:criterion3}
\end{equation}
\end{prop}

\begin{proof}
The natural transformation $\eta$ extends to the category $\mathbf{Fct}\left(\mathfrak{U}\boldsymbol{\beta},\mathbb{K}\textrm{-}\mathfrak{Mod}\right)$
if and only if for all $n,n'\in\mathbb{N}$ such that $n'\geq n$,
for all $\left[n'-n,\sigma\right]\in Hom_{\mathfrak{U}\boldsymbol{\beta}}\left(n,n'\right)$:
\[
M'\left(\left[n'-n,\sigma\right]\right)\circ\eta_{n}=\eta_{n'}\circ M\left(\left[n'-n,\sigma\right]\right).
\]
Since $\eta$ is a natural transformation in the category $\mathbf{Fct}\left(\boldsymbol{\beta},\mathbb{K}\textrm{-}\mathfrak{Mod}\right)$,
we already have $\eta_{n'}\circ M\left(\sigma\right)=M'\left(\sigma\right)\circ\eta_{n'}$.
Hence, this implies that the necessary and sufficient relation to
satisfy is relation (\ref{eq:criterion3}).
\end{proof}

\subsubsection{Pre-braided monoidal category\label{subsec:Pre-braided-monoidal-categories}}

We present the notion of a pre-braided category, introduced by Randal-Williams
and Wahl in \cite{WahlRandal-Williams}. This is a generalization
of that of a braided monoidal category\textit{.}
\begin{defn}
\cite[Definition 1.5]{WahlRandal-Williams}\label{def:defprebraided}
Let $\left(\mathfrak{C},\natural,0\right)$ be a strict monoidal category
such that the unit $0$ is initial. We say that the monoidal category
$\left(\mathfrak{C},\natural,0\right)$ is pre-braided if:

\begin{itemize}
\item The maximal subgroupoid $\mathscr{G}\mathfrak{r}\left(\mathfrak{C},\natural,0\right)$
is a braided monoidal category, where the monoidal structure is induced
by that of $\left(\mathfrak{C},\natural,0\right)$.
\item For all objects $A$ and $B$ of $\mathfrak{C}$, the braiding associated
with the maximal subgroupoid $b_{A,B}^{\mathfrak{C}}:A\natural B\longrightarrow B\natural A$
satisfies:
\[
b_{A,B}^{\mathfrak{C}}\circ\left(id_{A}\natural\iota_{B}\right)=\iota_{B}\natural id_{A}:A\longrightarrow B\natural A.
\]
Recall that the notation $\iota_{B}:0\rightarrow B$ was introduced
in Notation \ref{nota:notationsgenerales}.
\end{itemize}
Since the groupoid $\left(\boldsymbol{\beta},\natural,0\right)$ is
braided monoidal and it has no zero divisors, we deduce from \cite[Proposition 1.8]{WahlRandal-Williams}
the following properties.
\begin{figure}
\textcolor{white}{~~~~~~~~~~~~~~~~~~~~~~~~~~~~~~~~~~~~~~~~~~~~~~~~~~~~}\includegraphics[clip,scale=0.14]{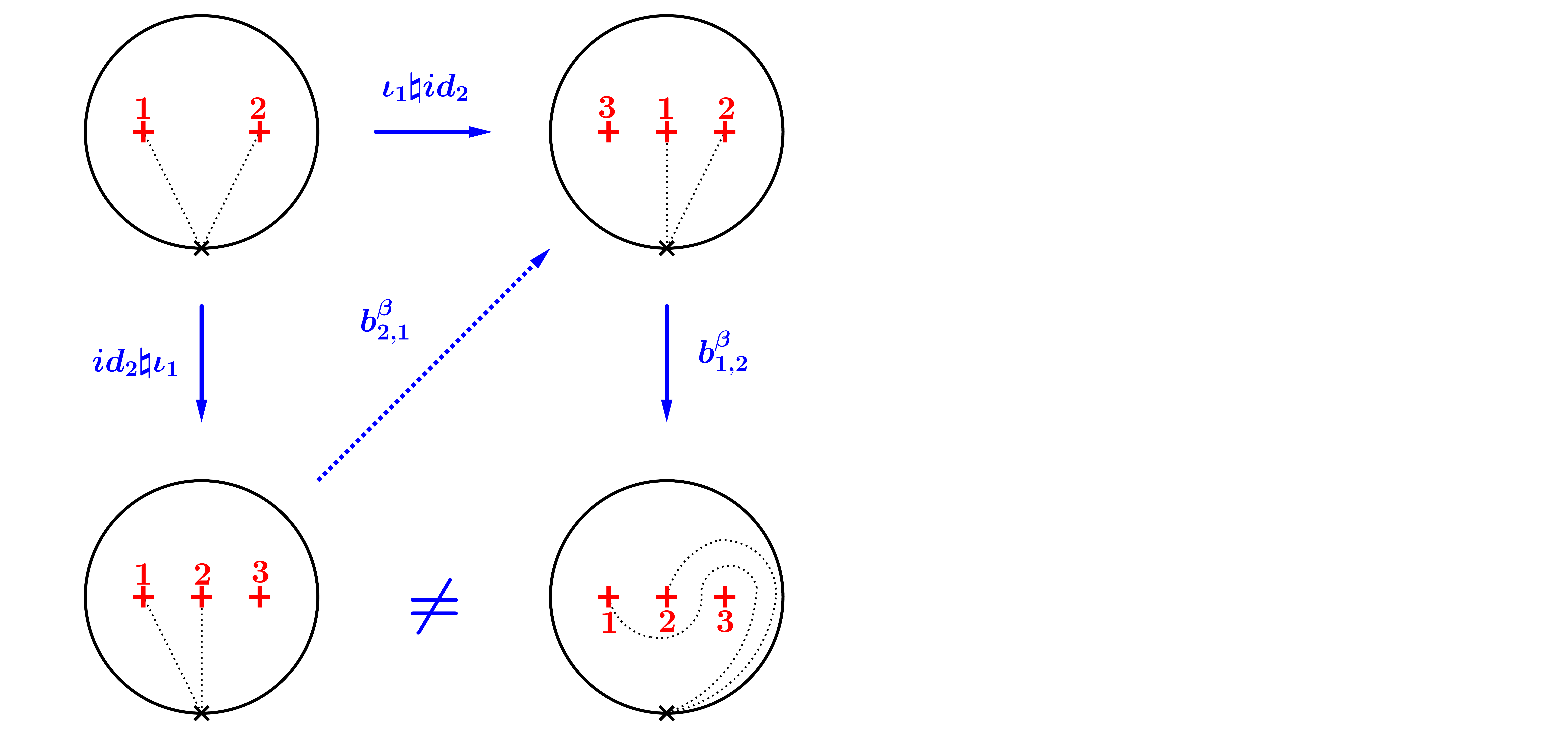}

\caption{Failure of the braiding property}
\end{figure}
\end{defn}

\begin{prop}
\label{prop:homogenousprebraided} The category $\mathfrak{U\boldsymbol{\beta}}$
is pre-braided monoidal. The monoidal structure $\left(\mathfrak{U\boldsymbol{\beta}},\natural,0\right)$
is defined on objects as that of $\left(\boldsymbol{\beta},\natural,0\right)$
and defined on morphisms letting for $\left[n'-n,f\right]\in Hom_{\mathfrak{U\boldsymbol{\beta}}}\left(n,n'\right)$
and $\left[m'-m,g\right]\in Hom_{\mathfrak{U\boldsymbol{\beta}}}\left(m,m'\right)$:
\[
\left[m'-m,g\right]\natural\left[n'-n,f\right]=\left[\left(m'-m\right)\natural\left(n'-n\right),\left(g\natural f\right)\circ\left(id_{m'-m}\natural\left(b_{m,n'-n}^{\mathfrak{\boldsymbol{\beta}}}\right)^{-1}\natural id_{n}\right)\right].
\]
In particular, the canonical functor $\mathfrak{\boldsymbol{\beta}}\rightarrow\mathfrak{U\boldsymbol{\beta}}$
is monoidal.
\end{prop}

\begin{rem}
The category $\left(\mathfrak{U\boldsymbol{\beta}},\natural,0\right)$
is pre-braided monoidal, but not braided. Indeed, as Figure $1$ shows,
the pre-braiding defined on $\mathfrak{U}\boldsymbol{\beta}$ is not
a braiding: Figure $1$ shows that $b_{1,2}^{\boldsymbol{\beta}}\circ\left(\iota_{1}\natural id{}_{2}\right)\neq id_{2}\natural\iota_{1}$
whereas these two morphisms should be equal if $b_{-,-}^{\boldsymbol{\beta}}$
were a braiding.
\end{rem}

\subsection{Examples of functors associated with braid representations\label{subsec:Examples-of-functors}}

Different families of representations of braid groups can be used
to form functors over the pre-braided category $\mathfrak{U}\boldsymbol{\beta}$
to the category $\mathbb{K}\textrm{-}\mathfrak{Mod}$. Namely, considering
$\left\{ M_{n}:\mathbf{B}_{n}\rightarrow\mathbb{K}\textrm{-}\mathfrak{Mod}\right\} _{n\in\mathbb{N}}$
representations of braid groups, or equivalently an object $M$ of
$\mathbf{Fct}\left(\boldsymbol{\beta},\mathbb{K}\textrm{-}\mathfrak{Mod}\right)$,
we are interested in the situations where Proposition \ref{prop:criterionfamilymorphismsfunctor}
applies so as to define an object of $\mathbf{Fct}\left(\mathfrak{U}\boldsymbol{\beta},\mathbb{K}\textrm{-}\mathfrak{Mod}\right)$.

\paragraph{Tong-Yang-Ma results}

In $1996$, in the article \cite{TYM}, Tong, Yang and Ma investigated
the representations of $\mathbf{B}{}_{n}$ where the $i$-th generator
is sent to a matrix of the form $Id_{i-1}\oplus T\oplus Id_{n-i-1}$,
with $T$ a $m\times m$ non-singular matrix and $m\geq2$. In particular,
for $m=2$, they prove that there exist up to equivalence only two
non trivial representations of this type. We give here their result
and an interpretation of their work from a functorial point of view,
considering the representations over the ring of Laurent polynomials
in one variable $\mathbb{C}\left[t^{\pm1}\right]$.
\begin{notation}
\label{nota:grandN}Let $\mathfrak{gr}$ denote the full subcategory
of $\mathfrak{Gr}$ of finitely generated free groups. The free product
$*:\mathfrak{gr}\times\mathfrak{gr}\rightarrow\mathfrak{gr}$ defines
a monoidal structure over $\mathfrak{gr}$, with $0$ the unit, denoted
by $\left(\mathfrak{gr},*,0\right)$.

Let $\left(\mathbb{N},\leq\right)$ denote the category of natural
numbers (natural means non-negative) considered as a poset. For all
natural numbers $n$, we denote by $\gamma_{n}$ the unique element
of $Hom_{\left(\mathbb{N},\leq\right)}\left(n,n+1\right)$. For all
natural numbers $n$ and $n'$ such that $n'\geq n$, we denote by
$\gamma{}_{n,n'}:n\rightarrow n'$ the unique element of $Hom_{\left(\mathbb{N},\leq\right)}\left(n,n'\right)$,
composition of the morphisms $\gamma{}_{n'-1}\circ\gamma{}_{n'-2}\circ\cdots\circ\gamma{}_{n+1}\circ\gamma{}_{n}$.
The addition defines a strict monoidal structure on $\left(\mathbb{N},\leq\right)$,
denoted by $\left(\left(\mathbb{N},\leq\right),+,0\right)$.
\end{notation}

\begin{defn}
\label{def:functorBandGL}Let $\mathbf{B}_{-}:\left(\mathbb{N},\leq\right)\rightarrow\mathfrak{Gr}$
and $GL_{-}:\left(\mathbb{N},\leq\right)\rightarrow\mathfrak{Gr}$
be the functors defined by:

\begin{itemize}
\item Objects: for all natural numbers $n$, $\mathbf{B}_{-}\left(n\right)=\mathbf{B}_{n}$
the braid group on $n$ strands and $GL_{-}\left(n\right)=GL_{n}\left(\mathbb{C}\left[t^{\pm1}\right]\right)$
the general linear group of degree $n$.
\item Morphisms: let $n$ be a natural number. We define $\mathbf{B}_{-}\left(\gamma{}_{n}\right)=id_{1}\natural-:\mathbf{B}_{n}\hookrightarrow\mathbf{B}_{n+1}$
(where $\natural$ is the monoidal product introduced in Example \ref{exa:defprodmonUbeta}).
We define $GL_{-}\left(\gamma{}_{n}\right):GL_{n}\left(\mathbb{C}\left[t^{\pm1}\right]\right)\hookrightarrow GL_{n+1}\left(\mathbb{C}\left[t^{\pm1}\right]\right)$
assigning $GL_{-}\left(\gamma{}_{n}\right)\left(\varphi\right)=id_{1}\oplus\varphi$
for all $\varphi\in GL_{n}\left(\mathbb{C}\left[t^{\pm1}\right]\right)$. 
\end{itemize}
\end{defn}

\begin{notation}
\label{nota:defincl}For all natural numbers $n\geq2$, for all $i\in\left\{ 1,\ldots,n-1\right\} $,
we denote by $incl_{i}^{n}:\mathbf{B}_{2}\cong\mathbb{Z}\hookrightarrow\mathbf{B}_{n}$
the inclusion morphism induced by:
\[
incl_{i}^{n}\left(\sigma_{1}\right)=\sigma_{i}.
\]
\end{notation}

\begin{thm}
\label{thm:resulttym}\cite[Part II]{TYM} Let $\eta:\mathbf{B}_{-}\longrightarrow GL_{-}$
be a natural transformation. Assume that for all natural numbers $n\geq2$,
for all $i\in\left\{ 1,\ldots,n-1\right\} $, the following diagram
is commutative:
\[
\xymatrix{\mathbf{B}_{n}\ar@{->}[rr]^{\eta_{n}} &  & GL_{n}\left(\mathbb{C}\left[t^{\pm1}\right]\right)\\
\mathbf{B}_{2}\ar@{->}[rr]_{\eta_{2}}\ar@{->}[u]^{incl_{i}^{n}} &  & GL_{2}\left(\mathbb{C}\left[t^{\pm1}\right]\right).\ar@{->}[u]_{id_{i-1}\oplus-\oplus id_{n-i-1}}
}
\]
Here, two such natural transformations $\eta$ and $\eta'$ are said
to be equivalent if there exists a natural equivalence $\mu:GL_{-}\longrightarrow GL_{-}$
such that $\mu\circ\eta=\eta'$ or if $\eta'=\eta^{*}$ where $-^{\text{*}}$
denotes the dual representation. Then, $\eta$ is equivalent to one
of the following natural transformations.

\begin{enumerate}
\item The trivial natural transformation, denoted by $\mathfrak{id}$: for
every generator $\sigma_{i}$ of $\mathbf{B}_{n}$, $\mathfrak{id}_{n}\left(\sigma_{i}\right)=Id_{GL_{n}\left(\mathbb{C}\left[t^{\pm1}\right]\right)}.$
\item The unreduced Burau natural transformation, denoted by $\mathfrak{bur}$:
for all generators $\sigma_{i}$ of $\mathbf{B}_{n}$,
\[
\mathfrak{bur}_{n,t}\left(\sigma_{i}\right)=Id_{i-1}\oplus B\left(t\right)\oplus Id_{n-i-1},
\]
with $B\left(t\right)=\left[\begin{array}{cc}
0 & t\\
1 & 1-t
\end{array}\right]$.
\item The natural transformation denoted by $\mathfrak{tym}$: for every
generator $\sigma_{i}$ of $\mathbf{B}_{n}$ if $n\geq2$,
\[
\mathfrak{tym}_{n,t}\left(\sigma_{i}\right)=Id_{i-1}\oplus TYM\left(t\right)\oplus Id_{n-i-1},
\]
with $TYM\left(t\right)=\left[\begin{array}{cc}
0 & t\\
1 & 0
\end{array}\right]$. We call it the Tong-Yang-Ma representation.
\end{enumerate}
\end{thm}

The unreduced Burau representation (see \cite[Section 3.1]{kasselturaev}
or \cite[Section 4.2]{BirmanBrendlesurvey} for more details about
this family of representations) is reducible but indecomposable, whereas
the Tong-Yang-Ma representation is irreducible (see \cite[Part II]{TYM}).
We can also consider a natural transformation using the family of
reduced Burau representations (see \cite[Section 3.3]{kasselturaev}
for more details about the associated family of representations):
these are irreducible subrepresentations of the unreduced Burau representations.
\begin{defn}
Let $\ensuremath{GL_{-\,\text{-}1}}:\left(\mathbb{N},\leq\right)\rightarrow\mathfrak{Gr}$
be the functor defined by:

\begin{itemize}
\item Objects: for all natural numbers $n$, $\ensuremath{GL_{-\,\text{-}1}}\left(n\right)=GL_{n-1}\left(\mathbb{C}\left[t^{\pm1}\right]\right)$
the general linear group of degree $n-1$.
\item Morphisms: let $n$ be a natural number. We define $\ensuremath{GL_{-\,\text{-}1}}\left(\gamma{}_{n}\right):GL_{n-1}\left(\mathbb{C}\left[t^{\pm1}\right]\right)\hookrightarrow GL_{n}\left(\mathbb{C}\left[t^{\pm1}\right]\right)$
assigning $GL_{-}\left(\gamma{}_{n}\right)\left(\varphi\right)=id_{1}\oplus\varphi$
for all $\varphi\in GL_{n-1}\left(\mathbb{C}\left[t^{\pm1}\right]\right)$.
\end{itemize}
\end{defn}

\begin{defn}
\label{def:defburred}The reduced Burau natural transformation, denoted
by $\mathfrak{\overline{bur}}:\mathbf{B}_{-}\rightarrow\ensuremath{GL_{-\,\text{-}1}}$
is defined by:

\begin{itemize}
\item For $n=2$, one assigns $\mathfrak{\overline{\mathfrak{bur}}}\left(\sigma_{1}\right)$
to be the automorphism of $\mathbb{C}\left[t^{\pm1}\right]$ defined
by the multiplication by $-t$.
\item For all natural numbers $n\geq3$, we define for every Artin generator
$\sigma_{i}$ of $\mathbf{B}_{n}$ with $i\in\left\{ 2,\ldots,n-2\right\} $:
\[
\mathfrak{\overline{\mathfrak{bur}}}_{n,t}\left(\sigma_{i}\right)=Id_{i-2}\oplus\overline{B}\left(t\right)\oplus Id_{n-i-2}
\]
with
\[
\overline{B}\left(t\right)=\left[\begin{array}{ccc}
1 & 0 & 0\\
1 & -t & t\\
0 & 0 & 1
\end{array}\right]
\]
and
\[
\mathfrak{\overline{\mathfrak{bur}}}_{n,t}\left(\sigma_{1}\right)=\left[\begin{array}{cc}
-t & t\\
0 & 1
\end{array}\right]\oplus Id_{n-3}\,\,\,\,\,\,\,\,\,\,;\,\,\,\,\,\,\,\,\mathfrak{\overline{\mathfrak{bur}}}_{n,t}\left(\sigma_{n-1}\right)=Id_{n-3}\oplus\left[\begin{array}{cc}
1 & 0\\
1 & -t
\end{array}\right].
\]
\end{itemize}
\end{defn}

Let us use these natural transformations to form functors over the
category $\mathfrak{U}\boldsymbol{\beta}$. Indeed, a natural transformation
$\eta:\mathbf{B}_{-}\rightarrow GL_{-}$ (or $\ensuremath{GL_{-\,\text{-}1}}$)
provides in particular:

\begin{itemize}
\item a functor $\mathfrak{N}:\boldsymbol{\beta}\longrightarrow\mathbb{C}\left[t^{\pm1}\right]\textrm{-}\mathfrak{Mod}$;
\item morphisms $\mathfrak{N}\left(\left[n'-n,id_{n'}\right]\right):\mathfrak{N}\left(n\right)\rightarrow\mathfrak{N}\left(n'\right)$
for all natural numbers $n'\geq n$, such that the relation (\ref{eq:criterion2})
of Proposition \ref{prop:criterionfamilymorphismsfunctor} is satisfied.
\end{itemize}
Therefore, according to Proposition \ref{prop:criterionfamilymorphismsfunctor},
it suffices to show that the relation (\ref{eq:criterion}) is satisfied
to prove that $\mathfrak{N}$ is an object of $\mathbf{Fct}\left(\mathfrak{U}\boldsymbol{\beta},\mathbb{C}\left[t^{\pm1}\right]\textrm{-}\mathfrak{Mod}\right)$.
\begin{notation}
\label{nota:defembeedingtymbu}Recall that $0$ is a null object in
the category of $R$-modules, and that the notation $\iota_{G}:0\rightarrow G$
was introduced in Notation \ref{nota:notationsgenerales}. Let $n\in\mathbb{N}$.
For all natural numbers $n$ and $n'$ such that $n'\geq n$, we define
$\iota_{\mathbb{C}\left[t^{\pm1}\right]^{\oplus n'-n}}\oplus id_{\mathbb{C}\left[t^{\pm1}\right]^{\oplus n}}:\mathbb{C}\left[t^{\pm1}\right]^{\oplus n}\hookrightarrow\mathbb{C}\left[t^{\pm1}\right]^{\oplus n'}$
the embedding of $\mathbb{C}\left[t^{\pm1}\right]^{\oplus n}$ as
the submodule of $\mathbb{C}\left[t^{\pm1}\right]^{\oplus n'}$ given
by the $n$ last copies of $\mathbb{C}\left[t^{\pm1}\right]$.
\end{notation}

\paragraph{Tong-Yang-Ma functor:\label{exa:deftym}}

This example is based on the family introduced by Tong, Yang and Ma
(see Theorem \ref{thm:resulttym}). Let $\mathfrak{TYM}_{t}:\boldsymbol{\beta}\rightarrow\mathbb{C}\left[t^{\pm1}\right]\textrm{-}\mathfrak{Mod}$
be the functor defined on objects by $\mathfrak{TYM}_{t}\left(n\right)=\mathbb{C}\left[t^{\pm1}\right]^{\oplus n}$
for all natural numbers $n$, and for all numbers $n\geq2$, for every
Artin generator $\sigma_{i}$ of $\mathbf{B}_{n}$, by $\mathfrak{TYM}_{t}\left(\sigma_{i}\right)=\mathfrak{tym}_{n,t}\left(\sigma_{i}\right)$
for morphisms. For all natural numbers $n$ and $n'$ such that $n'\geq n$,
we assign $\mathfrak{TYM}_{t}\left(\left[n'-n,id_{n'}\right]\right):\mathbb{C}\left[t^{\pm1}\right]^{\oplus n}\hookrightarrow\mathbb{C}\left[t^{\pm1}\right]^{\oplus n'}$
to be the embedding $\iota_{\mathbb{C}\left[t^{\pm1}\right]^{\oplus n'-n}}\oplus id_{\mathbb{C}\left[t^{\pm1}\right]^{\oplus n}}$
(where these morphisms are introduced in Notation \ref{nota:defembeedingtymbu}).
For all natural numbers $n''\geq n'\geq n$, for all Artin generators
$\sigma_{i}\in\mathbf{B}_{n}$ and all $\psi_{j}\in\mathbf{B}_{n'-n}$,
our assignments give:
\begin{eqnarray*}
\mathfrak{TYM}_{t}\left(\psi\natural\sigma\right)\circ\mathfrak{TYM}_{t}\left(\left[n'-n,id_{n'}\right]\right) & = & \left(Id_{j-1}\oplus TYM\left(t\right)\oplus Id_{\left(n'-n\right)-j-1}\oplus Id_{n'-n+i-1}\oplus TYM\left(t\right)\oplus Id_{n'-i-1}\right)\\
 &  & \circ\left(\iota_{\mathbb{C}\left[t^{\pm1}\right]^{\oplus n'-n}}\oplus id_{\mathbb{C}\left[t^{\pm1}\right]^{\oplus n}}\right).
\end{eqnarray*}
Remark that $\left(Id_{j-1}\oplus TYM\left(t\right)\oplus Id_{\left(n'-n\right)-j-1}\right)\circ\iota_{\mathbb{C}\left[t^{\pm1}\right]^{\oplus\left(n'-n\right)}}=\iota_{\mathbb{C}\left[t^{\pm1}\right]^{\oplus\left(n'-n\right)}}$.
Hence we deduce that
\[
\mathfrak{TYM}_{t}\left(\psi\natural\sigma\right)\circ\mathfrak{TYM}_{t}\left(\left[n'-n,id_{n'}\right]\right)=\mathfrak{TYM}_{t}\left(\left[n'-n,id_{n'}\right]\right)\circ\mathfrak{TYM}_{t}\left(\sigma\right)
\]
for all $\sigma\in\mathbf{B}_{n}$ and all $\psi\in\mathbf{B}_{n'-n}$.
According to Proposition \ref{prop:criterionfamilymorphismsfunctor},
our assignment defines a functor $\mathfrak{TYM}_{t}:\mathfrak{U}\boldsymbol{\beta}\rightarrow\mathbb{C}\left[t^{\pm1}\right]\textrm{-}\mathfrak{Mod}$,
called the Tong-Yang-Ma functor.

\paragraph{Burau functors:\label{exa:defunredbur}}

Other examples naturally arise from the Burau representations.

Let $\mathfrak{Bur}_{t}:\boldsymbol{\beta}\longrightarrow\mathbb{C}\left[t^{\pm1}\right]\textrm{-}\mathfrak{Mod}$
be the functor defined on objects by $\mathfrak{Bur}_{t}\left(n\right)=\mathbb{C}\left[t^{\pm1}\right]^{\oplus n}$
for all natural numbers $n$, and for all numbers $n\geq2$, for every
Artin generator $\sigma_{i}$ of $\mathbf{B}_{n}$, by $\mathfrak{Bur}_{t}\left(\sigma_{i}\right)=\mathfrak{bur}_{n,t}\left(\sigma_{i}\right)$
for morphisms. For all natural numbers $n$ and $n'$ such that $n'\geq n$,
we assign $\mathfrak{Bur}_{t}\left(\left[n'-n,id_{n'}\right]\right):\mathbb{C}\left[t^{\pm1}\right]^{\oplus n}\hookrightarrow\mathbb{C}\left[t^{\pm1}\right]^{\oplus n'}$
to be the embedding $\iota_{\mathbb{C}\left[t^{\pm1}\right]^{\oplus n'-n}}\oplus id_{\mathbb{C}\left[t^{\pm1}\right]^{\oplus n}}$
(where these morphisms are introduced in Notation \ref{nota:defembeedingtymbu}). 

As for the functor $\mathfrak{TYM}$, the assignment for $\mathfrak{Bur}$
implies that for all natural numbers $n''\geq n'\geq n$, for all
$\sigma\in\mathbf{B}_{n}$ and all $\psi\in\mathbf{B}_{n'-n}$, $\mathfrak{Bur}_{t}\left(\left[n'-n,id_{n'}\right]\right)\circ\mathfrak{Bur}_{t}\left(\sigma\right)=\mathfrak{Bur}_{t}\left(\psi\natural\sigma\right)\circ\mathfrak{Bur}_{t}\left(\left[n'-n,id_{n'}\right]\right)$.
According to Proposition \ref{prop:criterionfamilymorphismsfunctor},
our assignment defines a functor $\mathfrak{Bur}_{t}:\mathfrak{U}\boldsymbol{\beta}\longrightarrow\mathbb{C}\left[t^{\pm1}\right]\textrm{-}\mathfrak{Mod}$,
called the unreduced Burau functor. The dual version of the functor
$\mathfrak{Bur}_{t}$ was already considered by Randal-Williams and
Wahl in \cite[Example 4.3]{WahlRandal-Williams}.

Analogously, we can form a functor from the reduced Burau representations.
Let $\overline{\mathfrak{Bur}}_{t}:\boldsymbol{\beta}\longrightarrow\mathbb{C}\left[t^{\pm1}\right]\textrm{-}\mathfrak{Mod}$
be the functor defined on objects by $\overline{\mathfrak{Bur}}_{t}\left(0\right)=0$
and $\overline{\mathfrak{Bur}}_{t}\left(n\right)=\mathbb{C}\left[t^{\pm1}\right]^{\oplus n-1}$
for all nonzero natural numbers $n$, and by $\mathfrak{\overline{\mathfrak{Bur}}}_{t}\left(\sigma_{i}\right)=\mathfrak{\overline{\mathfrak{bur}}}_{n,t}\left(\sigma_{i}\right)$
for morphisms for every Artin generator $\sigma_{i}$ of $\mathbf{B}_{n}$
for all numbers $n\geq2$.

For all natural numbers $n$ and $n'$ such that $n'\geq n$, we assign
$\overline{\mathfrak{Bur}}_{t}\left(\left[n'-n,id_{n'}\right]\right):\mathbb{C}\left[t^{\pm1}\right]^{\oplus n-1}\hookrightarrow\mathbb{C}\left[t^{\pm1}\right]^{\oplus n'-1}$
to be the embedding $\iota_{\mathbb{C}\left[t^{\pm1}\right]^{\oplus n'-n}}\oplus id_{\mathbb{C}\left[t^{\pm1}\right]^{\oplus n-1}}$
(where these morphisms are introduced in Notation \ref{nota:defembeedingtymbu}).
Repeating mutadis mutandis the work done for the functor $\mathfrak{TYM}$,
the assignment for $\overline{\mathfrak{Bur}}_{t}$ implies that for
all natural numbers $n''\geq n'\geq n$, for all $\sigma\in\mathbf{B}_{n}$
and all $\psi\in\mathbf{B}_{n'-n}$, $\overline{\mathfrak{Bur}}_{t}\left(\left[n'-n,id_{n'}\right]\right)\circ\overline{\mathfrak{Bur}}_{t}\left(\sigma\right)=\overline{\mathfrak{Bur}}_{t}\left(\psi\natural\sigma\right)\circ\overline{\mathfrak{Bur}}_{t}\left(\left[n'-n,id_{n'}\right]\right)$.
According to Proposition \ref{prop:criterionfamilymorphismsfunctor},
our assignment defines a functor $\overline{\mathfrak{Bur}}_{t}:\mathfrak{U}\boldsymbol{\beta}\longrightarrow\mathbb{C}\left[t^{\pm1}\right]\textrm{-}\mathfrak{Mod}$,
called the reduced Burau functor.

\paragraph{Lawrence-Krammer functor:\label{exa:lkfunctor}}

The family of Lawrence-Krammer representations was notably used to
prove that braid groups are linear (see \cite{Bigelow2,kohno2012homological,KrammerLK}).
For this paragraph, we assign $\mathbb{K}=\mathbb{C}\left[t^{\pm1}\right]\left[q^{\pm1}\right]$
the ring of Laurent polynomials in two variables and consider the
functor $GL_{-}$ of Definition \ref{def:functorBandGL} with this
assignment. Let $\mathfrak{LK}:\mathfrak{U}\boldsymbol{\beta}\rightarrow\mathbb{C}\left[t^{\pm1}\right]\left[q^{\pm1}\right]\textrm{-}\mathfrak{Mod}$
be the assignment:

\begin{itemize}
\item Objects: for all natural numbers $n\geq2$, $\mathfrak{LK}\left(n\right)=\underset{1\leq j<k\leq n}{\bigoplus}V_{j,k}$,
with for all $1\leq j<k\leq n$, $V_{j,k}$ is a free $\mathbb{C}\left[t^{\pm1}\right]\left[q^{\pm1}\right]$-module
of rank one. Hence, $\mathfrak{LK}\left(n\right)\cong\left(\mathbb{C}\left[t^{\pm1}\right]\left[q^{\pm1}\right]\right)^{\oplus n\left(n-1\right)/2}$
as $\mathbb{C}\left[t^{\pm1}\right]\left[q^{\pm1}\right]$-modules.
Moreover, one assigns $\mathfrak{LK}\left(1\right)=0$ and $\mathfrak{LK}\left(0\right)=0$.
\item Morphisms:

\begin{itemize}
\item Automorphisms: for all natural numbers $n,$ for every Artin generator
$\sigma_{i}$ of $\mathbf{B}_{n}$ (with $i\in\left\{ 1,\ldots,n-1\right\} $),
for all $v_{j,k}\in V_{j,k}$ (with $1\leq j<k\leq n$),
\[
\mathfrak{LK}\left(\sigma_{i}\right)v_{j,k}=\begin{cases}
v_{j,k} & \textrm{if \ensuremath{i\notin\left\{ j-1,j,k-1,k\right\} },}\\
tv_{i,k}+\left(t^{2}-t\right)v_{i,i+1}+\left(1-t\right)v_{i+1,k} & \textrm{if \ensuremath{i=j-1}},\\
v_{i+1,k} & \textrm{if \ensuremath{i=j\neq k-1}},\\
tv_{j,i}+\left(1-t\right)v_{j,i+1}-\left(t^{2}-t\right)qv_{i,i+1} & \textrm{if \ensuremath{i=k-1\neq j}},\\
v_{j,i+1} & \textrm{if \ensuremath{i=k}},\\
-qt^{2}v_{i,i+1} & \textrm{if \ensuremath{i=j=k-1}}.
\end{cases}
\]
\item General morphisms: let $n,n'\in\mathbb{N}$, such that $n'\geq n$.
For all natural numbers $j$ and $k$ such that $1\leq j<k\leq n$,
we define the embedding $\mathfrak{V}_{j,k}^{n,n'}:V_{j,k}\overset{\sim}{\longrightarrow}V_{j+\left(n'-n\right),k+\left(n'-n\right)}\hookrightarrow\underset{1\leq j<k\leq n'}{\bigoplus}V_{j,k}$
of free $\mathbb{C}\left[t^{\pm1}\right]\left[q^{\pm1}\right]$-modules.
Then we define $\mathfrak{LK}\left(\left[n'-n,id_{n'}\right]\right):\underset{1\leq j<k\leq n}{\bigoplus}V_{j,k}\rightarrow\underset{1\leq j<k\leq n'}{\bigoplus}V_{j,k}$
to be the embedding $\underset{1\leq j<k\leq n}{\bigoplus}\mathfrak{V}_{j,k}^{n,n'}$.
\end{itemize}
\end{itemize}
Since we consider a family of representations of $\mathbf{B}_{n}$
(see \cite{KrammerLK}), the assignment $\mathfrak{LK}$ defines an
object of $\mathbf{Fct}\left(\boldsymbol{\beta},\mathbb{C}\left[t^{\pm1}\right]\textrm{-}\mathfrak{Mod}\right)$.

Let $n$, $n'$ and $n''$ be natural numbers such that $n''\geq n'\geq n$.
It follows directly from our definitions of $\mathfrak{LK}\left(\left[n'-n,id_{n'}\right]\right)$,
$\mathfrak{LK}\left(\left[n''-n',id_{n''}\right]\right)$ and $\mathfrak{LK}\left(\left[n''-n,id_{n''}\right]\right)$
that relation (\ref{eq:criterion2}) of Proposition \ref{prop:criterionfamilymorphismsfunctor}
is satisfied.

According to the definition of $\mathfrak{LK}\left(\sigma_{l}\right)$
with $\sigma_{l}$ an Artin generator of $\mathbf{B}_{n'-n}$, for
all $v_{j,k}\in V_{j,k}$ with $1+\left(n'-n\right)\leq j<k\leq n'$,
$\mathfrak{LK}\left(\sigma_{l}\right)v_{j,k}=v_{j,k}$. Hence for
all $\psi\in\mathbf{B}_{n'-n}$:
\[
\mathfrak{LK}\left(\psi\natural id_{n}\right)\circ\mathfrak{LK}\left(\left[n'-n,id_{n'}\right]\right)=\mathfrak{LK}\left(\left[n'-n,id_{n'}\right]\right).
\]
Note also that for all $l\in\left\{ 1,\ldots,n-1\right\} $, for all
$v_{j,k}\in V_{j,k}$ with $1+\left(n'-n\right)\leq j<k\leq n'$,
it follows from the assignment of $\mathfrak{LK}$ that:
\[
\mathfrak{LK}\left(id_{n'-n}\natural\sigma_{l}\right)\left(v_{\left(n'-n\right)+j,\left(n'-n\right)+k}\right)=\mathfrak{LK}\left(\sigma_{n'-n+l}\right)\left(v_{\left(n'-n\right)+j,\left(n'-n\right)+k}\right)=\mathfrak{LK}\left(\left[n'-n,id_{n'}\right]\right)\left(\mathfrak{LK}\left(\sigma_{l}\right)\left(v_{j,k}\right)\right).
\]
Therefore, this implies that for all $\sigma\in\mathbf{B}_{n}$, $\mathfrak{LK}\left(\left[n'-n,id_{n'}\right]\right)\circ\mathfrak{LK}\left(\sigma\right)=\mathfrak{LK}\left(id_{n'-n}\natural\sigma\right)\circ\mathfrak{LK}\left(\left[n'-n,id_{n'}\right]\right)$.
Hence, $\mathfrak{\mathfrak{LK}}$ satisfies the relation (\ref{eq:criterion})
of Proposition \ref{prop:criterionfamilymorphismsfunctor}. Hence,
the assignment defines a functor $\mathfrak{LK}:\mathfrak{U}\boldsymbol{\beta}\rightarrow\mathbb{C}\left[t^{\pm1}\right]\left[q^{\pm1}\right]\textrm{-}\mathfrak{Mod}$,
called the Lawrence-Krammer functor.

\section{Functoriality of the Long-Moody construction\label{sec:Functoriality-of-the}}

The principle of the Long-Moody construction, corresponding to Theorem
$2.1$ of \cite{Long1}, is to build a linear representation of the
braid group $\mathbf{B}_{n}$ starting from a representation $\mathbf{B}_{n+1}$.
We develop a functorial version of this construction, which leads
to the notion of Long-Moody functors (see Section \ref{subsec:The-Long-Moody-functors}).
Beforehand, we need to introduce various tools, which are consequences
of the relationships between braid groups and free groups (see Section
\ref{subsec:Braid-groups-and}). Finally, in Section \ref{subsec:Evaluation-of-the},
we investigate examples of functors which are recovered by Long-Moody
functors.

\subsection{Braid groups and free groups\label{subsec:Braid-groups-and}}

This section recalls some relationships between braid groups and free
groups. We also develop tools which will be used throughout our work
of Sections \ref{subsec:The-Long-Moody-functors} and \ref{sec:The-Long-Moody-functoreffect}.

We consider the free group on $n$ generators, which we denote by
$\mathbf{F}_{n}=\left\langle g_{1},\ldots,g_{n}\right\rangle $.
\begin{notation}
We denote by $e_{\mathbf{F}_{n}}$ the unit element of the free group
on $n$ generators $\mathbf{F}_{n}$, for all natural numbers $n$.
\end{notation}

Recall that the category of finitely generated free groups is monoidal
using free product of groups (see Notation \ref{nota:grandN}). The
object $0$ being null in the category $\mathfrak{gr}$, recall that
$\iota_{\mathbf{F}_{n}}:0\rightarrow\mathbf{F}_{n}$ denotes the unique
morphism from $0$ to $\mathbf{F}_{n}$ as in Notation \ref{nota:notationsgenerales}.
\begin{defn}
\label{def:functorF}Let $n$ be a natural number. We consider $\iota_{\mathbf{F}_{1}}*id_{\mathbf{F}_{n}}:\mathbf{F}_{n}\hookrightarrow\mathbf{F}_{n+1}$.
This corresponds to the identification of $\mathbf{F}_{n}$ as the
subgroup of $\mathbf{F}_{n+1}$ generated by the $n$ last copies
of $\mathbf{F}_{1}$ in $\mathbf{F}_{n+1}$. Iterating this morphism,
we obtain for all natural numbers $n'\geq n$ the morphism $\iota_{\mathbf{F}_{n'-n}}*id_{\mathbf{F}_{n}}:\mathbf{F}_{n}\hookrightarrow\mathbf{F}_{n'}$.
\end{defn}

Let $\left\{ \varsigma_{n}:\mathbf{F}_{n}\rightarrow\mathbf{B}_{n+1}\right\} _{n\in\mathbb{N}}$
be a family of group morphisms from the free group $\mathbf{F}_{n}$
to the braid group $\mathbf{B}_{n+1}$, for all natural numbers $n$.
We require these morphisms to satisfy the following crucial property.
\begin{condition}
\label{cond:conditionstability}For all elements $g\in\mathbf{F}_{n}$,
for all natural numbers $n'\geq n$, the following diagram is commutative
in the category $\mathfrak{U}\boldsymbol{\beta}$:

\[
\xymatrix{1\natural n\ar@{->}[rrr]^{\varsigma_{n}\left(g\right)}\ar@{->}[d]_{id_{1}\natural\left[n'-n,id_{n'}\right]} &  &  & 1\natural n\ar@{->}[d]^{id_{1}\natural\left[n'-n,id_{n'}\right]}\\
1\natural n'\ar@{->}[rrr]_{\varsigma_{n'}\left(e_{\mathbf{F}_{n'-n}}*g\right)} &  &  & 1\natural n'.
}
\]
\end{condition}

\begin{rem}
Condition \ref{cond:conditionstability} will be used to prove that
the Long-Moody functor is well defined on morphisms with respect to
the tensor product structure in Theorem \ref{Thm:LMFunctor}. Moreover,
it will also be used in the proof of Propositions \ref{prop:defvarepsi}
and \ref{prop:subfunclmtm}.
\end{rem}

\begin{lem}
\label{rem:noteworthyconseq} Condition \ref{cond:conditionstability}
is equivalent to assume that for all natural numbers $n$, for all
elements $g\in\mathbf{F}_{n}$, the morphisms $\left\{ \varsigma_{n}\right\} _{n\in\mathbb{N}}$
satisfy the following equality in $\mathbf{B}_{n+2}$:
\begin{align}
\left(\left(b_{1,1}^{\boldsymbol{\beta}}\right)^{-1}\natural id_{n}\right)\circ\left(id_{1}\natural\varsigma_{n}\left(g\right)\right)=\varsigma_{n+1}\left(e_{\mathbf{F}_{1}}*g\right)\circ\left(\left(b_{1,1}^{\boldsymbol{\beta}}\right)^{-1}\natural id_{n}\right).\label{eq:equiva}
\end{align}
\end{lem}

\begin{proof}
Let $n$ and $n'$ be natural numbers such that $n'\geq n$. The equality
(\ref{eq:equiva}) implies that for all $1\leq k\leq n'-n$, the following
diagram in the category $\boldsymbol{\beta}$ is commutative :
\[
\xymatrix{1\natural n'\ar@{->}[rrrr]^{id_{n'-\left(n+k\right)}\natural\varsigma_{n+k-1}\left(e_{\mathbf{F}_{k-1}}*g\right)}\ar@{->}[d]_{id_{n'-(n+k)}\natural\left(b_{1,1}^{\boldsymbol{\beta}}\right)^{-1}\natural id_{\left(k-1\right)+n}} &  &  &  & 1\natural n'\ar@{->}[d]^{id_{n'-(n+k)}\natural\left(b_{1,1}^{\boldsymbol{\beta}}\right)^{-1}\natural id_{\left(k-1\right)+n}}\\
1\natural n'\ar@{->}[rrrr]_{id_{n'-\left(n+k\right)}\natural\varsigma_{n+k}\left(e_{\mathbf{F}_{k}}*g\right)} &  &  &  & 1\natural n'.
}
\]
Hence composing squares, we obtain that the following diagram is commutative
in the category $\boldsymbol{\beta}$:

\[
\xymatrix{1\natural\cdots\natural\left(1\natural1\right)\natural n\ar@{->}[rrr]^{id_{n'-n-1}\natural\left(b_{1,1}^{\boldsymbol{\beta}}\right)^{-1}\natural id_{n}}\ar@{->}[d]_{id_{n'}\natural\varsigma_{n}\left(g\right)} &  &  & 1\natural\cdots\natural1\natural\left(1\natural n\right)\ar@{->}[d]_{id_{n'-1}\natural\varsigma_{n+1}\left(e_{\mathbf{F}_{1}}*g\right)}\ar@{->}[rrr]^{\,\,\,\,\,\,\,\,\,\,\,\,\,\,\,id_{n'-n-2}\natural\left(b_{1,1}^{\boldsymbol{\beta}}\right)^{-1}\natural id_{1+n}} &  &  & \cdots\ar@{->}[rr]^{\left(b_{1,1}^{\boldsymbol{\beta}}\right)^{-1}\natural id_{n'-1}} &  & 1\natural n'\ar@{->}[d]^{\varsigma_{n'}\left(e_{\mathbf{F}_{1}}*g\right)}\\
1\natural\cdots\natural1\natural n\ar@{->}[rrr]_{id_{n'-n-1}\natural\left(b_{1,1}^{\boldsymbol{\beta}}\right)^{-1}\natural id_{n}} &  &  & 1\natural\cdots\natural1\natural\left(1\natural n\right)\ar@{->}[rrr]_{\,\,\,\,\,\,\,\,\,\,\,\,\,\,\,id_{n'-n-2}\natural\left(b_{1,1}^{\boldsymbol{\beta}}\right)^{-1}\natural id_{1+n}} &  &  & \cdots\ar@{->}[rr]_{\left(b_{1,1}^{\boldsymbol{\beta}}\right)^{-1}\natural id_{n'-1}} &  & 1\natural n'.
}
\]
By definition of the braiding (see Definition \ref{def:defbraidgroupoid}),
we deduce that the composition of horizontal arrows is the morphism
$\left(b_{1,n'-n}^{\boldsymbol{\beta}}\right)^{-1}\natural id_{n}$
in $\boldsymbol{\beta}$. Recall from Proposition \ref{prop:homogenousprebraided}
that $id_{1}\natural\left[n'-n,\sigma\right]=\left[n'-n,\left(id_{1}\natural\sigma\right)\circ\left(\left(b_{1,n'-n}^{\boldsymbol{\beta}}\right)^{-1}\natural id_{n}\right)\right]$.
Hence Condition \ref{cond:conditionstability} is satisfied if we
assume that the equality (\ref{eq:equiva}) is satisfied for all natural
numbers $n$.

Conversely, assume that Condition \ref{cond:conditionstability} is
satisfied. Condition \ref{cond:conditionstability} with $n'=n+1$
ensures that:
\[
\left[1,\left(\left(b_{1,1}^{\boldsymbol{\beta}}\right)^{-1}\natural id_{n}\right)\circ\left(id_{1}\natural\varsigma_{n}\left(g\right)\right)\right]=\left[1,\varsigma_{n'}\left(e_{\mathbf{F}_{1}}*g\right)\circ\left(\left(b_{1,1}^{\boldsymbol{\beta}}\right)^{-1}\natural id_{n}\right)\right].
\]
Since $Aut_{\mathfrak{U}\boldsymbol{\beta}}\left(1\right)=\mathbf{B}_{1}$
is the trivial group, we deduce from the defining equivalence relation
of $\mathfrak{U}\boldsymbol{\beta}$ (see Definition \ref{def:defUB})
the equality in $\mathbf{B}_{n+2}$:
\[
\left(\left(b_{1,1}^{\boldsymbol{\beta}}\right)^{-1}\natural id_{n}\right)\circ\left(id_{1}\natural\varsigma_{n}\left(g\right)\right)=\varsigma_{1+n}\left(e_{\mathbf{F}_{1}}*g\right)\circ\left(\left(b_{1,1}^{\boldsymbol{\beta}}\right)^{-1}\natural id_{n}\right).
\]
\end{proof}
\begin{rem}
It follows from Lemma \ref{rem:noteworthyconseq} that, for $i\geq2$,
$\varsigma_{n}(g_{i})$ is determined by $\varsigma_{k}(g_{1})$ for
$k\leq n$ by the equalities (\ref{eq:equiva}).
\end{rem}

\begin{example}
\label{ex:ex1defmapfnbn}The family $\varsigma_{n,1}$, based on what
is called the pure braid local system in the literature (see \cite[Remark p.223]{Long1}),
is defined by the following inductive assignment for all natural numbers
$n\geq1$.
\begin{eqnarray*}
\varsigma_{n,1}:\mathbf{F}_{n} & \longrightarrow & \mathbf{B}_{n+1}\\
g_{i} & \longmapsto & \begin{cases}
\sigma_{1}^{2} & \textrm{if \ensuremath{i=1}}\\
\sigma_{1}^{-1}\circ\sigma_{2}^{-1}\circ\cdots\circ\sigma_{i-1}^{-1}\circ\sigma_{i}^{2}\circ\sigma_{i-1}\circ\cdots\circ\sigma_{2}\circ\sigma_{1} & \textrm{if \ensuremath{i\in\left\{ 2,\ldots,n\right\} .}}
\end{cases}
\end{eqnarray*}
We assign $\varsigma_{0,1}$ to be the trivial morphism.
\end{example}

\begin{prop}
The family of morphisms $\left\{ \varsigma_{n,1}\right\} _{n\in\mathbb{N}}$
satisfies Condition \ref{cond:conditionstability}.
\end{prop}

\begin{proof}
Relation (\ref{eq:equiva}) is trivially satisfied for $n=0$. Let
$n\geq1$ be a fixed natural number. By definition \ref{exa:defprodmonUbeta},
we have $\left(b_{1,1}^{\boldsymbol{\beta}}\right)^{-1}=\sigma_{1}^{-1}$.
Moreover, for all $i\in\left\{ 2,\ldots,n\right\} $, we have  $\varsigma_{n+1}\left(e_{\mathbf{F}_{1}}*g_{i-1}\right)=\varsigma_{n+1}\left(g_{i}\right)$
and
\[
id_{1}\natural\varsigma_{n,1}\left(g_{i-1}\right)=\sigma_{2}^{-1}\circ\cdots\circ\sigma_{i-1}^{-1}\circ\sigma_{i}^{2}\circ\sigma_{i-1}\circ\cdots\circ\sigma_{2}.
\]
We deduce that:
\[
\left(\left(b_{1,1}^{\boldsymbol{\beta}}\right)^{-1}\natural id_{n}\right)\circ\left(id_{1}\natural\varsigma_{n,1}\left(g_{i-1}\right)\right)\circ\left(b_{1,1}^{\boldsymbol{\beta}}\natural id_{n}\right)=\varsigma_{n,1}\left(g_{i}\right).
\]
Hence Relation (\ref{eq:equiva}) of Lemma \ref{rem:noteworthyconseq}
is satisfied for all natural numbers.
\end{proof}
\begin{example}
\label{exa:trivialsigma}Let us consider the trivial morphisms $\varsigma_{n,*}:\mathbf{F}_{n}\rightarrow0_{\mathfrak{Gr}}\rightarrow\mathbf{B}_{n+1}$
for all natural numbers $n$. The relation of Lemma \ref{rem:noteworthyconseq}
being easily checked, this family of morphisms $\left\{ \varsigma_{n,*}:\mathbf{F}_{n}\rightarrow\mathbf{B}_{n+1}\right\} _{n\in\mathbb{N}}$
satisfies Condition \ref{cond:conditionstability}.
\end{example}

\paragraph{Action of braid groups on automorphism groups of free groups:}

There are several ways to consider the group\textbf{ $\mathbf{B}_{n}$}
as a subgroup of $Aut\left(\mathbf{F}_{n}\right)$. For instance,
the geometric point of view of topology gives us an action of \textbf{$\mathbf{B}_{n}$}
on the free group $\mathbf{F}_{n}$ (see for example \cite{BirmanbraidslinksMCG}
or \cite{kasselturaev}) identifying $\mathbf{B}_{n}$ as the mapping
class group of a $n$-punctured disc $\Sigma_{0,1}^{n}$: fixing a
point $y$ on the boundary of the disc $\Sigma_{0,1}^{n}$, each free
generator $g_{i}$ can be taken as a loop of the disc based $y$ turning
around punctures. Each element $\sigma$ of $\mathbf{B}_{n}$, as
an automorphism up to isotopy of the disc $\Sigma_{0,1}^{n}$, induces
a well-defined action on the fundamental group $\pi_{1}\left(\Sigma_{0,1}^{n}\right)\cong\mathbf{F}_{n}$
called Artin representation (see Example \ref{exan1} for more details).\\

In the sequel, we fix a family of group actions of \textbf{$\mathbf{B}_{n}$}
on the free group $\mathbf{F}_{n}$: let $\left\{ a_{n}:\mathbf{B}_{n}\rightarrow Aut\left(\mathbf{F}_{n}\right)\right\} _{n\in\mathbb{N}}$
be a family of group morphisms from the braid group $\mathbf{B}_{n}$
to the automorphism group $Aut\left(\mathbf{F}_{n}\right)$. For the
work of Sections \ref{subsec:The-Long-Moody-functors} and \ref{sec:The-Long-Moody-functoreffect},
we need the morphisms $a_{n}:\mathbf{B}_{n}\rightarrow Aut\left(\mathbf{F}_{n}\right)$
to satisfy more properties.
\begin{condition}
\label{cond:coherenceconditionbnautfn1}Let $n$ and $n'$ be natural
numbers such that $n'\geq n$. We require $\left(\iota_{\mathbf{F}_{n'-n}}*id_{\mathbf{F}_{n}}\right)\circ\left(a_{n}\left(\sigma\right)\right)=\left(a_{n'}\left(\sigma'\natural\sigma\right)\right)\circ\left(\iota_{\mathbf{F}_{n'-n}}*id_{\mathbf{F}_{n}}\right)$
as morphisms $\mathbf{F}_{n}\rightarrow\mathbf{F}_{n'}$ for all elements
$\sigma$ of $\mathbf{B}_{n}$ and $\sigma'$ of $\mathbf{B}_{n'-n}$,
ie the following diagrams are commutative:

\[
\xymatrix{\mathbf{F}_{n}\ar@{->}[rr]^{a_{n}\left(\sigma\right)}\ar@{->}[d]_{\iota_{\mathbf{F}_{n'-n}}*id_{\mathbf{F}_{n}}} &  & \mathbf{F}_{n}\ar@{->}[d]^{\iota_{\mathbf{F}_{n'-n}}*id_{\mathbf{F}_{n}}} &  & \mathbf{F}_{n}\ar@{->}[rr]^{\iota_{\mathbf{F}_{n'-n}}*id_{\mathbf{F}_{n}}}\ar@{->}[dr]_{\iota_{\mathbf{F}_{n'-n}}*id_{\mathbf{F}_{n}}} &  & \mathbf{F}_{n'}\\
\mathbf{F}_{n'}\ar@{->}[rr]_{a_{n'}\left(id_{n'-n}\natural\sigma\right)} &  & \mathbf{F}_{n'} &  &  & \mathbf{F}_{n'}.\ar@{->}[ur]_{a_{n'}\left(\sigma'\natural id_{n}\right)}
}
\]
\end{condition}

\begin{rem}
Condition \ref{cond:coherenceconditionbnautfn1} will be used to define
the Long-Moody functor on morphisms in Theorem \ref{Thm:LMFunctor}.
Moreover, it will also be used for the proof of Propositions \ref{prop:defvarepsi}
and \ref{prop:subfunclmtm}.
\end{rem}

We will also require the families of morphisms $\left\{ \varsigma_{n}:\mathbf{F}_{n}\rightarrow\mathbf{B}_{n+1}\right\} _{n\in\mathbb{N}}$
and $\left\{ a_{n}:\mathbf{B}_{n}\rightarrow Aut\left(\mathbf{F}_{n}\right)\right\} _{n\in\mathbb{N}}$
to satisfy the following compatibility relations.
\begin{condition}
\label{cond:coherenceconditionsigmanan}Let $n$ be a natural number.
We assume that the morphism given by the coproduct $\varsigma_{n}*\left(id_{1}\natural-\right):\mathbf{F}_{n}*\mathbf{B}_{n}\rightarrow\mathbf{B}_{n+1}$
factors across the canonical surjection to $\mathbf{F}_{n}\underset{a_{n}}{\rtimes}\mathbf{B}_{n}$.
In other words, the following diagram is commutative:
\[
\xymatrix{\mathbf{F}_{n}\ar@{^{(}->}[r]\ar@{->}[dr]_{\varsigma_{n}} & \mathbf{F}_{n}\underset{a_{n}}{\rtimes}\mathbf{B}_{n}\ar@{->}[d] & \mathbf{B}_{n}\ar@{_{(}->}[l]\ar@{->}[dl]^{id_{1}\natural-}\\
 & \mathbf{B}_{n+1}.
}
\]
where the morphism $\mathbf{F}_{n}\underset{a_{n}}{\rtimes}\mathbf{B}_{n}\rightarrow\mathbf{B}_{n+1}$
is induced by the morphism $\mathbf{F}_{n}*\mathbf{B}_{n}\rightarrow\mathbf{B}_{n+1}$
and the group morphism $id_{1}\natural-:\mathbf{B}_{n}\rightarrow\mathbf{B}_{n+1}$
is induced by the monoidal structure. This is equivalent to requiring
that, for all elements $\sigma\in\mathbf{B}_{n}$ and $g\in\mathbf{F}_{n}$,
the following equality holds in $\mathbf{B}_{n+1}$:
\begin{equation}
\left(id_{1}\natural\sigma\right)\circ\varsigma_{n}\left(g\right)=\varsigma_{n}\left(a_{n}\left(\sigma\right)\left(g\right)\right)\circ\left(id_{1}\natural\sigma\right).\label{eq:equalitycoherenceconditionsigmanan}
\end{equation}
\end{condition}

\begin{rem}
Condition \ref{cond:coherenceconditionsigmanan} is essential in the
definition of the Long-Moody functor on objects in Theorem \ref{Thm:LMFunctor}.
\end{rem}

We fix a choice for these families of morphisms $\left\{ \varsigma_{n}:\mathbf{F}_{n}\rightarrow\mathbf{B}_{n+1}\right\} _{n\in\mathbb{N}}$
and $\left\{ a_{n}:\mathbf{B}_{n}\rightarrow Aut\left(\mathbf{F}_{n}\right)\right\} _{n\in\mathbb{N}}$.
\begin{defn}
\label{def:coherentmor}The families $\left\{ \varsigma_{n}:\mathbf{F}_{n}\rightarrow\mathbf{B}_{n+1}\right\} _{n\in\mathbb{N}}$
and $\left\{ a_{n}:\mathbf{B}_{n}\rightarrow Aut\left(\mathbf{F}_{n}\right)\right\} _{n\in\mathbb{N}}$
are said to be coherent if they satisfy conditions \ref{cond:conditionstability},
\ref{cond:coherenceconditionbnautfn1} and \ref{cond:coherenceconditionsigmanan}.
\end{defn}

\begin{example}
\label{exan1}A classical family is provided by the Artin representations
(see for example \cite[Section 1]{BirmanbraidslinksMCG}). For $n\in\mathbb{N}$,
$a_{n,1}:\mathbf{B}_{n}\rightarrow Aut\left(\mathbf{F}_{n}\right)$
is defined for all elementary braids $\sigma_{i}$ where $i\in\left\{ 1,\ldots,n-1\right\} $
by:
\begin{eqnarray*}
a_{n,1}\left(\sigma_{i}\right):\mathbf{F}_{n} & \longrightarrow & \mathbf{F}_{n}\\
g_{j} & \longmapsto & \begin{cases}
g_{i+1} & \textrm{if \ensuremath{j=i}}\\
g_{i+1}^{-1}g_{i}g_{i+1} & \textrm{if \ensuremath{j=i+1}}\\
g_{j} & \textrm{if \ensuremath{j\notin\left\{ i,i+1\right\} }.}
\end{cases}
\end{eqnarray*}
It clearly follows from their definitions that the morphisms $\left\{ a_{n,1}:\mathbf{B}_{n}\rightarrow Aut\left(\mathbf{F}_{n}\right)\right\} _{n\in\mathbb{N}}$
satisfy Condition \ref{cond:coherenceconditionbnautfn1}.
\end{example}

\begin{prop}
The morphisms $\left\{ a_{n,1}:\mathbf{B}_{n}\rightarrow Aut\left(\mathbf{F}_{n}\right)\right\} _{n\in\mathbb{N}}$
together with the morphisms $\left\{ \varsigma_{n,1}:\mathbf{F}_{n}\hookrightarrow\mathbf{B}_{n+1}\right\} _{n\in\mathbb{N}}$
of Example \ref{ex:ex1defmapfnbn} satisfy Condition \ref{cond:coherenceconditionsigmanan}.
\end{prop}

\begin{proof}
Let $i$ be a fixed natural number in $\left\{ 1,\ldots,n-1\right\} $.
We prove that the equality (\ref{eq:equalitycoherenceconditionsigmanan})
of Condition \ref{cond:coherenceconditionsigmanan} is satisfied for
all Artin generator $\sigma_{i}$ and all generator $g_{j}$ of the
free group (with $j\in\left\{ 1,\ldots,n\right\} $). First, it follows
from the braid relation $\sigma_{i}\sigma_{i+1}\sigma_{i}=\sigma_{i+1}\sigma_{i}\sigma_{i+1}$
that:
\[
\sigma_{1+i}^{-1}\circ\sigma_{i}^{-1}\circ\sigma_{1+i}^{-2}\circ\sigma_{i}^{2}\circ\sigma_{1+i}^{2}\circ\sigma_{i}\circ\sigma_{1+i}=\sigma_{i}^{-1}\circ\sigma_{1+i}^{2}\circ\sigma_{i},
\]
and we deduce that:
\[
\sigma_{1+i}^{-1}\circ\varsigma_{n,1}\left(a_{n,1}\left(\sigma_{i}\right)\left(g_{1+i}\right)\right)\circ\sigma_{1+i}=\varsigma_{n,1}\left(g_{1+i}\right).
\]
Also, the braid relation $\sigma_{i+1}\circ\sigma_{i}\circ\sigma_{i+1}=\sigma_{i}\circ\sigma_{i+1}\circ\sigma_{i}$
implies that $\sigma_{i+1}^{-1}\circ\sigma_{i}^{-1}\circ\sigma_{i+1}^{2}\circ\sigma_{i}\circ\sigma_{i+1}=\sigma_{i}^{2}$
and a fortiori:
\[
\sigma_{1+i}^{-1}\circ\varsigma_{n,1}\left(a_{n,1}\left(\sigma_{i}\right)\left(g_{i}\right)\right)\circ\sigma_{1+i}=\varsigma_{n,1}\left(g_{i}\right).
\]
Finally, for a fixed $\ensuremath{j\notin\left\{ i,i+1\right\} }$,
the commutation relation $\sigma_{i}\sigma_{j}=\sigma_{j}\sigma_{i}$
and the braid relation $\sigma_{i}\sigma_{i+1}\sigma_{i}=\sigma_{i+1}\sigma_{i}\sigma_{i+1}$
give directly:
\[
\varsigma_{n,1}\left(g_{j}\right)=\sigma_{1+i}^{-1}\circ\varsigma_{n,1}\left(a_{n,1}\left(\sigma_{i}\right)\left(g_{j}\right)\right)\circ\sigma_{1+i}.
\]
\end{proof}
\begin{cor}
The families of morphisms $\left\{ a_{n,1}:\mathbf{B}_{n}\rightarrow Aut\left(\mathbf{F}_{n}\right)\right\} _{n\in\mathbb{N}}$
and $\left\{ \varsigma_{n,1}:\mathbf{F}_{n}\rightarrow\mathbf{B}_{n+1}\right\} _{n\in\mathbb{N}}$
are coherent.
\end{cor}

\begin{example}
\label{exa:exant}Consider the family of morphisms $\left\{ \varsigma_{n,*}:\mathbf{F}_{n}\rightarrow\mathbf{B}_{n+1}\right\} _{n\in\mathbb{N}}$
of Example \ref{exa:trivialsigma} and any family of morphisms $\left\{ a_{n}:\mathbf{B}_{n}\rightarrow Aut\left(\mathbf{F}_{n}\right)\right\} _{n\in\mathbb{N}}$.
Then Condition \ref{cond:coherenceconditionsigmanan} is always satisfied.
As a consequence, these families of morphisms $\left\{ \varsigma_{n,*}:\mathbf{F}_{n}\rightarrow\mathbf{B}_{n+1}\right\} _{n\in\mathbb{N}}$
and $\left\{ a_{n}:\mathbf{B}_{n}\rightarrow Aut\left(\mathbf{F}_{n}\right)\right\} _{n\in\mathbb{N}}$
are coherent if and only if the family of morphisms $\left\{ a_{n}:\mathbf{B}_{n}\rightarrow Aut\left(\mathbf{F}_{n}\right)\right\} _{n\in\mathbb{N}}$
satisfies Condition \ref{cond:coherenceconditionbnautfn1}.
\end{example}

\subsection{The Long-Moody functors\label{subsec:The-Long-Moody-functors}}

In this section, we prove that the Long-Moody construction of \cite[Theorem 2.1 ]{Long1}
induces a functor \textit{
\[
\mathbf{LM}:\mathbf{Fct}\left(\mathfrak{U}\boldsymbol{\beta},\mathbb{K}\textrm{-}\mathfrak{Mod}\right)\rightarrow\mathbf{Fct}\left(\mathfrak{U}\boldsymbol{\beta},\mathbb{K}\textrm{-}\mathfrak{Mod}\right).
\]
} We fix families of morphisms $\left\{ \varsigma_{n}:\mathbf{F}_{n}\rightarrow\mathbf{B}_{n+1}\right\} _{n\in\mathbb{N}}$
and $\left\{ a_{n}:\mathbf{B}_{n}\rightarrow Aut\left(\mathbf{F}_{n}\right)\right\} _{n\in\mathbb{N}}$,
which are assumed to be coherent (see Definition \ref{def:coherentmor}).

We first need to make some observations and introduce some tools.
Let $F$ be an object of $\mathbf{Fct}\left(\mathfrak{U}\boldsymbol{\beta},\mathbb{K}\textrm{-}\mathfrak{Mod}\right)$
and $n$ be a natural number. A fortiori, the $\mathbb{K}$-module
$F\left(n+1\right)$ is endowed with a left $\mathbb{K}\left[\mathbf{B}_{n+1}\right]$-module
structure. Using the morphism $\varsigma_{n}:\mathbf{F}_{n}\rightarrow\mathbf{B}_{n+1}$,
$F\left(n+1\right)$ is a $\mathbb{K}\left[\mathbf{F}_{n}\right]$-module
by restriction.

Let us consider the augmentation ideal of the free group $\mathbf{F}_{n}$,
denoted by $\mathcal{I}_{\mathbb{K}\left[\mathbf{F}_{n}\right]}$.
Since it is a (right) $\mathbb{K}\left[\mathbf{F}_{n}\right]$-module,
one can form the tensor product $\mathcal{I}_{\mathbb{K}\left[\mathbf{F}_{n}\right]}\underset{\mathbb{K}\left[\mathbf{F}_{n}\right]}{\varotimes}F\left(n+1\right)$.
Also, for all natural numbers $n$ and $n'$ such that $n'\geq n$,
the morphism $\iota_{\mathbf{F}_{n'-n}}*id_{\mathbf{F}_{n}}:\mathbf{F}_{n}\hookrightarrow\mathbf{F}_{n'}$
canonically induces a morphism $\iota_{\mathcal{I}_{\mathbb{K}\left[\mathbf{F}_{n'-n}\right]}}*id_{\mathcal{I}_{\mathbb{K}\left[\mathbf{F}_{n}\right]}}:\mathcal{I}_{\mathbb{K}\left[\mathbf{F}_{n}\right]}\hookrightarrow\mathcal{I}_{\mathbb{K}\left[\mathbf{F}_{n'}\right]}$.
In addition, the augmentation ideal $\mathcal{I}_{\mathbb{K}\left[\mathbf{F}_{n}\right]}$
is a $\mathbb{K}\left[\mathbf{B}_{n}\right]$-module too:
\begin{lem}
The action $a_{n}:\mathbf{B}_{n}\rightarrow Aut\left(\mathbf{F}_{n}\right)$
canonically induces an action of $\mathbf{B}_{n}$ on $\mathcal{I}_{\mathbb{K}\left[\mathbf{F}_{n}\right]}$
denoted by $a_{n}:\mathbf{B}_{n}\rightarrow Aut\left(\mathcal{I}_{\mathbb{K}\left[\mathbf{F}_{n}\right]}\right)$
(abusing the notation).
\end{lem}

\begin{proof}
For any group morphism $H\rightarrow Aut\left(G\right)$, the group
ring $\mathbb{K}\left[G\right]$ is canonically an $H$-module and
so is the augmentation ideal $\mathcal{I}_{G}$, as a submodule of
$\mathbb{K}\left[G\right]$.
\end{proof}
\begin{rem}
If the family of morphisms $\left\{ a_{n}:\mathbf{B}_{n}\rightarrow Aut\left(\mathbf{F}_{n}\right)\right\} _{n\in\mathbb{N}}$
is coherent with respect to the family of morphisms $\left\{ \varsigma_{n}:\mathbf{F}_{n}\rightarrow\mathbf{B}_{n+1}\right\} _{n\in\mathbb{N}}$,
the relation of Condition \ref{cond:coherenceconditionbnautfn1} remains
true mutatis mutandis, for all natural numbers $n$ and $n'$, considering
the induced morphisms $a_{n}:\mathbf{B}_{n}\rightarrow Aut\left(\mathcal{I}_{\mathbb{K}\left[\mathbf{F}_{n}\right]}\right)$
and $\iota_{\mathcal{I}_{\mathbb{K}\left[\mathbf{F}_{n'-n}\right]}}*id_{\mathcal{I}_{\mathbb{K}\left[\mathbf{F}_{n}\right]}}:\mathcal{I}_{\mathbb{K}\left[\mathbf{F}_{n}\right]}\rightarrow\mathcal{I}_{\mathbb{K}\left[\mathbf{F}_{n'}\right]}$.
\end{rem}

In the following theorem, we define an endofunctor of $\mathbf{Fct}\left(\mathfrak{U}\boldsymbol{\beta},\mathbb{K}\textrm{-}\mathfrak{Mod}\right)$
corresponding to the Long-Moody construction. It will be called the
Long-Moody functor with respect to $\left\{ \varsigma_{n}:\mathbf{F}_{n}\rightarrow\mathbf{B}_{n+1}\right\} _{n\in\mathbb{N}}$
and $\left\{ a_{n}:\mathbf{B}_{n}\rightarrow Aut\left(\mathbf{F}_{n}\right)\right\} _{n\in\mathbb{N}}$.
\begin{thm}
\label{Thm:LMFunctor}Recall that we have fixed coherent families
of morphisms $\left\{ \varsigma_{n}:\mathbf{F}_{n}\rightarrow\mathbf{B}_{n+1}\right\} _{n\in\mathbb{N}}$
and $\left\{ a_{n}:\mathbf{B}_{n}\rightarrow Aut\left(\mathbf{F}_{n}\right)\right\} _{n\in\mathbb{N}}$.
The following assignment defines a functor $\mathbf{LM}_{a,\varsigma}:\mathbf{Fct}\left(\mathfrak{U}\boldsymbol{\beta},\mathbb{K}\textrm{-}\mathfrak{Mod}\right)\rightarrow\mathbf{Fct}\left(\mathfrak{U}\boldsymbol{\beta},\mathbb{K}\textrm{-}\mathfrak{Mod}\right)$.

\begin{itemize}
\item Objects: for $F\in Obj\left(\mathbf{Fct}\left(\mathfrak{U}\boldsymbol{\beta},\mathbb{K}\textrm{-}\mathfrak{Mod}\right)\right)$,
$\mathbf{LM}_{a,\varsigma}\left(F\right):\mathfrak{U}\boldsymbol{\beta}\rightarrow\mathbb{K}\textrm{-}\mathfrak{Mod}$
is defined by:

\begin{itemize}
\item Objects: $\forall n\in\mathbb{N}$, $\mathbf{LM}_{a,\varsigma}\left(F\right)\left(n\right)=\mathcal{I}_{\mathbb{K}\left[\mathbf{F}_{n}\right]}\underset{\mathbb{K}\left[\mathbf{F}_{n}\right]}{\varotimes}F\left(n+1\right)$.
\item Morphisms: for $n,n'\in\mathbb{N}$, such that $n'\geq n$, and $\left[n'-n,\sigma\right]\in Hom_{\mathfrak{U}\boldsymbol{\beta}}\left(n,n'\right)$,
assign:
\[
\mathbf{LM}_{a,\varsigma}\left(F\right)\left(\left[n'-n,\sigma\right]\right)\left(i\underset{\mathbb{K}\left[\mathbf{F}_{n}\right]}{\varotimes}v\right)=a_{n'}\left(\sigma\right)\left(\iota_{\mathcal{I}_{\mathbb{K}\left[\mathbf{F}_{n'-n}\right]}}*id_{\mathcal{I}_{\mathbb{K}\left[\mathbf{F}_{n}\right]}}\right)\left(i\right)\underset{\mathbb{K}\left[\mathbf{F}_{n'}\right]}{\varotimes}F\left(id_{1}\natural\left[n'-n,\sigma\right]\right)\left(v\right),
\]
for all $i\in\mathcal{I}_{\mathbb{K}\left[\mathbf{F}_{n}\right]}$
and $v\in F\left(n+1\right)$.
\end{itemize}
\item Morphisms: let $F$ and $G$ be two objects of $\mathbf{Fct}\left(\mathfrak{U}\boldsymbol{\beta},\mathbb{K}\textrm{-}\mathfrak{Mod}\right)$,
and $\eta:F\rightarrow G$ be a natural transformation. We define
$\mathbf{LM}_{a,\varsigma}\left(\eta\right):\mathbf{LM}_{a,\varsigma}\left(F\right)\rightarrow\mathbf{LM}_{a,\varsigma}\left(G\right)$
for all natural numbers $n$ by:
\[
\left(\mathbf{LM}_{a,\varsigma}\left(\eta\right)\right)_{n}=id_{\mathcal{I}_{\mathbb{K}\left[\mathbf{F}_{n}\right]}}\underset{\mathbb{K}\left[\mathbf{F}_{n}\right]}{\varotimes}\eta_{n+1}.
\]
\end{itemize}
In particular, the Long-Moody functor $\mathbf{LM}_{a,\varsigma}$
induces an endofunctor of the category $\mathbf{Fct}\left(\boldsymbol{\beta},\mathbb{K}\textrm{-}\mathfrak{Mod}\right)$.
\end{thm}

\begin{notation}
When there is no ambiguity, once the morphisms $\left\{ \varsigma_{n}:\mathbf{F}_{n}\rightarrow\mathbf{B}_{n+1}\right\} _{n\in\mathbb{N}}$
and $\left\{ a_{n}:\mathbf{B}_{n}\rightarrow Aut\left(\mathbf{F}_{n}\right)\right\} _{n\in\mathbb{N}}$
are fixed, we omit them from the notation $\mathbf{LM}_{a,\varsigma}$
for convenience (especially for proofs).
\end{notation}

\begin{proof}
For this proof, $n$, $n'$ and $n''$ are natural numbers such that
$n''\geq n'\geq n$.
\begin{enumerate}
\item First let us show that the assignment of $\mathbf{LM}$ defines an
endofunctor of $\mathbf{Fct}\left(\boldsymbol{\beta},\mathbb{K}\textrm{-}\mathfrak{Mod}\right)$.
The two first points generalize the proof of \cite[Theorem 2.1]{Long1}.
Let $F$, $G$ and $H$ be objects of $\mathbf{Fct}\left(\boldsymbol{\beta},\mathbb{K}\textrm{-}\mathfrak{Mod}\right)$.

\begin{enumerate}
\item \label{enu:We-first-check}We first check the compatibility of the
assignment $\mathbf{LM}\left(F\right)$ with respect to the tensor
product. Consider $\sigma\in\mathbf{B}_{n}$ $g\in\mathbf{F}_{n}$,
$i\in\mathcal{I}_{\mathbb{K}\left[\mathbf{F}_{n}\right]}$ and $v\in F\left(n+1\right)$.
Since $\left(id_{1}\natural\sigma\right)\circ\varsigma_{n}\left(g\right)=\varsigma_{n}\left(a_{n}\left(\sigma\right)\left(g\right)\right)\circ\left(id_{1}\natural\sigma\right)$
by Condition \ref{cond:coherenceconditionsigmanan}, we deduce that:
\begin{eqnarray*}
\mathbf{LM}\left(F\right)\left(\sigma\right)\left(i\underset{\mathbb{K}\left[\mathbf{F}_{n}\right]}{\varotimes}F\left(\varsigma_{n}\left(g\right)\right)\left(v\right)\right) & = & a_{n}\left(\sigma\right)\left(i\right)\underset{\mathbb{K}\left[\mathbf{F}_{n}\right]}{\varotimes}F\left(id_{1}\natural\sigma\right)\left(F\left(\varsigma_{n}\left(g\right)\right)\left(v\right)\right)\\
 & = & a_{n}\left(\sigma\right)\left(i\right)\underset{\mathbb{K}\left[\mathbf{F}_{n}\right]}{\varotimes}\left(F\left(\varsigma_{n}\left(a_{n}\left(\sigma\right)\left(g\right)\right)\right)\circ F\left(id_{1}\natural\sigma\right)\right)\left(v\right)\\
 & = & a_{n}\left(\sigma\right)\left(i\cdot g\right)\underset{\mathbb{K}\left[\mathbf{F}_{n}\right]}{\varotimes}F\left(id_{1}\natural\sigma\right)\left(v\right)\\
 & = & \mathbf{LM}\left(F\right)\left(\sigma\right)\left(i\cdot g\underset{\mathbb{K}\left[\mathbf{F}_{n}\right]}{\varotimes}\left(v\right)\right).
\end{eqnarray*}
\item \label{enu:Let-us-prove}Let us prove that the assignment $\mathbf{LM}\left(F\right)$
defines an object of $\mathbf{Fct}\left(\boldsymbol{\beta},\mathbb{K}\textrm{-}\mathfrak{Mod}\right)$.
According to our assignment and since $a_{n}$ and $id_{1}\natural-$
are group morphisms, it follows from the definition that $\mathbf{LM}\left(F\right)\left(id_{\mathbf{B}_{n}}\right)=id_{\mathbf{LM}\left(F\right)\left(n\right)}$.
Hence, it remains to prove that the composition axiom is satisfied.
Let $\sigma$ and $\sigma'$ be two elements of $\mathbf{B}_{n}$,
$i\in\mathcal{I}_{\mathbb{K}\left[\mathbf{F}_{n}\right]}$ and $v\in F\left(n+1\right)$.
From the functoriality of $F$ over $\boldsymbol{\beta}$ and the
compatibility of the monoidal structure $\natural$ with composition,
we deduce that $F\left(id_{1}\natural\left(\sigma'\right)\right)\circ F\left(id_{1}\natural\left(\sigma\right)\right)=F\left(id_{1}\natural\left(\sigma'\circ\sigma\right)\right)$.
Since $a_{n}$ is a group morphism, we have:
\[
\left(a_{n}\left(\sigma'\circ\sigma\right)\right)\left(i\right)=a_{n}\left(\sigma'\right)\left(a_{n}\left(\sigma\right)\left(i\right)\right).
\]
 Hence, it follows from the assignment of $\mathbf{LM}$ that:
\begin{eqnarray*}
\mathbf{LM}\left(F\right)\left(\sigma'\circ\sigma\right)\left(i\underset{\mathbb{K}\left[\mathbf{F}_{n}\right]}{\varotimes}v\right) & = & \left(a_{n}\left(\sigma'\circ\sigma\right)\right)\left(i\right)\underset{\mathbb{K}\left[\mathbf{F}_{n}\right]}{\varotimes}F\left(id_{1}\natural\left(\sigma'\circ\sigma\right)\right)\left(v\right)\\
 & = & a_{n}\left(\sigma'\right)\left(a_{n}\left(\sigma\right)\left(i\right)\right)\underset{\mathbb{K}\left[\mathbf{F}_{n}\right]}{\varotimes}\left(F\left(id_{1}\natural\left(\sigma'\right)\right)\circ F\left(id_{1}\natural\left(\sigma\right)\right)\right)\left(v\right)\\
 & = & \mathbf{LM}\left(F\right)\left(\sigma'\right)\circ\mathbf{LM}\left(F\right)\left(\sigma\right)\left(i\underset{\mathbb{K}\left[\mathbf{F}_{n}\right]}{\varotimes}v\right).
\end{eqnarray*}
\item \label{enu:It-remains-to}It remains to check the consistency of our
definition of $\mathbf{LM}$ on morphisms of $\mathbf{Fct}\left(\boldsymbol{\beta},\mathbb{K}\textrm{-}\mathfrak{Mod}\right)$.
Let $\eta:F\rightarrow G$ be a natural transformation. Hence, we
have that:
\[
G\left(id_{1}\natural\sigma\right)\circ\eta_{n+1}=\eta_{n'+1}\circ F\left(id_{1}\natural\sigma\right).
\]
Hence, it follows from the assignment of $\mathbf{LM}$ that:
\[
\mathbf{LM}\left(G\right)\left(\sigma\right)\circ\mathbf{LM}\left(\eta\right)_{n}=\mathbf{LM}\left(\eta\right)_{n'}\circ\mathbf{LM}\left(F\right)\left(\sigma\right)
\]
Therefore $\mathbf{LM}\left(\eta\right)$ is a morphism in the category
$\mathbf{Fct}\left(\boldsymbol{\beta},\mathbb{K}\textrm{-}\mathfrak{Mod}\right)$.
Denoting by $id_{F}:F\rightarrow F$ the identity natural transformation,
it is clear that $\mathbf{LM}\left(id_{F}\right)=id_{\mathbf{LM}\left(F\right)}.$
Finally, let us check the composition axiom. Let $\eta:F\rightarrow G$
and $\mu:G\rightarrow H$ be natural transformations. Let $n$ be
a natural number, $i\in\mathcal{I}_{\mathbb{K}\left[\mathbf{F}_{n}\right]}$
and $v\in F\left(n\right)$. Now, since $\mu$ and $\eta$ are morphisms
in the category $\mathbf{Fct}\left(\boldsymbol{\beta},\mathbb{K}\textrm{-}\mathfrak{Mod}\right)$:
\begin{eqnarray*}
\mathbf{LM}\left(\mu\circ\eta\right)_{n}\left(i\underset{\mathbb{K}\left[\mathbf{F}_{n}\right]}{\varotimes}v\right) & = & i\underset{\mathbb{K}\left[\mathbf{F}_{n}\right]}{\varotimes}\left(\mu_{n+1}\circ\eta_{n+1}\right)\left(v\right)=\mathbf{LM}\left(\mu\right)_{n}\circ\mathbf{LM}\left(\eta\right)_{n}\left(i\underset{\mathbb{K}\left[\mathbf{F}_{n}\right]}{\varotimes}v\right).
\end{eqnarray*}
\end{enumerate}
\item Let us prove that the assignment $\mathbf{LM}$ lifts to define an
endofunctor of $\mathbf{Fct}\left(\mathfrak{U}\boldsymbol{\beta},\mathbb{K}\textrm{-}\mathfrak{Mod}\right)$.
Let $F$, $G$ and $H$ be objects of $\mathbf{Fct}\left(\mathfrak{U}\boldsymbol{\beta},\mathbb{K}\textrm{-}\mathfrak{Mod}\right)$.

\begin{enumerate}
\item First, let us check the compatibility of the assignment $\mathbf{LM}\left(F\right)$
with respect to the tensor product. In fact, this compatibility being
already done for automorphisms (see \ref{enu:We-first-check}), the
remaining point to prove is the compatibility of $\mathbf{LM}\left(F\right)\left(\left[n'-n,id_{n'}\right]\right)$.
Let $g\in\mathbf{F}_{n}$, $i\in\mathcal{I}_{\mathbb{K}\left[\mathbf{F}_{n}\right]}$
and $v\in F\left(n+1\right)$. It follows from Condition \ref{cond:conditionstability}
that in $\mathbf{B}_{n+1}$:
\[
id_{1}\natural\left[n'-n,id_{n'-n}\natural\varsigma_{n}\left(g\right)\right]=\varsigma_{n'}\left(e_{\mathbf{F}_{n'-n}}*g\right)\circ\left(id_{1}\natural\left[n'-n,id_{n'}\right]\right).
\]
Since $\left(\iota_{\mathcal{I}_{\mathbb{K}\left[\mathbf{F}_{n'-n}\right]}}*id_{\mathcal{I}_{\mathbb{K}\left[\mathbf{F}_{n}\right]}}\right)\left(i\cdot g\right)=\left(e_{\mathcal{I}_{\mathbb{K}\left[\mathbf{F}_{n'-n}\right]}}*i\right)\cdot\left(e_{\mathbf{F}_{n'-n}}*g\right)$,
we deduce that:
\begin{eqnarray*}
 &  & \mathbf{LM}\left(F\right)\left(\left[n'-n,id_{n'}\right]\right)\left(i\underset{\mathbb{K}\left[\mathbf{F}_{n}\right]}{\varotimes}F\left(\varsigma_{n}\left(g\right)\right)\left(v\right)\right)\\
 & = & \left(\iota_{\mathcal{I}_{\mathbb{K}\left[\mathbf{F}_{n'-n}\right]}}*id_{\mathcal{I}_{\mathbb{K}\left[\mathbf{F}_{n}\right]}}\right)\left(i\right)\underset{\mathbb{K}\left[\mathbf{F}_{n'}\right]}{\varotimes}F\left(id_{1}\natural\left[n'-n,id_{n'}\right]\right)\left(F\left(\varsigma_{n}\left(g\right)\right)\left(v\right)\right)\\
 & = & \left(\iota_{\mathcal{I}_{\mathbb{K}\left[\mathbf{F}_{n'-n}\right]}}*id_{\mathcal{I}_{\mathbb{K}\left[\mathbf{F}_{n}\right]}}\right)\left(i\cdot g\right)\underset{\mathbb{K}\left[\mathbf{F}_{n'}\right]}{\varotimes}F\left(id_{1}\natural\left[n'-n,id_{n'}\right]\right)\left(v\right)\\
 & = & \mathbf{LM}\left(F\right)\left(\left[n'-n,id_{n'}\right]\right)\left(i\cdot g\underset{\mathbb{K}\left[\mathbf{F}_{n}\right]}{\varotimes}v\right).
\end{eqnarray*}
\item Let us prove that the assignment $\mathbf{LM}\left(F\right)$ defines
an object of $\mathbf{Fct}\left(\mathfrak{U}\boldsymbol{\beta},\mathbb{K}\textrm{-}\mathfrak{Mod}\right)$
using Proposition \ref{prop:criterionfamilymorphismsfunctor}. Recall
the compatibility of the monoidal structure $\natural$ with respect
to composition and that $F$ is an object of $\mathbf{Fct}\left(\mathfrak{U}\boldsymbol{\beta},\mathbb{K}\textrm{-}\mathfrak{Mod}\right)$.
Consider $\left[n'-n,\sigma\right]\in Hom_{\mathfrak{U}\boldsymbol{\beta}}\left(n,n'\right)$.
It follows from our assignment, that:\textit{
\[
\mathbf{LM}\left(F\right)\left(\left[n'-n,\sigma\right]\right)=\mathbf{LM}\left(F\right)\left(\sigma\right)\circ\mathbf{LM}\left(F\right)\left(\left[n'-n,id_{n'}\right]\right).
\]
}Moreover, the composition of morphisms introduced in Definition \ref{def:functorF}
implies that:\textit{
\begin{eqnarray*}
\mathbf{LM}\left(F\right)\left(\left[n''-n,id_{n''}\right]\right) & = & \mathbf{LM}\left(F\right)\left(\left[n''-n',id_{n''}\right]\right)\circ\mathbf{LM}\left(F\right)\left(\left[n'-n,id_{n'}\right]\right).
\end{eqnarray*}
}Hence, the relation (\ref{eq:criterion2}) of Proposition \ref{prop:criterionfamilymorphismsfunctor}
is satisfied. Let $\sigma\in\mathbf{B}_{n}$ and $\psi\in\mathbf{B}_{n'-n}$.
Since $\left(\iota_{n'-n}*id_{n}\right)\circ\left(a_{n}\left(\sigma\right)\right)=\left(a_{n'}\left(\psi\natural\sigma\right)\right)\circ\left(\iota_{n'-n}*id_{n}\right)$
by Condition \ref{cond:coherenceconditionbnautfn1}, we deduce that:
\begin{eqnarray*}
\mathbf{LM}\left(F\right)\left(\psi\natural\sigma\right)\circ\mathbf{LM}\left(F\right)\left(\left[n'-n,id_{n'}\right]\right) & = & \mathbf{LM}\left(F\right)\left(\left[n'-n,id_{n'}\right]\right)\circ\mathbf{LM}\left(F\right)\left(\sigma\right).
\end{eqnarray*}
Hence the relation (\ref{eq:criterion}) of Proposition \ref{prop:criterionfamilymorphismsfunctor}
is also satisfied. Therefore, according to Proposition \ref{prop:criterionfamilymorphismsfunctor},
since $\mathbf{LM}\left(F\right)$ is an object of $\mathbf{Fct}\left(\boldsymbol{\beta},\mathbb{K}\textrm{-}\mathfrak{Mod}\right)$,
the assignment $\mathbf{LM}\left(F\right)$ defines an object of $\mathbf{Fct}\left(\mathfrak{U}\boldsymbol{\beta},\mathbb{K}\textrm{-}\mathfrak{Mod}\right)$.
\item Finally, let us check the consistency of our assignment for $\mathbf{LM}$
on morphisms. Let $\eta:F\rightarrow G$ be a natural transformation.
We already proved in \ref{enu:It-remains-to} that $\mathbf{LM}\left(\eta\right)$
is a morphism in the category $\mathbf{Fct}\left(\boldsymbol{\beta},\mathbb{K}\textrm{-}\mathfrak{Mod}\right)$.
Since $\eta$ is a natural transformation between objects of $\mathbf{Fct}\left(\mathfrak{U}\boldsymbol{\beta},\mathbb{K}\textrm{-}\mathfrak{Mod}\right)$,
we have that:
\[
G\left(id_{1}\natural\left[n'-n,id_{n'}\right]\right)\circ\eta_{n+1}=\eta_{n'+1}\circ F\left(id_{1}\natural\left[n'-n,id_{n'}\right]\right).
\]
Hence, it follows from the assignment of $\mathbf{LM}$ that:
\[
\mathbf{LM}\left(G\right)\left(\left[n'-n,id_{n'}\right]\right)\circ\mathbf{LM}\left(\eta\right)_{n}=\mathbf{LM}\left(\eta\right)_{n'}\circ\mathbf{LM}\left(F\right)\left(\left[n'-n,id_{n'}\right]\right).
\]
Hence the relation (\ref{eq:criterion3}) of Proposition \ref{prop:criterionnaturaltransfo}
is satisfied, and we deduce from this last proposition that $\mathbf{LM}\left(\eta\right)$
is a morphism in the category $\mathbf{Fct}\left(\mathfrak{U}\boldsymbol{\beta},\mathbb{K}\textrm{-}\mathfrak{Mod}\right)$.
The verification of the composition axiom repeats mutatis mutandis
the one of \ref{enu:It-remains-to}.
\end{enumerate}
\end{enumerate}
\end{proof}
Recall the following fact on the augmentation ideal of the free group
$\mathbf{F}_{n}$ where $n\in\mathbb{N}$.
\begin{prop}
\cite[Chapter 6, Proposition 6.2.6]{Weibel1}\label{prop:-augmentationidealfreemodule}
The augmentation ideal $\mathcal{I}_{\mathbb{K}\left[\mathbf{F}_{n}\right]}$
is a free $\mathbb{K}\left[\mathbf{F}_{n}\right]$-module with basis
the set $\left\{ \left(g_{i}-1\right)\mid i\in\left\{ 1,\ldots,n\right\} \right\} $.
\end{prop}

This result allows us to prove the following properties.
\begin{prop}
\label{prop:exactnessLM}The functor $\mathbf{LM}_{a,\varsigma}:\mathbf{Fct}\left(\mathfrak{U}\boldsymbol{\beta},\mathbb{K}\textrm{-}\mathfrak{Mod}\right)\rightarrow\mathbf{Fct}\left(\mathfrak{U}\boldsymbol{\beta},\mathbb{K}\textrm{-}\mathfrak{Mod}\right)$
is reduced and exact. Moreover, it commutes with all colimits and
all finite limits.
\end{prop}

\begin{proof}
Let $0_{\mathbf{Fct}\left(\mathfrak{U}\boldsymbol{\beta},\mathbb{K}\textrm{-}\mathfrak{Mod}\right)}:\mathfrak{U}\boldsymbol{\beta}\rightarrow\mathbb{K}\textrm{-}\mathfrak{Mod}$
denote the null functor. It follows from the definition of the Long-Moody
functor that $\mathbf{LM}\left(0_{\mathbf{Fct}\left(\mathfrak{U}\boldsymbol{\beta},\mathbb{K}\textrm{-}\mathfrak{Mod}\right)}\right)=0_{\mathbf{Fct}\left(\mathfrak{U}\boldsymbol{\beta},\mathbb{K}\textrm{-}\mathfrak{Mod}\right)}$.\\
Let $n$ be a natural number. Since the augmentation ideal $\mathcal{I}_{\mathbb{K}\left[\mathbf{F}_{n}\right]}$
is a free $\mathbb{K}\left[\mathbf{F}_{n}\right]$-module (as stated
in Proposition \ref{prop:-augmentationidealfreemodule}), it is therefore
a flat $\mathbb{K}\left[\mathbf{F}_{n}\right]$-module. Then, the
result follows from the fact that the functor $\mathcal{I}_{\mathbb{K}\left[\mathbf{F}_{n}\right]}\underset{\mathbb{K}\left[\mathbf{F}_{n}\right]}{\varotimes}-:\mathbb{K}\textrm{-}\mathfrak{Mod}\rightarrow\mathbb{K}\textrm{-}\mathfrak{Mod}$
is an exact functor, the naturality for morphisms following from the
definition of the Long-Moody functor (see Theorem \ref{Thm:LMFunctor}).

Similarly, the fact that the functor \textit{$\mathbf{LM}_{a,\varsigma}$}
commutes with all colimits is a formal consequence of the commutation
with all colimits of the tensor products $\mathcal{I}_{\mathbb{K}\left[\mathbf{F}_{n}\right]}\underset{\mathbb{K}\left[\mathbf{F}_{n}\right]}{\varotimes}-$
for all natural numbers $n$. The commutation result for finite limits
is a property of exact functors (see for example \cite[Chapter 8, section 3]{MacLane1}).
\end{proof}
\begin{rem}
\label{rem:basisfunctor}Let $F$ be an object of $\mathbf{Fct}\left(\mathfrak{U}\boldsymbol{\beta},\mathbb{K}\textrm{-}\mathfrak{Mod}\right)$
and $n$ a natural number. For all $k\in\left\{ 1,\ldots,n\right\} $,
we denote $F\left(n+1\right)_{k}=\mathbb{K}\left[\left(g_{k}-1\right)\right]\underset{\mathbb{K}\left[\mathbf{F}_{n}\right]}{\varotimes}F\left(n+1\right)$
with $g_{k}$ a generator of $\mathbf{F}_{n}$. We define an isomorphism
\begin{eqnarray*}
\varLambda_{n,F}:\mathcal{I}_{\mathbb{K}\left[\mathbf{F}_{n}\right]}\underset{\mathbb{K}\left[\mathbf{F}_{n}\right]}{\varotimes}F\left(n+1\right) & \longrightarrow & \stackrel[k=1]{n}{\bigoplus}F\left(n+1\right)_{k}\cong\left(F\left(n+1\right)\right)^{\oplus n}\\
\left(g_{k}-1\right)\underset{\mathbb{K}\left[\mathbf{F}_{n}\right]}{\varotimes}v & \longmapsto & \left(0,\ldots,0,\overset{\overset{k\textrm{-}th}{\overbrace{}}}{v},0,\ldots,0\right).
\end{eqnarray*}
Thus, for \textit{$\eta:F\rightarrow G$} a natural transformation,
with $\varLambda$:
\[
\forall n\in\mathbb{N},\,\varLambda_{n}\left(\left(\mathbf{LM}\left(\eta\right)\right)_{n}\right)=\eta_{n+1}^{\oplus n}.
\]
Hence, we can have a matricial point of view on this construction
(see \cite[Theorem 2.2]{Long1}). Similarly, the study of Bigelow
and Tian in \cite{BigelowTian} is performed from a purely matricial
point of view.
\end{rem}

\paragraph{Case of trivial $\varsigma$:}

Finally, let us consider the family of morphisms $\left\{ \varsigma_{n,*}:\mathbf{F}_{n}\rightarrow\mathbf{B}_{n+1}\right\} _{n\in\mathbb{N}}$
of Example \ref{exa:trivialsigma}.
\begin{rem}
As stated in Example \ref{exa:exant}, we only need to consider a
family of morphisms $\left\{ a_{n}:\mathbf{B}_{n}\rightarrow Aut\left(\mathbf{F}_{n}\right)\right\} _{n\in\mathbb{N}}$
which satisfies Condition \ref{cond:coherenceconditionbnautfn1} so
that the families $\left\{ \varsigma_{n,*}:\mathbf{F}_{n}\rightarrow\mathbf{B}_{n+1}\right\} _{n\in\mathbb{N}}$
and $\left\{ a_{n}:\mathbf{B}_{n}\rightarrow Aut\left(\mathbf{F}_{n}\right)\right\} _{n\in\mathbb{N}}$
are coherent.
\end{rem}

\begin{notation}
\label{exa:xdeg0}We denote by $\mathfrak{X}:\mathfrak{U}\boldsymbol{\beta}\rightarrow\mathbb{K}\textrm{-}\mathfrak{Mod}$
the constant functor such that $\mathfrak{X}\left(n\right)=\mathbb{K}$
for all natural numbers $n$.
\end{notation}

We have the following remarkable property.
\begin{prop}
\label{prop:casesigmatrivial}Let $F$ be an object of $\mathbf{Fct}\left(\mathfrak{U}\boldsymbol{\beta},\mathbb{K}\textrm{-}\mathfrak{Mod}\right)$
and $\left\{ a_{n}:\mathbf{B}_{n}\rightarrow Aut\left(\mathbf{F}_{n}\right)\right\} _{n\in\mathbb{N}}$
a family of morphisms which satisfies Condition \ref{cond:coherenceconditionbnautfn1}.
Then, as objects of $\mathbf{Fct}\left(\mathfrak{U}\boldsymbol{\beta},\mathbb{K}\textrm{-}\mathfrak{Mod}\right)$,
$\mathbf{LM}_{a,\varsigma_{*}}\left(F\right)\cong\mathbf{LM}_{a,\varsigma_{*}}\left(\mathfrak{X}\right)\underset{\mathbb{K}}{\otimes}F\left(1\natural-\right)$.
\end{prop}

\begin{proof}
Remark \ref{rem:basisfunctor} shows that there is an isomorphism
of $\mathbb{K}$-modules of the form:
\[
\xymatrix{\mathbf{LM}_{a,\varsigma_{*}}\left(F\right)\left(n\right)\ar@{->}[rr]^{\varLambda_{n,F}} &  & \left(F\left(n+1\right)\right)^{\oplus n}\ar@{->}[rr]^{\left(\varLambda_{n,\mathfrak{X}}\underset{\mathbb{K}}{\otimes}id_{F\left(1\natural n\right)}\right)^{-1}\,\,\,\,\,\,\,\,\,\,\,\,\,\,\,\,} &  & \mathbf{LM}_{a,\varsigma_{*}}\left(\mathfrak{X}\right)\left(n\right)\underset{\mathbb{K}}{\otimes}F\left(1\natural n\right)}
.
\]
It is straightforward to check that this isomorphism is natural if
$\varsigma$ is trivial.
\end{proof}

\subsection{Evaluation of the Long-Moody functor\label{subsec:Evaluation-of-the}}

A first step to understand the behaviour of a Long-Moody endofunctor
is to investigate its effect on the constant functor $\mathfrak{X}$.
This is indeed the most basic functor to study. Moreover, as Proposition
\ref{prop:casesigmatrivial} shows, the evaluation on this functor
is the fundamental information to understand a given Long-Moody endofunctor
when we consider the family of morphisms $\left\{ \varsigma_{n,*}:\mathbf{F}_{n}\rightarrow\mathbf{B}_{n+1}\right\} _{n\in\mathbb{N}}$
of Example \ref{exa:trivialsigma}.

Fixing coherent families of morphisms $\left\{ \varsigma_{n}:\mathbf{F}_{n}\rightarrow\mathbf{B}_{n+1}\right\} _{n\in\mathbb{N}}$
and $\left\{ a_{n}:\mathbf{B}_{n}\rightarrow Aut\left(\mathbf{F}_{n}\right)\right\} _{n\in\mathbb{N}}$,
we consider the Long-Moody functor 
\[
\mathbf{LM}_{a,\varsigma}:\mathbf{Fct}\left(\boldsymbol{\beta},\mathbb{K}\textrm{-}\mathfrak{Mod}\right)\rightarrow\mathbf{Fct}\left(\boldsymbol{\beta},\mathbb{K}\textrm{-}\mathfrak{Mod}\right).
\]
For a fixed natural number $n$, using the isomorphism $\varLambda_{n}$
of Remark \ref{rem:basisfunctor}, we observe that $\mathbf{LM}_{a,\varsigma}\left(\mathfrak{X}\right)\left(n\right)\cong\mathbb{K}^{\oplus n}$.
\begin{notation}
\label{nota:yX}Let $y$ be an invertible element of $\mathbb{K}$.
Let $y\mathfrak{X}:\boldsymbol{\beta}\rightarrow\mathbb{K}\textrm{-}\mathfrak{Mod}$
be the functor defined for all natural numbers $n$ by $y\mathfrak{X}\left(n\right)=\mathbb{K}$
and such that:

\begin{itemize}
\item if $n=0$ or $n=1$, then $y\mathfrak{X}\left(id\right)=id_{\mathbb{K}}$;
\item if $n\geq2$, for every Artin generator $\sigma_{i}$ of $\mathbf{B}_{n}$,
$\left(y\mathfrak{X}\right)\left(\sigma_{i}\right):\mathbb{K}\rightarrow\mathbb{K}$
is the multiplication by $y$.
\end{itemize}
For an object $F$ of $\mathbf{Fct}\left(\boldsymbol{\beta},\mathbb{K}\textrm{-}\mathfrak{Mod}\right)$,
we denote the functor $y\mathfrak{X}\underset{\mathbb{K}}{\otimes}F:\boldsymbol{\beta}\rightarrow\textrm{\ensuremath{\mathbb{K}}-}\mathfrak{Mod}$
by $yF$. 
\end{notation}

\subsubsection{\label{subsec:casbur1}Computations for $\mathbf{LM}_{1}$}

Let us assume that $\mathbb{K}=\mathbb{C}\left[t^{\pm1}\right]$.
Let us consider the coherent families of morphisms $\left\{ \varsigma_{n,1}:\mathbf{F}_{n}\hookrightarrow\mathbf{B}_{n+1}\right\} _{n\in\mathbb{N}}$
(introduced in Example \ref{ex:ex1defmapfnbn}) and $\left\{ a_{n,1}:\mathbf{B}_{n}\rightarrow Aut\left(\mathbf{F}_{n}\right)\right\} _{n\in\mathbb{N}}$
(introduced in Example \ref{exan1}). We denote by $\mathbf{LM}_{1}$
the associated Long-Moody functor. We are interested in the behaviour
of the functor $t^{-1}\mathbf{LM}_{1}\left(t\mathfrak{\mathfrak{X}}\right):\boldsymbol{\beta}\longrightarrow\mathbb{C}\left[t^{\pm1}\right]\textrm{-}\mathfrak{Mod}$
on automorphisms of the category $\mathfrak{U}\boldsymbol{\beta}$.
Indeed, adding a parameter $t$ is necessary to recover functors specifically
associated with the category $\mathfrak{U}\boldsymbol{\beta}$, such
as $\mathfrak{Bur}_{t}$ (see Section \ref{subsec:Examples-of-functors}).
Let us fix $n$ a natural number and $\sigma_{i}$ an Artin generator
of $\mathbf{B}_{n}$.

Beforehand, let us understand the action $a_{n,1}:\mathbf{B}_{n}\longrightarrow Aut\left(\mathcal{I}_{\mathbb{K}\left[\mathbf{F}_{n}\right]}\right)$
induced by $a_{n,1}:\mathbf{B}_{n}\rightarrow Aut\left(\mathbf{F}_{n}\right)$.
We compute:
\begin{eqnarray*}
a_{n,1}\left(\sigma_{i}\right):\mathcal{I}_{\mathbb{K}\left[\mathbf{F}_{n}\right]} & \longrightarrow & \mathcal{I}_{\mathbb{K}\left[\mathbf{F}_{n}\right]}\\
g_{j}-1 & \longmapsto & \begin{cases}
g_{i+1}-1 & \textrm{if \ensuremath{j=i}}\\
g_{i+1}^{-1}g_{i}g_{i+1}-1=\left[g_{i}-1\right]g_{i+1}+\left[g_{i+1}-1\right]\left(1-g_{i+1}^{-1}g_{i}g_{i+1}\right) & \textrm{if \ensuremath{j=i+1}}\\
g_{j}-1 & \textrm{if \ensuremath{j\notin\left\{ i,i+1\right\} \textrm{.}}}
\end{cases}
\end{eqnarray*}

Hence, we have the following result.
\begin{prop}
\label{prop:recoveringunredBurau} As objects of $\mathbf{Fct}\left(\boldsymbol{\beta},\mathbb{K}\textrm{-}\mathfrak{Mod}\right)$,
$t^{-1}\mathbf{LM}_{1}\left(t\mathfrak{X}\right)=\mathfrak{Bur}_{t^{2}}$.
\end{prop}

\begin{proof}
Using the isomorphism $\varLambda_{n}$ of Remark \ref{rem:basisfunctor},
we obtain that for $\sigma_{i}$ an Artin generator of $\mathbf{B}_{n}$:
\[
t^{-1}\mathbf{LM}_{1}\left(t\mathfrak{\mathfrak{X}}\right)\left(\sigma_{i}\right)=Id_{i-1}\oplus\left[\begin{array}{cc}
0 & t^{2}\\
1 & 1-t^{2}
\end{array}\right]\oplus Id_{n-i-1}=\mathfrak{Bur}_{t^{2}}\left(\sigma_{i}\right).
\]
\end{proof}

\paragraph{Recovering of the Lawrence-Krammer functor:}

Let us first introduce the following result due to Long in \cite{Long1}.
We assume that $\mathbb{K}=\mathbb{C}\left[t^{\pm1}\right]\left[q^{\pm1}\right]$.
For this paragraph, we assume that $1+qt=0$, $q$ has a square root,
$q^{2}\neq1$ and $q^{3}\neq1$.
\begin{notation}
We denote by $\mathfrak{X}':\boldsymbol{\beta}\longrightarrow\mathbb{C}\left[t^{\pm1}\right]\left[q^{\pm1}\right]\textrm{-}\mathfrak{Mod}$
the constant functor such that $\mathfrak{X}'\left(n\right)=\mathbb{C}\left[t^{\pm1}\right]\left[q^{\pm1}\right]$
for all natural numbers $n$. Generally speaking, for $F$ an object
of $\mathbf{Fct}\left(\boldsymbol{\beta},\mathbb{K}\textrm{-}\mathfrak{Mod}\right)$
the representation of $\mathbf{B}_{n}$ induced by $F$ will be denoted
by $F_{\mid\mathbf{B}_{n}}$.
\end{notation}

\begin{prop}
\cite[ special case of Corollary 2.10]{Long1}\label{prop:lksubrep}
Let $n$ be a natural number such that $n\geq4$. Then, the Lawrence-Krammer
representation $\mathfrak{LK}{}_{\mid\mathbf{B}_{n}}$ is a subrepresentation
of $q^{-1}\left(\mathbf{LM}_{1}\left(q\left(t^{-1}\mathbf{LM}_{1}\left(t\mathfrak{X}\right)\right)\right)\right){}_{\mid\mathbf{B}_{n}}$.
\end{prop}

We first need to introduce new tools. Let $n$ and $m$ be two natural
numbers. Let $\underline{w}_{n}=\left(w_{1},\ldots,w_{n}\right)\in\mathbb{C}^{n}$
such that $w_{i}\neq w_{j}$ if $\ensuremath{i\neq j}$. We consider
the configuration space: 
\[
Y_{\underline{w}_{n},m}=\left\{ \left(z_{1},\ldots,z_{m}\right)\mid z_{i}\in\mathbb{C},\,z_{i}\neq w_{k}\,\textrm{for \ensuremath{1\leq k\leq n}},\,z_{i}\neq z_{j}\,\textrm{if \ensuremath{i\neq j}}\right\} .
\]
The two following results due to Long will be crucial to prove Proposition
\ref{prop:lksubrep}.
\begin{prop}
\cite[Corollary 2.7]{Long1}\label{prop:firstprooflksubrep} Let $n$
be a natural number and $\rho:\mathbf{B}_{n+1}\rightarrow GL\left(V\right)$
be a representation of $\mathbf{B}_{n}$ with $V$ a $\mathbb{C}\left[t^{\pm1}\right]\left[q^{\pm1}\right]$-module.
Then, the representation defined by Long in \cite[Theorem 2.1]{Long1},
which we denote by $\mathcal{LM}$, is a group morphism:
\[
q^{-1}\mathcal{LM}\left(q\rho\right):\mathbf{B}_{n}\rightarrow GL\left(H^{1}\left(Y_{\underline{w}_{n},1},E_{\rho}\right)\right)
\]
for $E_{\rho}$ a flat vector bundle associated with $\rho$ (see
\cite[p. 225-226]{Long1}).
\end{prop}

\begin{lem}
\cite[Lemma 2.9]{Long1}\label{lem:secondprooflksubrep} For all natural
numbers $m$, there is an isomorphism of abelian groups:
\[
H^{m+1}\left(Y_{\underline{w}_{n},m+1},E_{\mathfrak{X}_{\mid\mathbf{B}_{n}}}\right)\cong H^{1}\left(Y_{\underline{w}_{n},1},H^{m}\left(Y_{\underline{w}_{n+1},m},E_{\mathfrak{X}_{\mid\mathbf{B}_{n}}}\right)\right).
\]
In particular, for $m=1$, $H^{2}\left(Y_{\underline{w}_{n},2},E_{\mathfrak{X}_{\mid\mathbf{B}_{n}}}\right)\cong H^{1}\left(Y_{\underline{w}_{n},1},H^{1}\left(Y_{\underline{w}_{n+1},2},E_{\mathfrak{X}_{\mid\mathbf{B}_{n}}}\right)\right)$.
\begin{proof}
[Proof of Proposition 2.33]By Proposition \ref{prop:firstprooflksubrep},
we can write as a representation:
\[
q^{-1}\mathcal{LM}\left(q\left(t^{-1}\mathcal{LM}\left(t\mathfrak{X}\right)\right)\right):\mathbf{B}_{n}\rightarrow GL\left(H^{1}\left(Y_{\underline{w}_{n},1},E_{t^{-1}\mathcal{LM}\left(t\mathfrak{X}\right)}\right)\right).
\]
A fortiori by Lemma \ref{lem:secondprooflksubrep}, $q^{-1}\mathcal{LM}\left(q\left(t^{-1}\mathcal{LM}\left(t\mathfrak{X}_{\mid\mathbf{B}_{n}}\right)\right)\right)$
is an action of $\mathbf{B}_{n}$ on $H^{2}\left(Y_{\underline{w}_{n},2},E_{\mathfrak{X}_{\mid\mathbf{B}_{n}}}\right)$.
In particular, for $m=2$ and $n\geq4$, according to \cite[Theorem 5.1]{Lawrence},
the representation of $\mathbf{B}_{n}$ factoring through the Iwahori\textendash Hecke
algebra $H_{n}\left(t\right)$ corresponding to the Young diagram
$\left(n-2,2\right)$ is a subrepresentation of $q^{-1}\mathcal{LM}\left(q\left(t^{-1}\mathcal{LM}\left(t\mathfrak{X}_{\mid\mathbf{B}_{n}}\right)\right)\right)$.
Moreover, this representation is equivalent to the Lawrence-Krammer
representation by \cite[Section 5]{biglowlk}. By the definition of
the Long-Moody construction (see \cite[Theorem 2.1]{Long1}), $q^{-1}\mathcal{LM}\left(q\left(t^{-1}\mathcal{LM}\left(t\mathfrak{X}_{\mid\mathbf{B}_{n}}\right)\right)\right)$
is the representation $q^{-1}\left(\tau_{1}\mathbf{LM}_{1}\right)\left(q\left(t^{-1}\mathbf{LM}_{1}\left(t\mathfrak{X}\right)\right)\right){}_{\mid\mathbf{B}_{n}}$.
\end{proof}
\end{lem}

We denote by $\mathfrak{LK}^{\geq4}:\boldsymbol{\beta}\longrightarrow\left(\mathbb{C}\left[t^{\pm1}\right]\right)\left[q^{\pm1}\right]\textrm{-}\mathfrak{Mod}$
the subfunctor of the Lawrence-Krammer defined in Example \ref{exa:lkfunctor}
which is null on the objects such that $n<4$. The result of Proposition
\ref{prop:lksubrep} implies that:
\begin{prop}
The functor $\mathfrak{LK}^{\geq4}$ is a subfunctor of $q^{-1}\left(\tau_{1}\mathbf{LM}_{1}\right)\left(q\left(t^{-1}\mathbf{LM}_{1}\left(t\mathfrak{X}\right)\right)\right){}^{\geq4}$.
\end{prop}

\subsubsection{Computations for other cases}

Let us introduce examples of Long-Moody functors which arise using
other actions $a_{n}:\mathbf{B}_{n}\rightarrow Aut\left(\mathbf{F}_{n}\right)$.

\paragraph{Wada representations}

In $1992$, Wada introduced in \cite{wada1992group} a certain type
of family of representations of braid groups. We give here a functorial
approach to this work.
\begin{defn}
\label{def:functorAutF}Let $Aut_{-}:\left(\mathbb{N},\leq\right)\rightarrow\mathfrak{Gr}$
be the functor defined by:

\begin{itemize}
\item Objects: for all natural numbers $n$, $Aut_{-}\left(n\right)=Aut\left(\mathbf{F}_{n}\right)$
the automorphism group of the free group on $n$ generators;
\item Morphisms: let $n$ be a natural number. We define $Aut_{-}\left(\gamma{}_{n}\right):Aut\left(\mathbf{F}_{n}\right)\hookrightarrow Aut\left(\mathbf{F}_{n+1}\right)$
assigning $Aut_{-}\left(\gamma{}_{n}\right)\left(\varphi\right)=id_{1}*\varphi$
for all $\varphi\in Aut\left(\mathbf{F}_{n}\right)$, using the monoidal
category $\left(\mathfrak{gr},*,0\right)$ recalled in Notation \ref{nota:grandN}.
\end{itemize}
\end{defn}

\begin{defn}
\label{def:outilwada}Let us consider two different non-trivial reduced
words $W\left(g_{1},g_{2}\right)$ and $V\left(g_{1},g_{2}\right)$
on $\mathbf{F}_{2}$, such that:

\begin{itemize}
\item the assignments $g_{1}\mapsto W\left(g_{1},g_{2}\right)$ and $g_{2}\mapsto V\left(g_{1},g_{2}\right)$
define a automorphism of $\mathbf{F}_{2}$;
\item the assignment $\left(W,V\right):\mathbf{B}_{2}\longrightarrow Aut\left(\mathbf{F}_{2}\right)$:
\[
\left[\left(W,V\right)\left(\sigma_{1}\right)\right]\left(g_{j}\right)=\begin{cases}
W\left(g_{1},g_{2}\right) & \textrm{if \ensuremath{j=1}}\\
V\left(g_{1},g_{2}\right) & \textrm{if \ensuremath{j=2}}
\end{cases}
\]
is a morphism.
\end{itemize}
Two morphisms $\left(W,V\right):\mathbf{B}_{2}\longrightarrow Aut\left(\mathbf{F}_{2}\right)$
and $\left(W',V'\right):\mathbf{B}_{2}\rightarrow Aut\left(\mathbf{F}_{2}\right)$
are said to be swap-dual if $W'\left(g_{1},g_{2}\right)=V\left(g_{2},g_{1}\right)$
and $V'\left(g_{1},g_{2}\right)=W\left(g_{2},g_{1}\right)$, backward-dual
if $W'\left(g_{1},g_{2}\right)=\left(W\left(g_{1}^{-1},g_{2}^{-1}\right)\right)^{-1}$
and $V'\left(g_{1},g_{2}\right)=\left(V\left(g_{1}^{-1},g_{2}^{-1}\right)\right)^{-1}$,
inverse if $\left(W',V'\right)=\left(W,V\right)^{-1}$.
\end{defn}

\begin{defn}
\cite{wada1992group} Let $W\left(g_{1},g_{2}\right)$ and $V\left(g_{1},g_{2}\right)$
be two words on $\mathbf{F}_{2}$. A natural transformation $\mathcal{W}:\mathbf{B}_{-}\rightarrow Aut_{-}$
is said to be of Wada-type if for all natural numbers $n$, for all
$i\in\left\{ 1,\ldots,n-1\right\} $, the following diagram is commutative
(we recall that $incl_{i}^{n}$ was introduced in Notation \ref{nota:defincl}
and $Aut_{-}\left(\gamma{}_{2,i}\right)$ in Definition \ref{def:functorAutF}):
\[
\xymatrix{\mathbf{B}_{n}\ar@{->}[rr]^{\mathcal{W}_{n}} &  & Aut\left(\mathbf{F}_{n}\right)\\
\mathbf{B}_{2}\ar@{->}[rr]_{\left(W,V\right)}\ar@{->}[u]^{incl_{i}^{n}} &  & Aut\left(\mathbf{F}_{2}\right).\ar@{->}[u]_{Aut_{-}\left(\gamma{}_{2,i}\right)\ast id_{\mathbf{F}_{n-i-1}}}
}
\]
\end{defn}

\begin{rem}
Note that therefore a Wada-type natural transformation is entirely
determined by the choice of $\left(W,V\right)$.
\end{rem}

Wada conjectured a classification of these type of representations.
This conjecture was proved by Ito in \cite{Ito}.
\begin{thm}
\cite{Ito}\label{thm:Repwada} There are seven classes of Wada-type
natural transformation $\mathcal{W}$ up to the swap-dual, backward-dual
and inverse equivalences, listed below.

\begin{enumerate}
\item $\left(W,V\right)_{1,m}\left(g_{1},g_{2}\right)=\left(g_{2},g_{2}^{m}g_{1}g_{2}^{-m}\right)$
where $m\in\mathbb{Z}$;
\item $\left(W,V\right)_{2}\left(g_{1},g_{2}\right)=\left(g_{1},g_{2}\right)$;
\item $\left(W,V\right)_{3}\left(g_{1},g_{2}\right)=\left(g_{2},g_{1}^{-1}\right)$;
\item $\left(W,V\right)_{4}\left(g_{1},g_{2}\right)=\left(g_{2},g_{2}^{-1}g_{1}^{-1}g_{2}\right)$;
\item $\left(W,V\right)_{5}\left(g_{1},g_{2}\right)=\left(g_{2}^{-1},g_{1}^{-1}\right)$;
\item $\left(W,V\right)_{6}\left(g_{1},g_{2}\right)=\left(g_{2}^{-1},g_{2}g_{1}g_{2}\right)$;
\item $\left(W,V\right)_{7}\left(g_{1},g_{2}\right)=\left(g_{1}g_{2}^{-1}g_{1}^{-1},g_{1}g_{2}^{2}\right)$.
\end{enumerate}
\end{thm}

\begin{rem}
Note that the action given by the first Wada representation with $m=1$
is a generalization of the Artin representation.
\end{rem}

\begin{notation}
\label{nota:notationwada}The actions given by the $k$-th Wada-type
natural transformation will be denoted by $a_{n,k}:\mathbf{B}_{n}\hookrightarrow Aut\left(\mathbf{F}_{n}\right)$.
In particular, for $k=1$ with $m=1$, we recover the Artin representation
(see Example \ref{exan1}).\\

For all $1\leq k\leq8$, it clearly follows from their definitions
that the families of morphisms $\left\{ a_{n,k}:\mathbf{B}_{n}\rightarrow Aut\left(\mathbf{F}_{n}\right)\right\} _{n\in\mathbb{N}}$
satisfy Condition \ref{cond:coherenceconditionbnautfn1}. Hence, for
$1\leq k\leq8$, we consider a family of morphisms $\left\{ \varsigma_{n,k}:\mathbf{F}_{n}\rightarrow\mathbf{B}_{n+1}\right\} $
assumed to be coherent with respect to the morphisms $\left\{ a_{n,k}:\mathbf{B}_{n}\hookrightarrow Aut\left(\mathbf{F}_{n}\right)\right\} _{n\in\mathbb{N}}$
(in the sense of Definition \ref{def:coherentmor}). Such morphisms
$\varsigma_{n,k}$ always exist because we could at least take the
family of morphisms $\left\{ \varsigma_{n,*}:\mathbf{F}_{n}\rightarrow\mathbf{B}_{n+1}\right\} $
(see Example \ref{exa:exant}). We denote by $\mathbf{LM}_{k}:\mathbf{Fct}\left(\boldsymbol{\beta},\mathbb{K}\textrm{-}\mathfrak{Mod}\right)\rightarrow\mathbf{Fct}\left(\boldsymbol{\beta},\mathbb{K}\textrm{-}\mathfrak{Mod}\right)$
the corresponding Long-Moody functor defined in Theorem \ref{Thm:LMFunctor}
for $k\in\left\{ 1,\ldots,8\right\} $.
\end{notation}

Let us imitate the procedure of Section \ref{subsec:casbur1}. We
assume that $\mathbb{K}=\mathbb{C}\left[t^{\pm1}\right]$. Let $n$
be a fixed natural number. Let us consider the case of $k=2$. Using
the isomorphism $\varLambda_{n}$ of Remark \ref{rem:basisfunctor},
we obtain the functor $\mathbf{LM}_{2}\left(\mathfrak{\mathfrak{X}}\right):\boldsymbol{\beta}\rightarrow\mathbb{C}\left[t^{\pm1}\right]\textrm{-}\mathfrak{Mod}$,
defined for $\sigma_{i}\in\mathbf{B}_{n}$ by:\textit{
\[
\mathbf{LM}_{2}\left(F\right)\left(\sigma_{i}\right)=\left(F\left(\sigma_{i}\right)\right)^{\oplus n}.
\]
}

For $k=3$, using $\varLambda_{n}$, we compute that the functor $t^{-1}\mathbf{LM}_{3}\left(t\mathfrak{\mathfrak{X}}\right):\boldsymbol{\beta}\rightarrow\mathbb{C}\left[t^{\pm1}\right]\textrm{-}\mathfrak{Mod}$
is defined for $\sigma_{i}\in\mathbf{B}_{n}$ by:
\begin{eqnarray*}
t^{-1}\mathbf{LM}_{3}\left(t\mathfrak{\mathfrak{X}}\right)\left(\sigma_{i}\right) & = & Id_{i-1}\oplus\left[\begin{array}{cc}
0 & -\varsigma_{n,3}\left(g_{i}\right)\\
1 & 0
\end{array}\right]\oplus Id_{n-i-1}.
\end{eqnarray*}
Hence, the functor $t^{-1}\mathbf{LM}_{3}\left(t\mathfrak{X}\right)$
is very similar to the one associated with the Tong-Yang-Ma representations
(recall Definition \ref{exa:deftym}). We deduce that the identity
natural equivalence gives $t^{-1}\mathbf{LM}_{3}\left(t\mathfrak{X}\right)\cong\mathfrak{TYM}_{-\varsigma_{n,3}\left(g_{i}\right)}$
as objects of $\mathbf{Fct}\left(\boldsymbol{\beta},\mathbb{K}\textrm{-}\mathfrak{Mod}\right)$.

For the actions given by the Wada-type natural transformation $4$,
$5$, $6$ and $7$ in Theorem \ref{thm:Repwada}, the produced functors
$t^{-1}\mathbf{LM}_{i}\left(t\mathfrak{X}\right):\boldsymbol{\beta}\longrightarrow\mathbb{C}\left[t^{\pm1}\right]\textrm{-}\mathfrak{Mod}$
are mild variants of what is given by the case $i=1$.

\section{Strong polynomial functors\label{sec:Strong-polynomial-functors}}

We deal here with the concept of a strong polynomial functor. This
type of functor will be the core of our work in Section \ref{sec:The-Long-Moody-functoreffect}.
We review (and actually extend) the definition and properties of a
strong polynomial functor due to Djament and Vespa in \cite{DV3}
and also a particular case of coefficient systems of finite degree
used by Randal-Williams and Wahl in \cite{WahlRandal-Williams}.

In \cite[Section 1]{DV3}, Djament and Vespa construct a framework
to define strong polynomial functors in the category $\mathbf{Fct}\left(\mathfrak{M},\mathcal{A}\right)$,
where $\mathfrak{M}$ is a symmetric monoidal category, the unit is
an initial object and $\mathcal{A}$ is an abelian category. Here,
we generalize this definition for functors from pre-braided monoidal
categories having the same additional property. In particular, the
notion of strong polynomial functor will be defined for the category
$\mathbf{Fct}\left(\mathfrak{U}\boldsymbol{\beta},\mathbb{K}\textrm{-}\mathfrak{Mod}\right)$.
The keypoint of this section is Proposition \ref{prop:extprebraided},
in so far as it constitutes the crucial property necessary and sufficient
to extend the definition of strong polynomial functor to the pre-braided
case.

\subsection{Strong polynomiality}

We first introduce the translation functor, which plays the central
role in the definition of strong polynomiality.
\begin{defn}
\label{def:deftaux}Let $\left(\mathfrak{M},\natural,0\right)$ be
a strict monoidal small category, let $\mathfrak{D}$ be a category
and let $x$ be an object of $\mathfrak{M}$. The monoidal structure
defines the endofunctor $x\natural-:\mathfrak{M}\longrightarrow\mathfrak{M}$.
We define the translation by $x$ functor $\tau_{x}:\mathbf{Fct}\left(\mathfrak{M},\mathfrak{D}\right)\rightarrow\mathbf{Fct}\left(\mathfrak{M},\mathfrak{D}\right)$
to be the endofunctor obtained by precomposition by the functor $x\natural-$.
\end{defn}

The following proposition establishes the commutation of two translation
functors associated with two objects of $\mathfrak{M}$. It is the
keystone property to define strong polynomial functors.
\begin{prop}
\label{prop:extprebraided}Let $\left(\mathfrak{M},\natural,0\right)$
be a pre-braided strict monoidal small category and $\mathfrak{D}$
be a category. Let $x$ and $y$ be two objects of $\mathfrak{M}$.
Then, there exists a natural isomorphism between functors from $\mathbf{Fct}\left(\mathfrak{M},\mathfrak{\mathfrak{D}}\right)$
to $\mathbf{Fct}\left(\mathfrak{M},\mathfrak{\mathfrak{D}}\right)$:
\[
\tau_{x}\circ\tau_{y}\cong\tau_{y}\circ\tau_{x}.
\]
\end{prop}

\begin{proof}
First, because of the associativity of the monoidal product $\natural$
and the strictness of $\mathfrak{M}$, we have that $\tau_{x}\circ\tau_{y}=\tau_{x\natural y}$
and $\tau_{y}\circ\tau_{x}=\tau_{y\natural x}$. We denote by $b_{-,-}^{\mathfrak{M}}$
the pre-braiding of $\mathfrak{M}$. The key point is the fact that
as $b_{-,-}^{\mathfrak{M}}$ is a braiding on the maximal subgroupoid
of $\mathfrak{M}$ (see Definition \ref{def:defprebraided}), $b_{x,y}^{\mathfrak{M}}:x\natural y\overset{\cong}{\longrightarrow}y\mathcal{\natural}x$
defines an isomorphism. Hence, precomposition by $b_{x,y}^{\mathfrak{M}}\natural id_{\mathfrak{M}}$
defines a natural transformation $\left(b_{x,y}^{\mathfrak{M}}\natural id_{\mathfrak{M}}\right)^{*}:\tau_{x\natural y}\rightarrow\tau_{y\natural x}$.
It is an isomorphism since we analogously construct an inverse natural
transformation $\left(\left(b_{x,y}^{\mathfrak{M}}\right)^{-1}\natural id_{\mathfrak{M}}\right)^{*}:\tau_{y\natural x}\rightarrow\tau_{x\natural y}$.
\end{proof}
\begin{rem}
In Proposition \ref{prop:extprebraided}, the natural isomorphism
is not unique: as the proof shows, we could have used the morphism
$\left(b_{y,x}^{\mathfrak{M}}\right)^{-1}\natural id_{\mathfrak{M}}$
instead to define an isomorphism between $\tau_{x\natural y}\left(F\right)$
and $\tau_{y\natural x}\left(F\right)$. In fact, a category only
needs to be equipped with natural (in $x$ and $y$) isomorphisms
$x\natural y\cong y\natural x$ to satisfy the conclusion of Proposition
\ref{prop:extprebraided}.
\end{rem}

Let us move on to the introduction of the evanescence and difference
functors, which will characterize the (very) strong polynomiality
of a functor in $\mathbf{Fct}\left(\mathfrak{M},\mathcal{A}\right)$.
Recall that, if $\mathfrak{M}$ is a small category and $\mathcal{A}$
is an abelian category, then the functor category $\mathbf{Fct}\left(\mathfrak{M},\mathcal{A}\right)$
is an abelian category (see \cite[Chapter VIII]{MacLane1}).

\textbf{From now until the end of Section \ref{sec:Strong-polynomial-functors},
we fix $\left(\mathfrak{M},\natural,0\right)$ a pre-braided strict
monoidal category such that the monoidal unit $0$ is an initial object,
$\mathcal{A}$ an abelian category and $x$ denotes an object of $\mathfrak{M}$.}
\begin{defn}
\label{def:defix}For all objects $F$ of $\mathbf{Fct}\left(\mathfrak{M},\mathcal{A}\right)$,
we denote by $i_{x}\left(F\right):\tau_{0}\left(F\right)\rightarrow\tau_{x}\left(F\right)$
the natural transformation induced by the unique morphism $\iota_{x}:0\rightarrow x$
of $\mathfrak{M}$. This induces $i_{x}:Id_{\mathbf{Fct}\left(\mathfrak{M},\mathcal{A}\right)}\rightarrow\tau_{x}$
a natural transformation of $\mathbf{Fct}\left(\mathfrak{M},\mathcal{A}\right)$.
Since the category $\mathbf{Fct}\left(\mathfrak{M},\mathcal{A}\right)$
is abelian, the kernel and cokernel of the natural transformation
$i_{x}$ exist. We define the functors $\kappa_{x}=\ker\left(i_{x}\right)$
and $\delta_{x}=\textrm{coker}\left(i_{x}\right)$. The endofunctors
$\kappa_{x}$ and $\delta_{x}$ of $\mathbf{Fct}\left(\mathfrak{M},\mathcal{A}\right)$
are called respectively evanescence and difference functor associated
with $x$. 
\end{defn}

The following proposition presents elementary properties of the translation,
evanescence and difference functors. They are either consequences
of the definitions, or direct generalizations of the framework considered
in \cite{DV3} where $\mathfrak{M}$ is symmetric monoidal.
\begin{prop}
\label{prop:lemmecaract}Let $y$ be an object of $\mathfrak{M}$.
Then the translation functor $\tau_{x}$ is exact and we have the
following exact sequence in the category of endofunctors of $\mathbf{Fct}\left(\mathfrak{M},\mathcal{A}\right)$:
\begin{equation}
0\longrightarrow\kappa_{x}\overset{\Omega_{x}}{\longrightarrow}Id\overset{i_{x}}{\longrightarrow}\tau_{x}\overset{\varDelta_{x}}{\longrightarrow}\delta_{x}\longrightarrow0.\label{eq:ESCaract}
\end{equation}
Moreover, for a short exact sequence $0\longrightarrow F\longrightarrow G\longrightarrow H\longrightarrow0$
in the category $\mathbf{Fct}\left(\mathfrak{M},\mathcal{A}\right)$,
there is a natural exact sequence in the category $\mathbf{Fct}\left(\mathfrak{M},\mathcal{A}\right)$:
\begin{equation}
0\longrightarrow\kappa_{x}\left(F\right)\longrightarrow\kappa_{x}\left(G\right)\longrightarrow\kappa_{x}\left(H\right)\longrightarrow\delta_{x}\left(F\right)\longrightarrow\delta_{x}\left(G\right)\longrightarrow\delta_{x}\left(H\right)\longrightarrow0.\label{eq:LESkappadelta}
\end{equation}
In addition:

\begin{enumerate}
\item The translation endofunctor $\tau_{x}$ of $\mathbf{Fct}\left(\mathfrak{M},\mathcal{A}\right)$
commutes with limits and colimits.
\item The difference endofunctors $\delta_{x}$ and $\delta_{y}$ of $\mathbf{Fct}\left(\mathfrak{M},\mathcal{A}\right)$
commute up to natural isomorphism. They commute with colimits.
\item The endofunctors $\kappa_{x}$ and $\kappa_{y}$ of $\mathbf{Fct}\left(\mathfrak{M},\mathcal{A}\right)$
commute up to natural isomorphism. They commute with limits.
\item The natural inclusion $\kappa_{x}\circ\kappa_{x}\hookrightarrow\kappa_{x}$
is an isomorphism.
\item The translation endofunctor $\tau_{x}$ and the difference endofunctor
$\delta_{y}$ commute up to natural isomorphism.
\item The translation endofunctor $\tau_{x}$ and the endofunctor $\kappa_{y}$
commute up to natural isomorphism.
\item We have the following natural exact sequence in the category of endofunctors
of $\mathbf{Fct}\left(\mathfrak{M},\mathcal{A}\right)$:
\begin{equation}
0\longrightarrow\kappa_{y}\longrightarrow\kappa_{x\natural y}\longrightarrow\tau_{x}\kappa_{y}\longrightarrow\delta_{y}\longrightarrow\delta_{x\natural y}\longrightarrow\tau_{y}\delta_{x}\longrightarrow0.\label{eq:LESsumobject}
\end{equation}
\end{enumerate}
\end{prop}

\begin{proof}
In the symmetric monoidal case, this is \cite[Proposition 1.4]{DV3}:
the numbered properties are formal consequences of the commutation
property of the translation endofunctors given by Proposition \ref{prop:extprebraided}.
Hence, the proofs carry over mutatis mutandis to the pre-braided setting.
\end{proof}
Using Proposition \ref{prop:lemmecaract}, we can define strong polynomial
functors.
\begin{defn}
We recursively define on $n\in\mathbb{N}$ the category $\mathcal{P}ol_{n}^{strong}\left(\mathfrak{M},\mathcal{A}\right)$
of strong polynomial functors of degree less than or equal to $n$
to be the full subcategory of $\mathbf{Fct}\left(\mathfrak{M},\mathcal{A}\right)$
as follows:

\begin{enumerate}
\item If $n<0$, $\mathcal{P}ol_{n}^{strong}\left(\mathfrak{M},\mathcal{A}\right)=\left\{ 0\right\} $;
\item if $n\geq0$, the objects of $\mathcal{P}ol_{n}^{strong}\left(\mathfrak{M},\mathcal{A}\right)$
are the functors $F$ such that for all objects $x$ of $\mathfrak{M}$,
the functor $\delta_{x}\left(F\right)$ is an object of $\mathcal{P}ol_{n-1}^{strong}\left(\mathfrak{M},\mathcal{A}\right)$.
\end{enumerate}
For an object $F$ of $\mathbf{Fct}\left(\mathfrak{M},\mathcal{A}\right)$
which is strong polynomial of degree less than or equal to $n\in\mathbb{N}$,
the smallest $d\in\mathbb{N}$ ($d\leq n$) for which $F$ is an object
of $\mathcal{P}ol_{d}^{strong}\left(\mathfrak{M},\mathcal{A}\right)$
is called the strong degree of $F$.
\end{defn}

\begin{rem}
By Proposition \ref{prop:homogenousprebraided}, the category $\left(\mathfrak{U}\boldsymbol{\beta},\natural,0\right)$
is a pre-braided monoidal category such that $0$ is initial object.
This example is the first one which led us to extend the definition
of \cite{DV3}. Thus, we have a well-defined notion of strong polynomial
functor for the category $\mathfrak{U}\boldsymbol{\beta}$.
\end{rem}

The following three propositions are important properties of the framework
in \cite{DV3} adapted to the pre-braided case. Their proofs follow
directly from those of their analogues in \cite[Propositions 1.7, 1.8 and 1.9]{DV3}.
\begin{prop}
\cite[Proposition 1.7]{DV3} Let $\mathfrak{M}'$ be another pre-braided
strict monoidal category such that such that its monoidal unit is
an initial object and $\alpha:\mathfrak{M}\rightarrow\mathfrak{M}'$
be a strong monoidal functor. Then, the precomposition by $\alpha$
provides a functor  $\mathcal{P}ol_{n}^{strong}\left(\mathfrak{M},\mathcal{A}\right)\rightarrow\mathcal{P}ol_{n}^{strong}\left(\mathfrak{M}',\mathcal{A}\right)$.
\end{prop}

\begin{prop}
\cite[Proposition 1.8]{DV3}\label{prop:proppoln} The category $\mathcal{P}ol_{n}^{strong}\left(\mathfrak{M},\mathcal{A}\right)$
is closed under the translation endofunctor $\tau_{x}$, under quotient,
under extension and under colimits. Moreover, assuming that there
exists a set $\mathfrak{E}$ of objects of $\mathfrak{M}$ such that:
\[
\forall m\in Obj\left(\mathfrak{M}\right),\,\,\exists\left\{ e_{i}\right\} _{i\in I}\in Obj\left(\mathfrak{E}\right)\textrm{ where \ensuremath{I} is finite},\,\,m\cong\underset{i\in I}{\natural}e_{i},
\]
then, an object $F$ of $\mathbf{Fct}\left(\mathfrak{M},\mathcal{A}\right)$
belongs to $\mathcal{P}ol_{n}^{strong}\left(\mathfrak{M},\mathcal{A}\right)$
if and only if $\delta_{e}\left(F\right)$ is an object of $\mathcal{P}ol_{n-1}^{strong}\left(\mathfrak{M},\mathcal{A}\right)$
for all objects $e$ of $\mathfrak{E}$.
\end{prop}

\begin{cor}
\label{Corlem:directsummand} Let $n$ be a natural number. Let $F$
be a strong polynomial functor of degree $n$ in the category $\mathbf{Fct}\left(\mathfrak{M},\mathcal{A}\right)$.
Then a direct summand of $F$ is necessarily an object of the category
$\mathcal{P}ol_{n}^{strong}\left(\mathfrak{M},\mathcal{A}\right)$.
\end{cor}

\begin{proof}
According to Proposition \ref{prop:proppoln}, the category $\mathcal{P}ol_{n}^{strong}\left(\mathfrak{M},\mathcal{A}\right)$
is closed under quotients.
\end{proof}
\begin{rem}
\label{rem:notclosedundersubobjects}The category $\mathcal{P}ol_{n}^{strong}\left(\mathfrak{M},\mathcal{A}\right)$
is not necessarily closed under subobjects. For example, we will see
in  Section \ref{subsec:Examples-of-polynomial} that for $\mathfrak{M}=\mathfrak{U}\boldsymbol{\beta}$
and $\mathcal{A}=\mathbb{C}\left[t^{\pm1}\right]\textrm{-}\mathfrak{Mod}$,
the functor $\overline{\mathfrak{Bur}}_{t}$ is a subobject of $\tau_{1}\overline{\mathfrak{Bur}}_{t}$
(see Proposition \ref{Prop:reduburdeg2}), $\overline{\mathfrak{Bur}}_{t}$
is strong polynomial of degree $2$ (see Proposition \ref{Prop:reduburdeg2})
whereas $\tau_{1}\overline{\mathfrak{Bur}}_{t}$ is strong polynomial
of degree $1$ (see Proposition \ref{prop:tau1bureddeg1}). If we
assume that the unit $0$ is also a terminal object of $\mathfrak{M}$,
then $\kappa_{x}$ is the null endofunctor, $\delta_{x}$ is exact
and commutes with all limits. In this case, the category $\mathcal{P}ol_{n}^{strong}\left(\mathfrak{M},\mathcal{A}\right)$
is closed under subobjects.
\end{rem}

\begin{rem}
\label{exa:suffit1}If we consider $\mathfrak{M}=\mathfrak{U}\boldsymbol{\beta}$,
then each object $n$ (ie a natural number) is clearly $1^{\natural n}$.
Hence, because of the last statement of Proposition \ref{prop:proppoln},
when we will deal with strong polynomiality of objects in $\mathbf{Fct}\left(\mathfrak{U}\boldsymbol{\beta},\mathcal{A}\right)$,
it will suffice to consider $\tau_{1}$.
\end{rem}

\begin{prop}
\cite[Proposition 1.9]{DV3}\label{prop:strongdef0} Let $F$ be an
object of $\mathbf{Fct}\left(\mathfrak{M},\mathcal{A}\right)$. Then,
the functor $F$ is an object of $\mathcal{P}ol_{0}^{strong}\left(\mathfrak{M},\mathcal{A}\right)$
if and only if it is the quotient of a constant functor of $\mathbf{Fct}\left(\mathfrak{M},\mathcal{A}\right)$.
\end{prop}

Finally, let us point out the following property of the strong polynomial
degree with respect to the translation functor.
\begin{lem}
\label{lem:varepsiFdecreasedegree}Let $d$ and $k$ be natural numbers
and $F$ be an object of $\mathbf{Fct}\left(\mathfrak{U}\boldsymbol{\beta},\textrm{\ensuremath{\mathbb{K}}-}\mathfrak{Mod}\right)$
such that $\tau_{k}\left(F\right)$ is an object of $\mathcal{P}ol_{d}^{strong}\left(\mathfrak{U}\boldsymbol{\beta},\textrm{\ensuremath{\mathbb{K}}-}\mathfrak{Mod}\right)$.
Then, $F$ is an object of $\mathcal{P}ol_{d+k}\left(\mathfrak{U}\boldsymbol{\beta},\textrm{\ensuremath{\mathbb{K}}-}\mathfrak{Mod}\right)$.
\end{lem}

\begin{proof}
We proceed by induction on the degree of polynomiality of $\tau_{k}\left(F\right)$.
First, assuming that $\tau_{k}\left(F\right)$ belongs to $\mathcal{P}ol_{0}^{strong}\left(\mathfrak{U}\boldsymbol{\beta},\textrm{\ensuremath{\mathbb{K}}-}\mathfrak{Mod}\right)$,
we deduce from the commutation property $6$ of Proposition \ref{prop:lemmecaract}
that $\tau_{k}\left(\delta_{1}F\right)=0$. It follows from the definition
of $\tau_{k}\left(F\right)$ (see Definition \ref{def:deftaux}) that
for all $n\geq2$, $\delta_{1}\left(F\right)\left(n\right)=0$. Hence
\[
\underset{\textrm{\ensuremath{k+1} times}}{\underbrace{\delta_{1}\cdots\delta_{1}\delta_{1}}}\left(F\right)\cong0
\]
and therefore $F$ is an object of $\mathcal{P}ol_{k}\left(\mathfrak{U}\boldsymbol{\beta},\textrm{\ensuremath{\mathbb{K}}-}\mathfrak{Mod}\right)$.
Now, assume that $\tau_{k}\left(F\right)$ is a strong polynomial
functor of degree $d\geq0$. Since $\left(\tau_{k}\circ\delta_{1}\right)\left(F\right)\cong\left(\delta_{1}\circ\tau_{k}\right)\left(F\right)$
by the commutation property $6$ of Proposition \ref{prop:lemmecaract},
$\left(\tau_{k}\circ\delta_{1}\right)\left(F\right)$ is an object
of \textit{$\mathcal{P}ol_{d-1}^{strong}\left(\mathfrak{U}\boldsymbol{\beta},\textrm{\ensuremath{\mathbb{K}}-}\mathfrak{Mod}\right)$.}
The inductive hypothesis implies that $\delta_{1}\left(F\right)$
is an object of \textit{$\mathcal{P}ol_{d+k}^{strong}\left(\mathfrak{U}\boldsymbol{\beta},\textrm{\ensuremath{\mathbb{K}}-}\mathfrak{Mod}\right)$.}
\end{proof}
\begin{rem}
Let us consider the atomic functor $\mathfrak{A}_{n}$ (with $n>0$),
which is strong polynomial of degree $n$ (see Example \ref{exa:atomic functor}).
Then $\tau_{k}\left(\mathfrak{A}_{n}\right)\cong\mathfrak{A}_{n-k}^{\oplus n}$
is strong polynomial of degree $n-k$, for $k$ a natural number such
that $k\leq n$. This illustrates the fact that $d+k$ is the best
boundary for the degree of polynomiality in Lemma \ref{lem:varepsiFdecreasedegree}.
\end{rem}

\subsection{Very strong polynomial functors}

Let us introduce a particular type of strong polynomial functor, related
to coefficient systems of finite degree (see Remark \ref{rem:coefficientsyst}
below). We recall that we consider a pre-braided strict monoidal category
$\left(\mathfrak{M},\natural,0\right)$ such that the monoidal unit
$0$ is an initial object and an abelian category $\mathcal{A}$.
\begin{defn}
\label{def:defverystrong}We recursively define the category $\mathcal{VP}ol_{n}\left(\mathfrak{M},\mathcal{A}\right)$
of very strong polynomial functors of degree less than or equal to
$n$ to be the full subcategory of $\mathcal{P}ol_{n}^{strong}\left(\mathfrak{M},\mathcal{A}\right)$
as follows:

\begin{enumerate}
\item If $n<0$, $\mathcal{VP}ol_{n}\left(\mathfrak{M},\mathcal{A}\right)=\left\{ 0\right\} $;
\item if $n\geq0$, a functor $F\in\mathcal{P}ol_{n}^{strong}\left(\mathfrak{M},\mathcal{A}\right)$
is an object of $\mathcal{VP}ol_{n}\left(\mathfrak{M},\mathcal{A}\right)$
if for all objects $x$ of $\mathfrak{M}$, $\kappa_{x}\left(F\right)=0$
and the functor $\delta_{x}\left(F\right)$ is an object of $\mathcal{\mathcal{VP}}ol_{n-1}\left(\mathfrak{M},\mathcal{A}\right)$.
\end{enumerate}
\end{defn}

For an object $F$ of $\mathbf{Fct}\left(\mathfrak{M},\mathcal{A}\right)$
which is very strong polynomial of degree less than or equal to $n\in\mathbb{N}$,
the smallest $d\in\mathbb{N}$ ($d\leq n$) for which $F$ is an object
of $\mathcal{VP}ol_{d}\left(\mathfrak{M},\mathcal{A}\right)$ is called
the very strong degree of $F$.
\begin{rem}
\label{rem:coefficientsyst}A certain type of functor, called a coefficient
system of finite degree, closely related to the strong polynomial
one, is used by Randal-Williams and Wahl in \cite[Definition 4.10]{WahlRandal-Williams}
for their homological stability theorems, generalizing the concept
introduced by van der Kallen for general linear groups \cite{vanderKallen}.
Using the framework introduced by Randal-Williams and Wahl, a coefficient
system in every object $x$ of $\mathfrak{M}$ of degree $n$ at $N=0$
is a very strong polynomial functor.
\end{rem}

\begin{rem}
\label{rem:notclosedundersubobjects'}As we force $\kappa_{x}$ to
be null for all objects $x$ of $\mathfrak{M}$, the category $\mathcal{\mathcal{VP}}ol_{n}\left(\mathfrak{M},\mathcal{A}\right)$
is closed under kernel functors of the epimorphisms. In particular,
this category is closed under direct summands. However, $\mathcal{\mathcal{VP}}ol_{n}\left(\mathfrak{M},\mathcal{A}\right)$
is not necessarily closed under subobjects. For instance, as for Remark
\ref{rem:notclosedundersubobjects}, we have that the functor $\overline{\mathfrak{Bur}}_{t}$
is strong polynomial of degree $2$ (see Proposition \ref{Prop:reduburdeg2}),
the functor $\tau_{1}\overline{\mathfrak{Bur}}_{t}$ is very strong
polynomial of degree $1$ (see Proposition \ref{prop:tau1bureddeg1}),
but $\overline{\mathfrak{Bur}}_{t}$ is a subobject of $\tau_{1}\overline{\mathfrak{Bur}}_{t}$
(see Proposition \ref{Prop:reduburdeg2}).
\end{rem}

\begin{prop}
\label{prop:prop verystrongpoly}The category $\mathcal{VP}ol_{n}\left(\mathfrak{M},\mathcal{A}\right)$
is closed under the translation endofunctor $\tau_{x}$, under kernel
of epimorphism and under extension. Moreover, assuming that there
exists a set $\mathfrak{E}$ of objects of $\mathfrak{M}$ such that:
\[
\forall m\in Obj\left(\mathfrak{M}\right),\,\,\exists\left\{ e_{i}\right\} _{i\in I}\in Obj\left(\mathfrak{E}\right)\textrm{ (where \ensuremath{I} is finite)},\,\,m\cong\underset{i\in I}{\natural}e_{i},
\]
then, an object $F$ of $\mathbf{Fct}\left(\mathfrak{M},\mathcal{A}\right)$
belongs to $\mathcal{VP}ol_{n}\left(\mathfrak{M},\mathcal{A}\right)$
if and only if $\kappa_{e}\left(F\right)=0$ and $\delta_{e}\left(F\right)$
is an object of $\mathcal{V}\mathcal{P}ol_{n-1}\left(\mathfrak{M},\mathcal{A}\right)$
for all objects $e$ of $\mathfrak{E}$.
\end{prop}

\begin{proof}
The first assertion follows from the fact that for all objects $x$
of $\mathfrak{M}$, the endofunctor $\tau_{x}$ commutes with the
endofunctors $\delta_{x}$ and $\kappa_{x}$ (see Proposition \ref{prop:lemmecaract}).
For the second and third assertions, let us consider two short exact
sequences of $\mathbf{Fct}\left(\mathfrak{M},\mathcal{A}\right)$:
$0\longrightarrow G\longrightarrow F_{1}\longrightarrow F_{2}\longrightarrow0$
and $0\longrightarrow F_{3}\longrightarrow H\longrightarrow F_{4}\longrightarrow0$
with $F_{i}$ a very strong polynomial functor of degree $n$ for
all $i$. Let $x$ be an object of $\mathfrak{M}$. We use the exact
sequence (\ref{eq:LESkappadelta}) of Proposition \ref{prop:lemmecaract}
to obtain the two following exact sequences in the category $\mathbf{Fct}\left(\mathfrak{M},\mathcal{A}\right)$:
\[
0\longrightarrow\kappa_{x}\left(G\right)\longrightarrow0\longrightarrow0\longrightarrow\delta_{x}\left(G\right)\longrightarrow\delta_{x}\left(F_{1}\right)\longrightarrow\delta_{x}\left(F_{2}\right)\longrightarrow0;
\]
\[
0\longrightarrow0\longrightarrow\kappa_{x}\left(H\right)\longrightarrow0\longrightarrow\delta_{x}\left(F_{3}\right)\longrightarrow\delta_{x}\left(H\right)\longrightarrow\delta_{x}\left(F_{4}\right)\longrightarrow0.
\]
Therefore, $\kappa_{x}\left(G\right)=\kappa_{x}\left(H\right)=0$
and the result follows directly by induction on the degree of polynomiality.
For the last point, we consider the long exact sequence (\ref{eq:LESsumobject})
of Proposition \ref{prop:lemmecaract} applied to an object $F$ of
$\mathcal{VP}ol_{n}\left(\mathfrak{M},\mathcal{A}\right)$ to obtain
the following exact sequence in the category $\mathbf{Fct}\left(\mathfrak{M},\mathcal{A}\right)$:
\[
0\longrightarrow\kappa_{y}\left(F\right)\longrightarrow\kappa_{x\natural y}\left(F\right)\longrightarrow\tau_{x}\kappa_{y}\left(F\right)\longrightarrow\delta_{y}\left(F\right)\longrightarrow\delta_{x\natural y}\left(F\right)\longrightarrow\tau_{y}\delta_{x}\left(F\right)\longrightarrow0.
\]
Hence, by induction on the length of objects as monoidal product of
$\left\{ e_{i}\right\} _{i\in I}$, we deduce that $\kappa_{m}\left(F\right)=0$
for all objects $m$ of $\mathfrak{M}$ if and only if $\kappa_{e}\left(F\right)=0$
for all objects $e$ of $\mathfrak{E}$. Moreover, since $\mathcal{VP}ol_{n}\left(\mathfrak{M},\mathcal{A}\right)$
is closed under extension and by the translation endofunctor $\tau_{y}$,
the result follows by induction on the degree of polynomiality $n$.
\end{proof}
\begin{prop}
\label{prop:degre0verystrong}Let $F$ be an object of $\mathbf{Fct}\left(\mathfrak{\mathfrak{M}},\mathcal{A}\right)$.
The functor $F$ is an object of $\mathcal{VP}ol_{0}\left(\mathfrak{M},\mathcal{A}\right)$
if and only if it is isomorphic to $\tau_{k}F$ for all natural numbers
$k$.
\end{prop}

\begin{proof}
The result follows using the long exact sequence (\ref{eq:ESCaract})
of Proposition \ref{prop:lemmecaract} applied to $F$.
\end{proof}
The following example show that there exist strong polynomial functors
which are not very strong polynomial in any degree.
\begin{example}
\label{exa:atomic functor}Let us consider the categories $\mathfrak{U}\boldsymbol{\beta}$
and $\mathbb{K}\textrm{-}\mathfrak{Mod}$, and $n$ a natural number.
Let $\mathbb{K}$ be considered as an object of $\mathbb{K}\textrm{-}\mathfrak{Mod}$
and $0$ be the trivial $\mathbb{K}$-module. Let $\mathfrak{A}_{n}$
be an object of $\mathbf{Fct}\left(\mathfrak{U}\boldsymbol{\beta},\mathbb{K}\textrm{-}\mathfrak{Mod}\right)$,
defined by:

\begin{itemize}
\item Objects: $\forall m\in\mathbb{N}$, $\mathfrak{A}_{n}\left(m\right)=\begin{cases}
\mathbb{K} & \textrm{if \ensuremath{n=m}}\\
0 & \textrm{otherwise}
\end{cases}$.
\item Morphisms: let \emph{$\left[j-i,f\right]$ }with $f\in\mathbf{B}_{n}$
be a morphism from $i$ to $j$ in the category $\mathfrak{U}\boldsymbol{\beta}$.
Then: 
\[
\mathfrak{A}_{n}\left(f\right)=\begin{cases}
id_{\mathbb{K}} & \textrm{if \ensuremath{i=j=n}}\\
0 & \textrm{otherwise.}
\end{cases}
\]
\end{itemize}
The functor $\mathfrak{A}_{n}$ is called an atomic functor in $\mathbb{K}$
of degree $n$. For coherence, we fix $\mathfrak{A}_{-1}$ to be the
null functor of $\mathbf{Fct}\left(\mathfrak{U}\boldsymbol{\beta},\mathbb{K}\textrm{-}\mathfrak{Mod}\right)$.
Then, it is clear that $i_{p}\left(\mathfrak{A}_{n}\right)$ is the
zero natural transformation. On the one hand, we deduce the following
natural equivalence $\kappa_{1}\left(\mathfrak{A}_{n}\right)\cong\mathfrak{A}_{n}$
and a fortiori $\mathfrak{A}_{n}$ is not a very strong polynomial
functor. On the other hand, it is worth noting the natural equivalence
$\delta_{1}\left(\mathfrak{A}_{n}\right)\cong\tau_{1}\left(\mathfrak{A}_{n}\right)$
and the fact that $\tau_{1}\left(\mathfrak{A}_{n}\right)\cong\mathfrak{A}_{n-1}$.
Therefore, we recursively prove that $\mathfrak{A}_{n}$ is a strong
polynomial functor of degree $n$.
\end{example}

\begin{rem}
Contrary to $\mathcal{P}ol_{n}^{strong}\left(\mathfrak{M},\mathcal{A}\right)$,
a quotient of an object $F$ of $\mathcal{VP}ol_{n}\left(\mathfrak{M},\mathcal{A}\right)$
is not necessarily a very strong polynomial functor. For example,
for $\mathfrak{M}=\mathfrak{U}\boldsymbol{\beta}$ and $\mathcal{A}=\mathbb{K}\textrm{-}\mathfrak{Mod}$,
let us consider the functor $\mathfrak{A}_{0}$ defined in Example
\ref{exa:atomic functor}, which we proved to be a strong polynomial
functor of degree $0$. Let $\mathfrak{A}$ be the constant object
of $\mathbf{Fct}\left(\mathfrak{U}\boldsymbol{\beta},\mathbb{K}\textrm{-}\mathfrak{Mod}\right)$
equal to $\mathbb{K}$. Then, we define a natural transformation $\alpha:\mathfrak{A}\rightarrow\mathfrak{A}_{0}$
assigning:
\[
\forall n\in\mathbb{N},\,\,\alpha_{n}=\begin{cases}
id_{\mathbb{K}} & \textrm{if \ensuremath{n=0}}\\
t_{\mathbb{K}} & \textrm{otherwise.}
\end{cases}
\]
Moreover, it is an epimorphism in the category $\mathbf{Fct}\left(\mathfrak{U}\boldsymbol{\beta},\mathbb{K}\textrm{-}\mathfrak{Mod}\right)$
since for all natural numbers $n$, $coker\left(\alpha_{n}\right)=0_{\mathbb{K}\textrm{-}\mathfrak{Mod}}$.
We proved in Example \ref{exa:atomic functor} that $\mathfrak{A}_{0}$
is not a very strong polynomial functor of degree $0$ whereas $\mathfrak{A}$
is a very strong polynomial functor of degree $0$ by Proposition
\ref{prop:degre0verystrong}.
\end{rem}

Finally, let us remark the following behaviour of the translation
functor with respect to very strong polynomial degree.
\begin{lem}
\label{lem:hypotheseverystrongok}Let $d$ and $k$ be a natural numbers
and $F$ be an object of $\mathcal{VP}ol_{d}\left(\mathfrak{\mathfrak{M}},\textrm{\ensuremath{\mathbb{K}}-}\mathfrak{Mod}\right)$.
Then the functor $\tau_{k}\left(F\right)$ is very strong polynomial
of degree equal to that of $F$.
\end{lem}

\begin{proof}
We proceed by induction on the degree of polynomiality of $F$. First,
if we assume that $F$ belongs to $\mathcal{VP}ol_{0}\left(\mathfrak{\mathfrak{M}},\textrm{\ensuremath{\mathbb{K}}-}\mathfrak{Mod}\right)$,
then according to Proposition \ref{prop:degre0verystrong}, $\tau_{k}\left(F\right)\cong F$
is a degree $0$ very strong polynomial functor. Now, assume that
$F$ is a very strong polynomial functor of degree $n\geq0$. Using
the commutation properties $5$ and $6$ of Proposition \ref{prop:lemmecaract},
we deduce that $\left(\kappa_{1}\circ\tau_{k}\right)\left(F\right)\cong\left(\tau_{k}\circ\kappa_{1}\right)\left(F\right)=0$
and $\left(\delta_{1}\circ\tau_{k}\right)\left(F\right)\cong\left(\tau_{k}\circ\delta_{1}\right)\left(F\right)$.
Since the functor $\delta_{1}\left(F\right)$ is a degree $n-1$ very
strong polynomial functor, the result follows from the inductive hypothesis.
\end{proof}
\begin{rem}
The previous proof does not work for strong polynomial functors since
the initial step fails. Indeed, considering the atomic functor $\mathfrak{A}_{1}$,
which is strong polynomial of degree $1$ (see Example \ref{exa:atomic functor}),
then $\tau_{2}\left(\mathfrak{A}_{0}\right)=0$.
\end{rem}

\subsection{Examples of polynomial functors over $\mathfrak{U}\boldsymbol{\beta}$\label{subsec:Examples-of-polynomial}}

The different functors introduced in Section \ref{subsec:Examples-of-functors}
are strong polynomial functors.

\paragraph{Very strong polynomial functors of degree one:}

Let us first investigate the polynomiality of the functors $\mathfrak{Bur}_{t}$
and $\mathfrak{TYM}_{t}$.
\begin{prop}
\label{prop:BurTYMverystrong}The functors $\mathfrak{Bur}_{t}$ and
$\mathfrak{TYM}_{t}$ are very strong polynomial functors of degree
$1$.
\begin{proof}
For the functor $\mathfrak{Bur}_{t}$, the proof is mutatis mutandis
the same as the one for the dual version considered in \cite[Example 4.15]{WahlRandal-Williams}.
We will thus focus on the case of the functor $\mathfrak{TYM}_{t}$.
Let $n$ be a natural number. By Remark \ref{exa:suffit1}, it is
enough to consider the application $i_{1}\mathfrak{TYM_{t}}\left(\left[0,id_{n}\right]\right)=\iota_{\mathbb{C}\left[t^{\pm1}\right]^{\oplus n'-n}}\oplus id_{\mathbb{C}\left[t^{\pm1}\right]^{\oplus n}}$.
This map is a monomorphism and its cokernel is $\mathbb{C}\left[t^{\pm1}\right]$.
Hence $\kappa_{1}\mathfrak{TYM}_{t}$ is the null functor of $\mathbf{Fct}\left(\mathfrak{U}\boldsymbol{\beta},\mathbb{C}\left[t^{\pm1}\right]\textrm{-}\mathfrak{Mod}\right)$.
Let $n'$ be a natural number such that $n'\geq n$ and let $\left[n'-n,\sigma\right]\in Hom_{\mathfrak{U}\boldsymbol{\beta}}\left(n,n'\right)$.
By naturality and the universal property of the cokernel, there exists
a unique endomorphism of $\mathbb{C}\left[t^{\pm1}\right]$ such that
the following diagram commutes, where the lines are exact. It is exactly
the definition of $\delta_{1}\mathfrak{TYM}_{t}\left(\left[n'-n,\sigma\right]\right)$.
\[
\xymatrix{0\ar@{->}[r] & \mathbb{C}\left[t^{\pm1}\right]^{\oplus n}\ar@{->}[rrr]^{\iota_{\mathbb{C}\left[t^{\pm1}\right]}\oplus id_{\mathbb{C}\left[t^{\pm1}\right]^{\oplus n}}}\ar@{->}[d]_{\mathfrak{TYM}\left(\left[n'-n,\sigma\right]\right)} &  &  & \mathbb{C}\left[t^{\pm1}\right]^{\oplus n+1}\ar@{->}[rr]^{\pi_{n+1}}\ar@{->}[d]^{\tau_{1}\left(\mathfrak{TYM}\right)\left(\left[n'-n,\sigma\right]\right)} &  & \mathbb{C}\left[t^{\pm1}\right]\ar@{->}[r]\ar@{.>}[d]^{\exists!} & 0\\
0\ar@{->}[r] & \mathbb{C}\left[t^{\pm1}\right]^{\oplus n'}\ar@{->}[rrr]_{\iota_{\mathbb{C}\left[t^{\pm1}\right]}\oplus id_{\mathbb{C}\left[t^{\pm1}\right]^{\oplus n'}}} &  &  & \mathbb{C}\left[t^{\pm1}\right]^{\oplus n'+1}\ar@{->}[rr]_{\pi_{n'+1}} &  & \mathbb{C}\left[t^{\pm1}\right]\ar@{->}[r] & 0.
}
\]
For all $\left(a,b\right)\in\mathbb{C}\left[t^{\pm1}\right]\oplus\mathbb{C}\left[t^{\pm1}\right]^{\oplus n}=\mathbb{C}\left[t^{\pm1}\right]^{\oplus n+1}$,
$\tau_{1}\left(\mathfrak{TYM}_{t}\right)\left(\left[n'-n,\sigma\right]\right)\left(a,b\right)=\left(a,\mathfrak{TYM}_{t}\left(\left[n'-n,\sigma\right]\right)\left(b\right)\right)$.
Therefore, $\left(\pi_{n'+1}\circ\tau_{1}\left(\mathfrak{TYM}_{t}\right)\left(\left[n'-n,\sigma\right]\right)\right)\left(a,b\right)=a=\pi_{n+1}\left(a,b\right)$.
Hence, $id_{\mathbb{C}\left[t^{\pm1}\right]}$ also makes the diagram
commutative and thus $\delta_{1}\mathfrak{TYM}_{t}\left(\left[n'-n,\sigma\right]\right)=id_{\mathbb{C}\left[t^{\pm1}\right]}.$
Hence, $\delta_{1}\mathfrak{TYM}_{t}$ is the constant functor equal
to $\mathbb{C}\left[t^{\pm1}\right]$. A fortiori, because of Proposition
\ref{prop:degre0verystrong}, $\delta_{1}\mathfrak{TYM}_{t}$ is a
very strong polynomial functor of degree $0$.
\end{proof}
\end{prop}

\paragraph{The particular case of $\overline{\mathfrak{Bur}}_{t}$:}
\begin{defn}
\label{def:outilssupppl}Let $\mathcal{T}_{1}:\mathfrak{U}\boldsymbol{\beta}\longrightarrow\mathbb{C}\left[t^{\pm1}\right]\textrm{-}\mathfrak{Mod}$
be the subobject of the constant functor $\mathfrak{X}$ (see Notation
\ref{exa:xdeg0}) such that $\mathcal{T}_{1}\left(0\right)=0$ and
$\mathcal{T}_{1}\left(n\right)=\mathbb{C}\left[t^{\pm1}\right]$ for
all non-zero natural numbers $n$.
\end{defn}

\begin{rem}
\label{rem:T1strongpolydeg1}It follows from Definition \ref{def:outilssupppl}
that $\delta_{1}\mathcal{T}_{1}\cong\mathfrak{A}_{0}$ (where $\mathfrak{A}_{0}$
is introduced in Example \ref{exa:atomic functor}). Therefore, $\mathcal{T}_{1}$
is a strong polynomial functor of degree $1$, but is not very strong
polynomial. Nevertheless, it is worth noting that $\kappa_{1}\mathcal{T}_{1}=0$.
\end{rem}

\begin{prop}
\label{Prop:reduburdeg2}The functor $\overline{\mathfrak{Bur}}$
is a strong polynomial functor of degree $2$. This functor is not
very strong polynomial. More precisely, we have the following short
exact sequence in $\mathbf{Fct}\left(\mathfrak{U}\boldsymbol{\beta},\mathbb{C}\left[t^{\pm1}\right]\textrm{-}\mathfrak{Mod}\right)$:
\[
\xymatrix{0\ar@{->}[r] & \overline{\mathfrak{Bur}}_{t}\ar@{->}[r] & \tau_{1}\overline{\mathfrak{Bur}}_{t}\ar@{->}[r] & \mathcal{T}_{1}\ar@{->}[r] & 0}
.
\]
\begin{proof}
The natural transformation $i_{1}\left(\overline{\mathfrak{Bur}}_{t}\right)_{n}:\overline{\mathfrak{Bur}}_{t}\left(n\right)\rightarrow\tau_{1}\overline{\mathfrak{Bur}}_{t}\left(n\right)$
(introduced in Definition \ref{def:defix}) is defined to be $\iota_{\mathbb{C}\left[t^{\pm1}\right]^{\oplus n'-n}}\oplus id_{\mathbb{C}\left[t^{\pm1}\right]^{\oplus n-1}}$.
Let $n\geq2$ be a natural number. This map is a monomorphism (so
$\kappa_{1}\overline{\mathfrak{Bur}}_{t}=0$) and its cokernel is
$\mathbb{C}\left[t^{\pm1}\right]$. Repeating mutatis mutandis the
work done in the proof of Proposition \ref{prop:BurTYMverystrong},
we deduce that for all $\left[n'-n,\sigma\right]\in Hom_{\mathfrak{U}\boldsymbol{\beta}}\left(n,n'\right)$
(with $n'\geq n\geq2$), $\delta_{1}\overline{\mathfrak{Bur}}_{t}\left(\left[n'-n,\sigma\right]\right)=Id_{\mathbb{C}\left[t^{\pm1}\right]}$.
In addition, since $\overline{\mathfrak{Bur}}_{t}\left(1\right)=0$
and $\tau_{1}\overline{\mathfrak{Bur}}_{t}\left(1\right)=\mathbb{C}\left[t^{\pm1}\right]$,
we deduce that $\delta_{1}\overline{\mathfrak{Bur}}_{t}\left(1\right)=\mathbb{C}\left[t^{\pm1}\right]$
and for all $n'\geq1$, for all $\left[n'-1,\sigma\right]\in Hom_{\mathfrak{U}\boldsymbol{\beta}}\left(1,n'\right)$,
$\delta_{1}\overline{\mathfrak{Bur}}_{t}\left(\left[n'-1,\sigma\right]\right)=Id_{\mathbb{C}\left[t^{\pm1}\right]}$.
Hence, we prove that $\delta_{1}\overline{\mathfrak{Bur}}_{t}\cong\mathcal{T}_{1}$
where $\mathcal{T}_{1}$ is introduced in Definition \ref{def:outilssupppl}.
The results follow from the fact that $\delta_{1}\mathcal{T}_{1}\cong\mathfrak{A}_{0}$
by Remark \ref{rem:T1strongpolydeg1}.
\end{proof}
\end{prop}

For formal reasons (see Proposition \ref{prop:lemmecaract}), $\overline{\mathfrak{Bur}}_{t}$
is a subfunctor of $\tau_{1}\overline{\mathfrak{Bur}}_{t}$. The following
proposition illustrates Remarks \ref{rem:notclosedundersubobjects}
and \ref{rem:notclosedundersubobjects'}.
\begin{prop}
\label{prop:tau1bureddeg1}The functor $\tau_{1}\overline{\mathfrak{Bur}}_{t}$
is a very strong polynomial functor of degree $1$.
\end{prop}

\begin{proof}
Repeating mutatis mutandis the work done in the proof of Proposition
\ref{Prop:reduburdeg2}, we prove that $\delta_{1}\tau_{1}\overline{\mathfrak{Bur}}_{t}$
is the constant functor equal to $\mathbb{C}\left[t^{\pm1}\right]$
(denoted by $\mathfrak{X}$ in Notation \ref{exa:xdeg0}). Since $\mathfrak{X}$
is a constant functor, $\delta_{1}\tau_{1}\overline{\mathfrak{Bur}}_{t}$
is by Proposition \ref{prop:degre0verystrong} a very strong polynomial
functor of degree $0$.
\end{proof}

\paragraph{A very strong polynomial functor of degree two:}

We could have defined the unreduced Burau functor of Example \ref{exa:defunredbur}
assigning $\left(\left(\mathbb{C}\left[t^{\pm1}\right]\right)\left[q^{\pm1}\right]\right)^{\oplus n}$
to each object $n\in\mathbb{N}$.
\begin{notation}
\label{nota:defburhat}Abusing the notation, $\left(\mathbb{C}\left[t^{\pm1}\right]\right)\left[q^{\pm1}\right]:\mathfrak{U}\boldsymbol{\beta}\rightarrow\left(\mathbb{C}\left[t^{\pm1}\right]\right)\left[q^{\pm1}\right]\textrm{-}\mathfrak{Mod}$
denotes the constant functor at $\left(\mathbb{C}\left[t^{\pm1}\right]\right)\left[q^{\pm1}\right]$.
The functor $\mathfrak{Bur}_{t}\underset{\mathbb{C}\left[t^{\pm1}\right]}{\otimes}\left(\mathbb{C}\left[t^{\pm1}\right]\right)\left[q^{\pm1}\right]$
is denoted by $\check{\mathfrak{Bur}_{t}}:\mathfrak{U}\boldsymbol{\beta}\rightarrow\left(\mathbb{C}\left[t^{\pm1}\right]\right)\left[q^{\pm1}\right]\textrm{-}\mathfrak{Mod}$.
\end{notation}

\begin{rem}
These functors $\left(\mathbb{C}\left[t^{\pm1}\right]\right)\left[q^{\pm1}\right]$
and $\check{\mathfrak{Bur}_{t}}$ are also very strong polynomial
of degree one (the proof is exactly the same as the one for $\mathfrak{Bur}_{t}$
in Proposition \ref{rem:T1strongpolydeg1}).
\end{rem}

\begin{lem}
\label{lem:delta1lkbur}Considering the modified version of the unreduced
Burau functor of Remark \ref{nota:defburhat}, then $\delta_{1}\mathfrak{LK}$
is equivalent to $\check{\mathfrak{Bur}_{t}}$.
\end{lem}

\begin{proof}
We consider the application $i_{1}\mathfrak{LK}\left(\left[0,id_{n}\right]\right)$.
This map is a monomorphism and its cokernel is $\underset{2\leq l\leq n+1}{\bigoplus}V_{1,l}$.
Let $n$ and $n'$ be two natural numbers such that $n'\geq n$. Let
$\left[n'-n,\sigma\right]\in Hom_{\mathfrak{U}\boldsymbol{\beta}}\left(n,n'\right)$.
By naturality and because of the universal property of the cokernel,
there exists a unique endomorphism of $\left(\mathbb{C}\left[t^{\pm1}\right]\right)\left[q^{\pm1}\right]$-modules
such that the following diagram commutes, where the lines are exact.
It is exactly the definition of $\delta_{1}\mathfrak{LK}\left(\left[n'-n,\sigma\right]\right)$.
\[
\xymatrix{0\ar@{->}[r] & \underset{1\leq j<k\leq n}{\bigoplus}V_{j,k}\ar@{->}[rrr]^{\mathfrak{LK}\left(\left[1,id_{1+n}\right]\right)}\ar@{->}[d]_{\mathfrak{LK}\left(\left[n'-n,\sigma\right]\right)} &  &  & \underset{1\leq i<l\leq n+1}{\bigoplus}V_{i,l}\ar@{->}[rrr]^{{\color{white}ooo}\pi_{n}}\ar@{->}[d]^{\tau_{1}\left(\mathfrak{LK}\right)\left(\left[n'-n,\sigma\right]\right)} &  &  & \underset{2\leq l\leq n+1}{\bigoplus}V_{1,l}\ar@{->}[r]\ar@{.>}[d]^{\exists!} & 0\\
0\ar@{->}[r] & \underset{1\leq j'<k'\leq n'}{\bigoplus}V_{j',k'}\ar@{->}[rrr]_{\mathfrak{LK}\left(\left[1,id_{1+n'}\right]\right)} &  &  & \underset{1\leq l'\leq n'+1}{\bigoplus}V_{i',l'}\ar@{->}[rrr]_{{\color{white}ooo}\pi_{n'}} &  &  & \underset{2\leq l'\leq n'+1}{\bigoplus}V_{1,l'}\ar@{->}[r] & 0.
}
\]
Let $i\in\left\{ 1,\ldots,n-1\right\} $, $l\in\left\{ 2,\ldots,n+1\right\} $
and $v_{1,l}$ be an element of $V_{1,l}$. Then we compute:
\[
\tau_{1}\mathfrak{LK}\left(\sigma_{i}\right)v_{1,l}=\mathfrak{LK}\left(\sigma_{1+i}\right)\left(v_{1,l}\right)=\begin{cases}
v_{1,l} & \textrm{if \ensuremath{i+1\notin\left\{ l-1,l\right\} },}\\
tv_{1,i+1}+\left(1-t\right)v_{1,i+2}-\left(t^{2}-t\right)qv_{i+1,i+2} & \textrm{if \ensuremath{i+2=l}},\\
v_{1,i+2} & \textrm{if \ensuremath{i+1=l}}.
\end{cases}
\]
We deduce that in the canonical basis $\left\{ \mathbf{e}_{1,2},\mathbf{e}_{1,3},\ldots,\mathbf{e}_{1,n+1}\right\} $
of $\underset{2\leq l\leq n+1}{\bigoplus}V_{1,l}$: 
\[
\delta_{1}\mathfrak{LK}\left(\sigma_{i}\right)=Id_{i-1}\oplus\left[\begin{array}{cc}
0 & t\\
1 & 1-t
\end{array}\right]\oplus Id_{n-i-1}=\check{\mathfrak{Bur}_{t}}\left(\sigma_{i}\right).
\]
So as to identify $\delta_{1}\mathfrak{LK}$, it remains to consider
the action on morphisms of type $\left[1,id_{n+1}\right]$. According
to the definition of the Lawrence-Krammer functor, we have $\tau_{1}\left(\mathfrak{LK}\right)\left(\left[1,id_{n+1}\right]\right)=\mathfrak{LK}\left(\sigma_{1}^{-1}\right)\circ\mathfrak{LK}\left(\left[1,id_{n+2}\right]\right)$
and:
\[
\mathfrak{LK}\left(\sigma_{1}\right)\left(v_{1,k}\right)=\begin{cases}
v_{2,k} & \textrm{if \ensuremath{k\in\left\{ 3,\ldots,n+2\right\} }},\\
-qt^{2}v_{1,2} & \textrm{if \ensuremath{k=2}}.
\end{cases}
\]
It follows that for all $v_{i,l}\in V_{i,l}$ with $1\leq i<l\leq n+1$:
\[
\pi_{n+1}\circ\tau_{1}\left(\mathfrak{LK}\right)\left(\left[1,id_{n+1}\right]\right)\left(v_{i,l}\right)=\begin{cases}
v_{1,l+1} & \textrm{if \ensuremath{i=1} and \ensuremath{l\in\left\{ 2,\ldots,n+1\right\} },}\\
0 & \textrm{otherwise.}
\end{cases}
\]
Hence, we deduce that for all $2\leq l\leq n+1$, $\delta_{1}\mathfrak{LK}\left(\left[1,id_{n+1}\right]\right)\left(v_{1,l}\right)=v_{1,l+1}=\check{\mathfrak{Bur}_{t}}\left(\left[1,id_{n+1}\right]\right)\left(v_{1,l}\right)$.
\end{proof}
\begin{prop}
\label{prop:LKpoly2}The functor $\mathfrak{LK}$ is a very strong
polynomial functor of degree $2$.
\begin{proof}
Let $n$ be a natural number. By Remark \ref{exa:suffit1}, we only
have to consider the application $i_{1}\mathfrak{LK}\left(\left[0,id_{n}\right]\right)$.
Since this map is a monomorphism with cokernel $\underset{1\leq i\leq n}{\bigoplus}V_{i,n+1}$,
$\kappa_{1}\mathfrak{\mathfrak{LK}}$ is the null constant functor.
Since the functor $\check{\mathfrak{Bur}_{t}}$ is very strong polynomial
of degree one (following exactly the same proof as the one of Proposition
\ref{prop:BurTYMverystrong}), we deduce from Lemma \ref{lem:delta1lkbur}
that $\mathfrak{LK}$ is very strong polynomial of degree two.
\end{proof}
\end{prop}

\section{The Long-Moody functor applied to polynomial functors\label{sec:The-Long-Moody-functoreffect}}

Let us move on to the effect of the Long-Moody functors on (very)
strong polynomial functors. For this purpose, it is enough by Remark
\ref{exa:suffit1} to consider the cokernel of the map $i_{1}\mathbf{LM}$.
First, we decompose the functor $\tau_{1}\circ\mathbf{LM}$ (see Proposition
\ref{prop:splittingtranslation}) so as to understand the behaviour
of the image of $i_{1}\mathbf{LM}$ through this decomposition. This
allows us to prove a splitting decomposition of the difference functor
(see Theorem \ref{thm:Splitting LM}). This is the key point to prove
our main results, namely Corollary \ref{cor:Main result} and Theorem
\ref{thm:Main result2}. Finally, we give some additional properties
of Long-Moody functors with respect to polynomial functors.

Let $\left\{ \varsigma_{n}:\mathbf{F}_{n}\hookrightarrow\mathbf{B}_{n+1}\right\} _{n\in\mathbb{N}}$
and $\left\{ a_{n}:\mathbf{B}_{n}\rightarrow Aut\left(\mathbf{F}_{n}\right)\right\} _{n\in\mathbb{N}}$
be coherent families of morphisms (see Definition \ref{def:coherentmor}),
with associated Long-Moody functor $\mathbf{LM}_{a,\varsigma}$ (see
Theorem \ref{Thm:LMFunctor}), which we fix for all the work of this
section (in particular, we omit the $"a,\varsigma"$ from the notation).

\subsection{Decomposition of the translation functor}

We introduce two functors which will play a key role in the main result.
First, let us recall the following crucial property of the augmentation
ideal of a free product of groups, which follows by combining \cite[Lemma 4.3]{cohencohomo}
and \cite[Theorem 4.7]{cohencohomo}.
\begin{prop}
\label{prop:Splitting ideal}Let $G$ and $H$ be groups. Then, there
is a natural $\mathbb{K}\left[G\ast H\right]$-module isomorphism:
\[
\mathcal{I}_{\mathbb{K}\left[G\ast H\right]}\cong\left(\mathcal{I}_{\mathbb{K}\left[G\right]}\underset{\mathbb{K}\left[G\right]}{\otimes}\mathbb{K}\left[G\ast H\right]\right)\oplus\left(\mathcal{I}_{\mathbb{K}\left[H\right]}\underset{\mathbb{K}\left[H\right]}{\otimes}\mathbb{K}\left[G\ast H\right]\right).
\]
\end{prop}

\begin{rem}
\label{rem:structureKg*hmodule}In the statement of Proposition \ref{prop:Splitting ideal},
recall that the augmentation ideal $\mathcal{I}_{\mathbb{K}\left[G\right]}$
(respectively $\mathcal{I}_{\mathbb{K}\left[H\right]}$) is a free
right $\mathbb{K}\left[G\right]$-module (respectively $\mathbb{K}\left[H\right]$-module)
by Proposition \ref{prop:-augmentationidealfreemodule}. Moreover,
the group ring $\mathbb{K}\left[G\ast H\right]$ is a left $\mathbb{K}\left[G\right]$-module
(respectively left $\mathbb{K}\left[H\right]$-module) via the morphism
$id_{G}\ast\iota_{H}:G\rightarrow G*H$ (respectively $\iota_{G}\ast id_{H}:H\rightarrow G*H$
).
\end{rem}

\begin{notation}
\label{def:complementgammafn}Let $n$ and $n'$ be natural numbers
such that $n'\geq n$. We consider the morphism $id_{\mathbf{F}_{n}}*\iota_{\mathbf{F}_{n'-n}}:\mathbf{F}_{n}\hookrightarrow\mathbf{F}_{n'}$.
This corresponds to the identification of $\mathbf{F}_{n}$ as the
subgroup of $\mathbf{F}_{n'}$ generated by the $n$ first copies
of $\mathbf{F}_{1}$ in $\mathbf{F}_{n'}$.

In addition, the group morphism $id_{\mathbf{F}_{n}}*\iota_{\mathbf{F}_{n'-n}}:\mathbf{F}_{n}\hookrightarrow\mathbf{F}_{n'}$
canonically induces a $\mathbb{K}$-module morphism $id_{\mathcal{I}_{\mathbb{K}\left[\mathbf{F}_{n}\right]}}*\iota_{\mathcal{I}_{\mathbb{K}\left[\mathbf{F}_{n'-n}\right]}}:\mathcal{I}_{\mathbb{K}\left[\mathbf{F}_{n}\right]}\hookrightarrow\mathcal{I}_{\mathbb{K}\left[\mathbf{F}_{n'}\right]}$.
\end{notation}

For $F$ an object of $\mathbf{Fct}\left(\mathfrak{U}\boldsymbol{\beta},\mathbb{K}\textrm{-}\mathfrak{Mod}\right)$,
we consider the functor $\left(\tau_{1}\circ\mathbf{LM}\right)\left(F\right)$.
For all natural numbers $n$, by Proposition \ref{prop:Splitting ideal},
we have a $\mathbb{K}\left[\mathbf{F}_{1+n}\right]$-module isomorphism:

\begin{eqnarray*}
 &  & \mathcal{I}_{\mathbb{K}\left[\mathbf{F}_{1+n}\right]}\underset{\mathbb{K}\left[\mathbf{F}_{1+n}\right]}{\varotimes}F\left(n+2\right)\\
 &  & \cong\left(\left(\mathcal{I}_{\mathbb{K}\left[\mathbf{F}_{1}\right]}\underset{\mathbb{K}\left[\mathbf{F}_{1}\right]}{\varotimes}\mathbb{K}\left[\mathbf{F}_{1+n}\right]\right)\oplus\left(\mathcal{I}_{\mathbb{K}\left[\mathbf{F}_{n}\right]}\underset{\mathbb{K}\left[\mathbf{F}_{n}\right]}{\varotimes}\mathbb{K}\left[\mathbf{F}_{1+n}\right]\right)\right)\underset{\mathbb{K}\left[\mathbf{F}_{1+n}\right]}{\varotimes}F\left(n+2\right).
\end{eqnarray*}
Now, by Remark \ref{rem:structureKg*hmodule}, the $\mathbb{K}\left[\mathbf{F}_{n+1}\right]$-module
$F\left(n+2\right)$ is a $\mathbb{K}\left[\mathbf{F}_{1}\right]$-module
via 
\[
F\left(\varsigma_{1+n}\left(id_{\mathbf{F}_{1}}\ast\iota_{\mathbf{F}_{n}}\right)\right):\mathbf{F}_{1}\rightarrow Aut_{\mathbb{K}\textrm{-}\mathfrak{Mod}}\left(F\left(n+2\right)\right)
\]
 and $\mathbb{K}\left[\mathbf{F}_{n}\right]$-module via 
\[
F\left(\varsigma_{1+n}\left(\iota_{\mathbf{F}_{1}}\ast id_{\mathbf{F}_{n}}\right)\right):\mathbf{F}_{n}\rightarrow Aut_{\mathbb{K}\textrm{-}\mathfrak{Mod}}\left(F\left(n+2\right)\right).
\]
Therefore, because of the distributivity of tensor product with respect
to the direct sum, we have the following proposition.
\begin{prop}
\label{prop:splittingtaul1}Let $F\in Obj\left(\mathbf{Fct}\left(\mathfrak{U}\boldsymbol{\beta},\mathbb{K}\textrm{-}\mathfrak{Mod}\right)\right)$
and $n$ be a natural number. Then, we have the following $\mathbb{K}$-module
isomorphism:
\begin{eqnarray}
\tau_{1}\mathbf{LM}\left(F\right)\left(n\right) & \cong & \left(\mathcal{I}_{\mathbb{K}\left[\mathbf{F}_{1}\right]}\underset{\mathbb{K}\left[\mathbf{F}_{1}\right]}{\varotimes}F\left(n+2\right)\right)\oplus\left(\mathcal{I}_{\mathbb{K}\left[\mathbf{F}_{n}\right]}\underset{\mathbb{K}\left[\mathbf{F}_{n}\right]}{\varotimes}F\left(n+2\right)\right).\label{eq:(*)}
\end{eqnarray}
\end{prop}

\begin{defn}
\label{def:morphcaract1}For all natural numbers $n$ and $F\in Obj\left(\mathbf{Fct}\left(\mathfrak{U}\boldsymbol{\beta},\mathbb{K}\textrm{-}\mathfrak{Mod}\right)\right)$,
we denote by

\begin{itemize}
\item $\upsilon\left(F\right)_{n}$ the monomorphism of $\mathbb{K}$-modules
$\left(id_{\mathcal{I}_{\mathbb{K}\left[\mathbf{F}_{1}\right]}}\ast\iota_{\mathcal{I}_{\mathbb{K}\left[\mathbf{F}_{n}\right]}}\right)\underset{\mathbb{K}\left[\mathbf{F}_{1+n}\right]}{\varotimes}id_{F\left(n+2\right)}:\mathcal{I}_{\mathbb{K}\left[\mathbf{F}_{1}\right]}\underset{\mathbb{K}\left[\mathbf{F}_{1}\right]}{\varotimes}F\left(n+2\right)\hookrightarrow\tau_{1}\mathbf{LM}\left(F\right)\left(n\right)$,
\item $\xi\left(F\right)_{n}$ the monomorphism of $\mathbb{K}$-modules
$\left(\iota_{\mathcal{I}_{\mathbb{K}\left[\mathbf{F}_{1}\right]}}*id_{\mathcal{I}_{\mathbb{K}\left[\mathbf{F}_{n}\right]}}\right)\underset{\mathbb{K}\left[\mathbf{F}_{1+n}\right]}{\varotimes}id_{F\left(n+2\right)}:\mathcal{I}_{\mathbb{K}\left[\mathbf{F}_{n}\right]}\underset{\mathbb{K}\left[\mathbf{F}_{n}\right]}{\varotimes}F\left(n+2\right)\hookrightarrow\tau_{1}\mathbf{LM}\left(F\right)\left(n\right)$,
\end{itemize}
associated with the direct sum of Proposition \ref{prop:splittingtaul1}.
\end{defn}

The aim of this section is in fact to show that this $\mathbb{K}$-module
decomposition leads to a decomposition of $\tau_{1}\mathbf{LM}$ (see
Theorem \ref{thm:Splitting LM}) as a functor.

\subsubsection{Additional conditions\label{subsec:Additional-conditions}}

We need two additional conditions so as to make the decomposition
of Proposition \ref{prop:splittingtaul1} functorial. First, we require
the morphisms $\left\{ a_{n}:\mathbf{B}_{n}\rightarrow Aut\left(\mathbf{F}_{n}\right)\right\} _{n\in\mathbb{N}}$
to satisfy the following property.
\begin{condition}
\label{cond:coherenceconditionbnautfn2}Let $n$ and $n'$ be natural
numbers such that $n'\geq n$. We require $a_{1+n'}\left(\left(b_{1,n'-n}^{\boldsymbol{\beta}}\right)^{-1}\natural id_{n}\right)\circ\left(\iota_{\mathbf{F}_{n'-n}}*id_{\mathbf{F}_{n+1}}\right)\circ\left(id_{\mathbf{F}_{1}}\ast\iota_{\mathbf{F}_{n}}\right)=id_{\mathbf{F}_{1}}\ast\iota_{\mathbf{F}_{n'}}$.
In other words, the following diagram is commutative:

\[
\xymatrix{\mathbf{F}_{1}\ar@{->}[d]_{id_{\mathbf{F}_{1}}*\iota_{\mathbf{F}_{n}}}\ar@{->}[rrrr]^{id_{\mathbf{F}_{1}}\ast\iota_{\mathbf{F}_{n'}}} &  &  &  & \mathbf{F}_{1+n'}\\
\mathbf{F}_{1+n}\ar@{->}[rrrr]_{\iota_{\mathbf{F}_{n'-n}}*id_{\mathbf{F}_{1+n}}} &  &  &  & \mathbf{F}_{n'-n}*\mathbf{F}_{1+n}\cong\mathbf{F}_{1+n'}.\ar@{->}[u]_{a_{1+n'}\left(\left(b_{1,n'-n}^{\boldsymbol{\beta}}\right)^{-1}\natural id_{n}\right)}
}
\]
 
\end{condition}

\begin{rem}
Condition \ref{cond:coherenceconditionbnautfn2} will be used to define
an intermediary functor (see Proposition \ref{prop:defvarepsi}).
\end{rem}

In addition, we will assume that the morphisms $\left\{ a_{n}:\mathbf{B}_{n}\rightarrow Aut\left(\mathbf{F}_{n}\right)\right\} _{n\in\mathbb{N}}$
satisfy the following condition.
\begin{condition}
\label{cond:coherenceconditionbnautfn3}Let $n$ and $n'$ be natural
numbers such that $n'\geq n$. We require $a_{n'}\left(id_{n'-n}\natural-\right):\mathbf{B}_{n}\rightarrow Aut\left(\mathbf{F}_{n'}\right)$
maps to the stabilizer of the homomorphism $id_{\mathbf{F}_{n'-n}}*\iota_{\mathbf{F}_{n}}:\mathbf{F}_{n'-n}\longrightarrow\mathbf{F}_{n'}$,
ie for all element $\sigma$ of $\mathbf{B}_{n}$ the following diagram
is commutative:

\[
\xymatrix{\mathbf{F}_{n'-n}\ar@{->}[rr]^{id_{\mathbf{F}_{n'-n}}*\iota_{\mathbf{F}_{n}}}\ar@{->}[dr]_{id_{\mathbf{F}_{n'-n}}*\iota_{\mathbf{F}_{n}}} &  & \mathbf{F}_{n'}\\
 & \mathbf{F}_{n'}.\ar@{->}[ur]_{a_{n'}\left(id_{n'-n}\natural\sigma\right)}
}
\]
\end{condition}

\begin{rem}
Condition \ref{cond:coherenceconditionbnautfn3} will be used in the
proof of Propositions \ref{prop:defvarepsi} and \ref{prop:subfuncotorvarepsi}.
\end{rem}

\begin{rem}
The relations of Conditions \ref{cond:coherenceconditionbnautfn2}
and \ref{cond:coherenceconditionbnautfn3} remain true mutatis mutandis,
for all natural numbers $n$, considering the induced morphisms $a_{n}:\mathbf{B}_{n}\rightarrow Aut\left(\mathcal{I}_{\mathbb{K}\left[\mathbf{F}_{n}\right]}\right)$
and $id_{\mathcal{I}_{\mathbb{K}\left[\mathbf{F}_{n}\right]}}*\iota_{\mathcal{I}_{\mathbb{K}\left[\mathbf{F}_{n'-n}\right]}}:\mathcal{I}_{\mathbb{K}\left[\mathbf{F}_{n}\right]}\hookrightarrow\mathcal{I}_{\mathbb{K}\left[\mathbf{F}_{n'}\right]}$.
\end{rem}

\begin{defn}
If the morphisms $\left\{ a_{n}:\mathbf{B}_{n}\rightarrow Aut\left(\mathbf{F}_{n}\right)\right\} _{n\in\mathbb{N}}$
also satisfy conditions \ref{cond:coherenceconditionbnautfn2} and
\ref{cond:coherenceconditionbnautfn3}, the coherent families of morphisms
$\left\{ \varsigma_{n}:\mathbf{F}_{n}\hookrightarrow\mathbf{B}_{n+1}\right\} _{n\in\mathbb{N}}$
and $\left\{ a_{n}:\mathbf{B}_{n}\rightarrow Aut\left(\mathbf{F}_{n}\right)\right\} _{n\in\mathbb{N}}$
are said to be reliable.
\end{defn}

\begin{prop}
The coherent families of morphisms $\left\{ a_{n,1}:\mathbf{B}_{n}\rightarrow Aut\left(\mathbf{F}_{n}\right)\right\} _{n\in\mathbb{N}}$
and $\left\{ \varsigma_{n,1}:\mathbf{F}_{n}\hookrightarrow\mathbf{B}_{n+1}\right\} _{n\in\mathbb{N}}$
of Examples \ref{ex:ex1defmapfnbn} and \ref{exan1} are reliable.
\end{prop}

\begin{proof}
Recall from Definition \ref{exa:defprodmonUbeta} that $\left(b_{1,n'-n}^{\boldsymbol{\beta}}\right)^{-1}=\sigma_{1}^{-1}\circ\sigma_{2}^{-1}\circ\cdots\circ\sigma_{n'-n}^{-1}$.
We consider the element $e_{\mathbf{F}_{n'-n}}*g_{1}*e_{\mathbf{F}_{n}}=g_{n'-n+1}\in\mathbf{F}_{\left(n'-n\right)+1+n}$.
The definition of $a_{n,1}$ gives that $a_{1+n',1}\left(\sigma_{n'-n}\right)\left(g_{n'-n}\right)=g_{n'-n+1}$.
Therefore, we have that:
\[
a_{1+n',1}\left(\sigma_{n'-n}^{-1}\right)\left(g_{n'-n+1}\right)=g_{n'-n}.
\]
Iterating this observation, we deduce that $a_{1+n'}\left(\left(b_{1,n'-n}^{\boldsymbol{\beta}}\right)^{-1}\natural id_{n}\right)\left(g_{n'-n+1}\right)=g_{1}\in\mathbf{F}_{1+n'}$.
Hence, the family of morphisms $\left\{ a_{n,1}:\mathbf{B}_{n}\rightarrow Aut\left(\mathbf{F}_{n}\right)\right\} _{n\in\mathbb{N}}$
satisfies Condition \ref{cond:coherenceconditionbnautfn2}.

Similarly to Example \ref{exan1} earlier, for all $g\in\mathbf{F}_{n'-n}$
and each Artin generator $\sigma_{i}\in\mathbf{B}_{n}$, $a_{n'}\left(id_{n'-n}\natural\sigma_{i}\right)\left(g*e_{\mathbf{F}_{n}}\right)=g*e_{\mathbf{F}_{n}}$.
Hence, the family of morphisms $\left\{ a_{n,1}:\mathbf{B}_{n}\rightarrow Aut\left(\mathbf{F}_{n}\right)\right\} _{n\in\mathbb{N}}$
satisfies Condition \ref{cond:coherenceconditionbnautfn3}.
\end{proof}
\textbf{From now until the end of Section \ref{sec:The-Long-Moody-functoreffect},
we fix coherent reliable families of morphisms $\left\{ \varsigma_{n}:\mathbf{F}_{n}\hookrightarrow\mathbf{B}_{n+1}\right\} _{n\in\mathbb{N}}$
and $\left\{ a_{n}:\mathbf{B}_{n}\rightarrow Aut\left(\mathbf{F}_{n}\right)\right\} _{n\in\mathbb{N}}$.}

\subsubsection{The intermediary functors}

\paragraph{The functor $\tau_{2}$:}

Let us consider the factor $\mathcal{I}_{\mathbb{K}\left[\mathbf{F}_{1}\right]}\underset{\mathbb{K}\left[\mathbf{F}_{1}\right]}{\varotimes}F\left(n+2\right)$
of $\tau_{1}\mathbf{LM}\left(F\right)\left(n\right)$ in the decomposition
of Proposition \ref{prop:splittingtaul1}.
\begin{notation}
For all objects $F$ of $\mathbf{Fct}\left(\mathfrak{U}\boldsymbol{\beta},\mathbb{K}\textrm{-}\mathfrak{Mod}\right)$,
for all natural numbers $n$, we denote $\mathcal{I}_{\mathbb{K}\left[\mathbf{F}_{1}\right]}\underset{\mathbb{K}\left[\mathbf{F}_{1}\right]}{\varotimes}F\left(n+2\right)$
by $\varUpsilon\left(F\right)\left(n\right)$.
\end{notation}

Recall the monomorphisms $\left\{ \upsilon\left(F\right)_{n}:\varUpsilon\left(F\right)\left(n\right)\hookrightarrow\tau_{1}\mathbf{LM}\left(F\right)\left(n\right)\right\} _{n\in\mathbb{N}}$
of Definition \ref{def:morphcaract1}.
\begin{prop}
\label{prop:defvarepsi}Let $F$ be an object of $\mathbf{Fct}\left(\mathfrak{U}\boldsymbol{\beta},\mathbb{K}\textrm{-}\mathfrak{Mod}\right)$.
For all natural numbers $n$ and $n'$ such that $n'\geq n$, and
for all $\left[n'-n,\sigma\right]\in Hom_{\mathfrak{U}\boldsymbol{\beta}}\left(n,n'\right)$,
assign: 
\[
\varUpsilon\left(F\right)\left(\left[n'-n,\sigma\right]\right)=id_{\mathcal{I}_{\mathbb{K}\left[\mathbf{F}_{1}\right]}}\underset{\mathbb{K}\left[\mathbf{F}_{1}\right]}{\varotimes}F\left(id_{2}\natural\left[n'-n,\sigma\right]\right).
\]
This defines a subfunctor $\varUpsilon\left(F\right):\mathfrak{U}\boldsymbol{\beta}\rightarrow\mathbb{K}\textrm{-}\mathfrak{Mod}$
of $\tau_{1}\mathbf{LM}\left(F\right)$, using the monomorphisms $\left\{ \upsilon\left(F\right)_{n}\right\} _{n\in\mathbb{N}}$.
\end{prop}

\begin{proof}
Let us check that the assignment $\varUpsilon\left(F\right)$ is well
defined with respect to the tensor product. Let $n$ and $n'$ be
natural numbers such that $n'\geq n$, and $\left[n'-n,\sigma\right]\in Hom_{\mathfrak{U}\boldsymbol{\beta}}\left(n,n'\right)$
with $\sigma\in\mathbf{B}_{n'}$. Recall from Proposition \ref{prop:homogenousprebraided}
that $id_{2}\natural\left[n'-n,\sigma\right]=\left[n'-n,\left(id_{2}\natural\sigma\right)\circ\left(\left(b_{2,n'-n}^{\boldsymbol{\beta}}\right)^{-1}\natural id_{n}\right)\right]$.
On the one hand, by Condition \ref{cond:coherenceconditionsigmanan},
we have: 
\[
\left(id_{2}\natural\sigma\right)\circ\varsigma_{1+n'}\left(g_{1}\right)=\varsigma_{1+n'}\left(a_{1+n'}\left(id_{1}\natural\sigma\right)\left(g_{1}\right)\right)\circ\left(id_{2}\natural\sigma\right).
\]
Hence, it follows from Condition \ref{cond:coherenceconditionbnautfn3}
that 
\begin{equation}
\left(id_{2}\natural\sigma\right)\circ\varsigma_{1+n'}\left(g_{1}\right)=\varsigma_{1+n'}\left(g_{1}\right)\circ\left(id_{2}\natural\sigma\right).\label{eq:etape1}
\end{equation}
On the other hand, Condition \ref{cond:coherenceconditionbnautfn2}
gives that 
\[
g_{1}=a_{2+n'}\left(\left(b_{1,n'-n}^{\boldsymbol{\beta}}\right)^{-1}\natural id_{n+1}\right)\left(g_{n'-n+1}\right)
\]
and by Condition \ref{cond:coherenceconditionbnautfn3} we have 
\[
g_{1}=a_{2+n'}\left(id_{1}\natural\left(b_{1,n'-n}^{\boldsymbol{\beta}}\right)^{-1}\natural id_{n}\right)\left(g_{1}\right).
\]
By the definition of the braiding $b_{-,-}^{\boldsymbol{\beta}}$
(see Definition \ref{exa:defprodmonUbeta}), we deduce that:
\begin{eqnarray*}
\varsigma_{1+n'}\left(g_{1}\right) & = & \varsigma_{1+n'}\left(a_{2+n'}\left(\left(b_{2,n'-n}^{\boldsymbol{\beta}}\right)^{-1}\natural id_{n}\right)\left(g_{n'-n+1}\right)\right).
\end{eqnarray*}
Then, it follows from the combination of Conditions \ref{cond:conditionstability}
and \ref{cond:coherenceconditionsigmanan} that as morphisms in $\mathfrak{U}\boldsymbol{\beta}$:
\begin{align}
 & \left[n'-n,\varsigma_{1+n'}\left(g_{1}\right)\circ\left(\left(b_{2,n'-n}^{\boldsymbol{\beta}}\right)^{-1}\natural id_{n}\right)\right]\nonumber \\
= & \left[n'-n,\left(\left(b_{2,n'-n}^{\boldsymbol{\beta}}\right)^{-1}\natural id_{n}\right)\circ\left(id_{n'-n}\natural\varsigma_{1+n}\left(g_{1}\right)\right)\right].\label{eq:etape2}
\end{align}
Hence, we deduce from the relations (\ref{eq:etape1}) and (\ref{eq:etape2})
that:
\begin{align*}
 & \left[n'-n,\left(\left(id_{2}\natural\sigma\right)\circ\left(\left(b_{2,n'-n}^{\boldsymbol{\beta}}\right)^{-1}\natural id_{n}\right)\right)\circ\left(id_{n'-n}\natural\varsigma_{1+n}\left(g_{1}\right)\right)\right]\\
= & \left[n'-n,\varsigma_{1+n'}\left(g_{1}\right)\circ\left(\left(id_{2}\natural\sigma\right)\circ\left(\left(b_{2,n'-n}^{\boldsymbol{\beta}}\right)^{-1}\natural id_{n}\right)\right)\right].
\end{align*}
A fortiori, $F\left(id_{2}\natural\left[n'-n,\sigma\right]\right)\circ F\left(\varsigma_{1+n}\left(g_{1}\right)\right)=F\left(\varsigma_{1+n'}\left(g_{1}\right)\right)\circ F\left(id_{2}\natural\left[n'-n,\sigma\right]\right)$.
Hence, our assignment is well defined with respect to the tensor product.

Let us prove that the subspaces $\varUpsilon\left(F\right)\left(n\right)$
are stable under the action of $\mathfrak{U}\boldsymbol{\beta}$.
Let $i\in\mathcal{I}_{\mathbb{K}\left[\mathbf{F}_{1}\right]}$ and
$v\in F\left(n+2\right)$. We deduce from the definition of the monoidal
structure morphisms of $\mathfrak{U}\boldsymbol{\beta}$ (see Proposition
\ref{prop:homogenousprebraided}) and from the definition of the Long-Moody
functor (see Theorem \ref{Thm:LMFunctor}) that, for all $i\in\mathcal{I}_{\mathbb{K}\left[\mathbf{F}_{1}\right]}$
and for all $v\in F\left(n+2\right)$:
\begin{align*}
 & \left(\left(\tau_{1}\mathbf{LM}\left(F\right)\left(\left[n'-n,\sigma\right]\right)\right)\circ\upsilon\left(F\right)_{n}\right)\left(i\underset{\mathbb{K}\left[\mathbf{F}_{1}\right]}{\varotimes}v\right)\\
= & a_{1+n'}\left(id_{1}\natural\sigma\right)\left(a_{1+n'}\left(\left(b_{1,n'-n}^{\mathfrak{\boldsymbol{\beta}}}\right)^{-1}\natural id_{n}\right)\left(\iota_{\mathcal{I}_{\mathbb{K}\left[\mathbf{F}_{n'-n}\right]}}*id_{\mathcal{I}_{\mathbb{K}\left[\mathbf{F}_{1}\right]}}*\iota_{\mathcal{I}_{\mathbb{K}\left[\mathbf{F}_{n}\right]}}\right)\left(i\right)\right)\\
 & \underset{\mathbb{K}\left[\mathbf{F}_{n'+1}\right]}{\varotimes}F\left(id_{1}\natural id_{1}\natural\left[n'-n,\sigma\right]\right)\left(v\right).
\end{align*}
It follows from Condition \ref{cond:coherenceconditionbnautfn2} that:
\[
a_{1+n'}\left(\left(b_{1,n'-n}^{\mathfrak{\boldsymbol{\beta}}}\right)^{-1}\natural id_{n}\right)\left(\iota_{\mathcal{I}_{\mathbb{K}\left[\mathbf{F}_{n'-n}\right]}}*id_{\mathcal{I}_{\mathbb{K}\left[\mathbf{F}_{1}\right]}}*\iota_{\mathcal{I}_{\mathbb{K}\left[\mathbf{F}_{n}\right]}}\right)\left(i\right)=\left(id_{\mathcal{I}_{\mathbb{K}\left[\mathbf{F}_{1}\right]}}*\iota_{\mathcal{I}_{\mathbb{K}\left[\mathbf{F}_{n'}\right]}}\right)\left(i\right).
\]
Since by Condition \ref{cond:coherenceconditionbnautfn3}, $a_{1+n'}\left(id_{1}\natural\sigma\right)\left(id_{\mathcal{I}_{\mathbb{K}\left[\mathbf{F}_{1}\right]}}*\iota_{\mathcal{I}_{\mathbb{K}\left[\mathbf{F}_{n'}\right]}}\right)\left(i\right)=\left(id_{\mathcal{I}_{\mathbb{K}\left[\mathbf{F}_{1}\right]}}*\iota_{\mathcal{I}_{\mathbb{K}\left[\mathbf{F}_{n'}\right]}}\right)\left(i\right)$
for all elements $\sigma$ of $\mathbf{B}_{n'}$, we deduce that:
\begin{eqnarray*}
\left(\tau_{1}\mathbf{LM}\left(F\right)\left(\left[n'-n,\sigma\right]\right)\circ\upsilon\left(F\right)_{n}\right)\left(i\underset{\mathbb{K}\left[\mathbf{F}_{1}\right]}{\varotimes}v\right) & = & \left(\upsilon\left(F\right)_{n'}\circ\varUpsilon\left(F\right)\left(\left[n'-n,\sigma\right]\right)\right)\left(i\underset{\mathbb{K}\left[\mathbf{F}_{m}\right]}{\varotimes}v\right).
\end{eqnarray*}
Therefore, the functorial structure of $\tau_{1}\mathbf{LM}\left(F\right)$
induces by restriction the one of $\varUpsilon\left(F\right)$.
\end{proof}
Now, we can lift this link between \textit{$\varUpsilon\left(F\right)$
of $\tau_{1}\mathbf{LM}\left(F\right)$} to endofunctors of $\mathbf{Fct}\left(\mathfrak{U}\boldsymbol{\beta},\mathbb{K}\textrm{-}\mathfrak{Mod}\right)$.
\begin{prop}
\label{prop:subfuncotorvarepsi}Let $F$ and $G$ be two objects of
$\mathbf{Fct}\left(\mathfrak{U}\boldsymbol{\beta},\mathbb{K}\textrm{-}\mathfrak{Mod}\right)$,
and $\eta:F\rightarrow G$ be a natural transformation. For all natural
numbers $n$, assign :
\[
\left(\varUpsilon\left(\eta\right)\right)_{n}=id_{\mathcal{I}_{\mathbb{K}\left[\mathbf{F}_{1}\right]}}\underset{\mathbb{K}\left[\mathbf{F}_{1}\right]}{\varotimes}\eta_{n+2}.
\]
Then we define a subfunctor $\varUpsilon:\mathbf{Fct}\left(\mathfrak{U}\boldsymbol{\beta},\mathbb{K}\textrm{-}\mathfrak{Mod}\right)\rightarrow\mathbf{Fct}\left(\mathfrak{U}\boldsymbol{\beta},\mathbb{K}\textrm{-}\mathfrak{Mod}\right)$
of $\tau_{1}\mathbf{LM}$ using the monomorphisms $\left\{ \upsilon\left(F\right)_{n}\right\} _{n\in\mathbb{N}}$. 
\end{prop}

\begin{proof}
The consistency of our definition follows repeating mutatis mutandis
point (\ref{enu:It-remains-to}) of the proof of Theorem \ref{Thm:LMFunctor}.
It directly follows from the definitions of $\left(\varUpsilon\left(\eta\right)\right)_{n}$,
$\upsilon\left(G\right)_{n}$ and $\tau_{1}\circ\mathbf{LM}$ (see
Definition \ref{subsec:The-Long-Moody-functors}) that $\upsilon\left(G\right)_{n}\circ\left(\varUpsilon\right)\left(\eta\right)_{n}=\left(\tau_{1}\circ\mathbf{LM}\right)\left(\eta\right)_{n}\circ\upsilon\left(F\right)_{n}.$
\end{proof}
In fact, we have an easy description of the functor $\varUpsilon$.
\begin{prop}
\label{prop:equivalencevarepsitau2}There is a natural equivalence
$\varUpsilon\cong\tau_{2}$ where $\tau_{2}$ is the translation functor
introduced in Definition \ref{def:deftaux}.
\end{prop}

\begin{proof}
Let $F$ be an object of $\mathbf{Fct}\left(\mathfrak{U}\boldsymbol{\beta},\mathbb{K}\textrm{-}\mathfrak{Mod}\right)$.
By Proposition \ref{prop:-augmentationidealfreemodule}, for all natural
numbers $n$, we have an isomorphism:
\begin{eqnarray*}
\chi_{n,F}:\mathcal{I}_{\mathbb{K}\left[\mathbf{F}_{1}\right]}\underset{\mathbb{K}\left[\mathbf{F}_{1}\right]}{\varotimes}F\left(n+2\right) & \overset{\cong}{\longrightarrow} & F\left(n+2\right).\\
\left(g_{1}-1\right)\underset{\mathbb{K}\left[\mathbf{F}_{n}\right]}{\varotimes}v & \longmapsto & v
\end{eqnarray*}
It follows from Definition \ref{def:deftaux} and Proposition \ref{prop:defvarepsi}
that the isomorphisms $\left\{ \chi_{n,F}\right\} _{n\in\mathbb{N}}$
define the desired natural equivalence $\varUpsilon\overset{\chi}{\rightarrow}\tau_{2}$.
\end{proof}

\paragraph{The functor $\mathbf{LM}\circ\tau_{1}$:}

Now, let us consider the part $\mathcal{I}_{\mathbb{K}\left[\mathbf{F}_{n}\right]}\underset{\mathbb{K}\left[\mathbf{F}_{n}\right]}{\varotimes}F\left(n+2\right)$
of $\tau_{1}\circ\mathbf{LM}\left(F\right)\left(n\right)$ in the
decomposition of Proposition \ref{prop:splittingtaul1}. In fact,
we are going to prove that these modules assemble to form a functor
which identifies with $\mathbf{LM}\left(\tau_{1}F\right)$. We recall
from Theorem \ref{Thm:LMFunctor} and Definition \ref{def:deftaux}
the following fact.
\begin{rem}
\label{def:lmtau1}The functor $\mathbf{LM}\circ\tau_{1}:\mathbf{Fct}\left(\mathfrak{U}\boldsymbol{\beta},\mathbb{K}\textrm{-}\mathfrak{Mod}\right)\rightarrow\mathbf{Fct}\left(\mathfrak{U}\boldsymbol{\beta},\mathbb{K}\textrm{-}\mathfrak{Mod}\right)$
is defined by:

\begin{itemize}
\item for $F\in Obj\left(\mathbf{Fct}\left(\mathfrak{U}\boldsymbol{\beta},\mathbb{K}\textrm{-}\mathfrak{Mod}\right)\right)$,
$\forall n\in\mathbb{N}$, $\left(\mathbf{LM}\circ\tau_{1}\right)\left(F\right)\left(n\right)=\mathcal{I}_{\mathbb{K}\left[\mathbf{F}_{n}\right]}\underset{\mathbb{K}\left[\mathbf{F}_{n}\right]}{\varotimes}F\left(n+2\right)$,
where $F\left(n+2\right)$ is a left $\mathbb{K}\left[\mathbf{F}_{n}\right]$-module
using $F\left(id_{1}\natural\varsigma_{n}\left(-\right)\right):\mathbf{F}_{n}\rightarrow Aut_{\mathbb{K}\textrm{-}\mathfrak{Mod}}\left(F\left(n+2\right)\right)$.
For $n,n'\in\mathbb{N}$, such that $n'\geq n$, and $\left[n'-n,\sigma\right]\in Hom_{\mathfrak{U}\boldsymbol{\beta}}\left(n,n'\right)$:
\[
\left(\mathbf{LM}\circ\tau_{1}\right)\left(F\right)\left(\left[n'-n,\sigma\right]\right)=a_{n'}\left(\sigma\right)\left(\iota_{\mathcal{I}_{\mathbb{K}\left[\mathbf{F}_{n'-n}\right]}}*id_{\mathcal{I}_{\mathbb{K}\left[\mathbf{F}_{n}\right]}}\right)\underset{\mathbb{K}\left[\mathbf{F}_{n'}\right]}{\varotimes}F\left(id_{1}\natural id_{1}\natural\left[n'-n,\sigma\right]\right).
\]
\item Morphisms: let $F$ and $G$ be two objects of $\mathbf{Fct}\left(\mathfrak{U}\boldsymbol{\beta},\mathbb{K}\textrm{-}\mathfrak{Mod}\right)$,
and $\eta:F\rightarrow G$ be a natural transformation. The natural
transformation $\left(\mathbf{LM}\circ\tau_{1}\right)\left(\eta\right):\left(\mathbf{LM}\circ\tau_{1}\right)\left(F\right)\rightarrow\left(\mathbf{LM}\circ\tau_{1}\right)\left(G\right)$
for all natural numbers $n$ is given by:
\[
\left(\left(\mathbf{LM}\circ\tau_{1}\right)\left(\eta\right)\right)_{n}=id_{\mathcal{I}_{\mathbb{K}\left[\mathbf{F}_{n}\right]}}\underset{\mathbb{K}\left[\mathbf{F}_{n}\right]}{\varotimes}\eta_{n+2}.
\]
\end{itemize}
\end{rem}

\begin{prop}
\label{prop:subfunclmtm}For all $F\in Obj\left(\mathbf{Fct}\left(\mathfrak{U}\boldsymbol{\beta},\mathbb{K}\textrm{-}\mathfrak{Mod}\right)\right)$,
the monomorphisms $\left\{ \xi\left(F\right)_{n}\right\} _{n\in\mathbb{N}}$
(see Definition \ref{def:morphcaract1}) allow to define a natural
transformation $\xi'\left(F\right):\left(\mathbf{LM}\circ\tau_{1}\right)\left(F\right)\rightarrow\left(\tau_{1}\circ\mathbf{LM}\right)\left(F\right)$
where, for all natural numbers $n$: 
\[
\xi'\left(F\right)_{n}=\left(\iota_{\mathcal{I}_{\mathbb{K}\left[\mathbf{F}_{1}\right]}}*id_{\mathcal{I}_{\mathbb{K}\left[\mathbf{F}_{n}\right]}}\right)\underset{\mathbb{K}\left[\mathbf{F}_{1+n}\right]}{\varotimes}F\left(\left(b_{1,1}^{\mathfrak{\boldsymbol{\beta}}}\right)^{-1}\natural id_{n}\right).
\]
This yields a natural transformation $\xi':\mathbf{LM}\circ\tau_{1}\rightarrow\tau_{1}\circ\mathbf{LM}$.
\end{prop}

\begin{proof}
Let $n$ and $n'$ be natural numbers such that $n'\geq n$, and $\left[n'-n,\sigma\right]\in Hom_{\mathfrak{U}\boldsymbol{\beta}}\left(n,n'\right)$
with $\sigma\in\mathbf{B}_{n'}$. Let $i\in\mathcal{I}_{\mathbb{K}\left[\mathbf{F}_{n}\right]}$,
$v\in F\left(n+2\right)$ and $g\in\mathbf{F}_{n}$. By Condition
\ref{cond:conditionstability} (using Lemma \ref{rem:noteworthyconseq}
with $n'=n+1$) the following equality holds in $\mathbf{B}_{n+2}$:
\[
\left(\left(b_{1,1}^{\boldsymbol{\beta}}\right)^{-1}\natural id_{n}\right)\circ\left(id_{1}\natural\varsigma_{n}\left(g\right)\right)=\varsigma_{1+n}\left(e_{\mathbf{F}_{1}}*g\right)\circ\left(\left(b_{1,1}^{\boldsymbol{\beta}}\right)^{-1}\natural id_{n}\right).
\]
Recall that $F\left(n+2\right)$ is a $\mathbb{K}\left[\mathbf{F}_{n}\right]$-module
via $F\left(\varsigma_{1+n}\circ\left(\iota_{\mathbf{F}_{1}}\ast id_{\mathbf{F}_{n}}\right)\right)$
and $\tau_{1}F\left(n+1\right)$ is a $\mathbb{K}\left[\mathbf{F}_{n}\right]$-module
via $F\left(id_{1}\natural\left(\varsigma_{n}\circ id_{\mathbf{F}_{n}}\right)\right)$.
Then it follows that the assignment $\xi'\left(F\right)_{n}$ is well-defined
with respect to the tensor product structures of $\left(\mathbf{LM}\circ\tau_{1}\right)\left(F\right)\left(n\right)$
and $\left(\tau_{1}\circ\mathbf{LM}\right)\left(F\right)\left(n\right)$.
Moreover, we compute that:
\begin{eqnarray*}
 &  & \left(\left(\tau_{1}\circ\mathbf{LM}\right)\left(F\right)\left(\left[n'-n,\sigma\right]\right)\right)\circ\left(\xi'\left(F\right)_{n}\right)\left(i\underset{\mathbb{K}\left[\mathbf{F}_{n}\right]}{\varotimes}v\right)\\
 & = & a_{1+n'}\left(id_{1}\natural\sigma\right)\left(a_{1+n'}\left(\left(b_{1,n'-n}^{\mathfrak{\boldsymbol{\beta}}}\right)^{-1}\natural id_{n}\right)\left(\iota_{\mathcal{I}_{\mathbb{K}\left[\mathbf{F}_{1+n'-n}\right]}}*id_{\mathcal{I}_{\mathbb{K}\left[\mathbf{F}_{n}\right]}}\right)\left(i\right)\right)\\
 &  & \underset{\mathbb{K}\left[\mathbf{F}_{n'+1}\right]}{\varotimes}F\left(\left(b_{1,1}^{\mathfrak{\boldsymbol{\beta}}}\right)^{-1}\natural\left[n'-n,\sigma\right]\right)\left(v\right).
\end{eqnarray*}
It follows from Condition \ref{cond:coherenceconditionbnautfn1} that:
\[
a_{1+n'}\left(\left(b_{1,n'-n}^{\mathfrak{\boldsymbol{\beta}}}\right)^{-1}\natural id_{n}\right)\circ\left(\iota_{\mathcal{I}_{\mathbb{K}\left[\mathbf{F}_{1+n'-n}\right]}}*id_{\mathcal{I}_{\mathbb{K}\left[\mathbf{F}_{n}\right]}}\right)\left(i\right)=\left(\iota_{\mathcal{I}_{\mathbb{K}\left[\mathbf{F}_{1+n'-n}\right]}}*id_{\mathcal{I}_{\mathbb{K}\left[\mathbf{F}_{n}\right]}}\right)\left(i\right).
\]
Again by Condition \ref{cond:coherenceconditionbnautfn1}, we deduce
that:
\[
a_{1+n'}\left(id_{1}\natural\sigma\right)\circ\left(\iota_{\mathcal{I}_{\mathbb{K}\left[\mathbf{F}_{1+n'-n}\right]}}*id_{\mathcal{I}_{\mathbb{K}\left[\mathbf{F}_{n}\right]}}\right)\left(i\right)=\iota_{\mathcal{I}_{\mathbb{K}\left[\mathbf{F}_{1}\right]}}*a_{n'}\left(\sigma\right)\left(\iota_{\mathcal{I}_{\mathbb{K}\left[\mathbf{F}_{n'-n}\right]}}*id_{\mathcal{I}_{\mathbb{K}\left[\mathbf{F}_{n}\right]}}\right)\left(i\right).
\]
Hence, we deduce that:
\[
\left(\left(\tau_{1}\circ\mathbf{LM}\right)\left(F\right)\left(\left[n'-n,\sigma\right]\right)\right)\circ\left(\xi'\left(F\right)_{n}\right)=\left(\xi'\left(F\right)_{n'}\right)\circ\left(\left(\mathbf{LM}\circ\tau_{1}\right)\left(F\right)\left(\left[n'-n,\sigma\right]\right)\right).
\]

Let $\eta:F\rightarrow G$ be a natural transformation in the category
$\mathbf{Fct}\left(\mathfrak{U}\boldsymbol{\beta},\mathbb{K}\textrm{-}\mathfrak{Mod}\right)$
and let $n$ be a natural number. Since $\eta$ is a natural transformation,
we have:
\[
G\left(\left(b_{1,1}^{\mathfrak{\boldsymbol{\beta}}}\right)^{-1}\natural id_{n}\right)\circ\eta_{n+2}=\eta_{n+2}\circ F\left(\left(b_{1,1}^{\mathfrak{\boldsymbol{\beta}}}\right)^{-1}\natural id_{n}\right).
\]
Hence, we deduce from the definitions of $\tau_{1}\circ\mathbf{LM}$
(see Theorem \ref{Thm:LMFunctor}) and of $\mathbf{LM}\circ\tau_{1}$
(see Remark \ref{def:lmtau1}) that:
\[
\xi'\left(G\right)_{n}\circ\left(\mathbf{LM}\circ\tau_{1}\right)\left(\eta\right)_{n}=\left(\tau_{1}\circ\mathbf{LM}\right)\left(\eta\right)_{n}\circ\xi'\left(F\right)_{n}.
\]
\end{proof}

\subsubsection{Splitting of the translation functor}

Now, we can establish a decomposition result for the translation functor
applied to a Long-Moody functor.
\begin{prop}
\label{prop:splittingtranslation}There is a natural equivalence of
endofunctors of $\mathbf{Fct}\left(\mathfrak{U}\boldsymbol{\beta},\mathbb{K}\textrm{-}\mathfrak{Mod}\right)$:
\[
\tau_{1}\circ\mathbf{LM}\cong\tau_{2}\oplus\left(\mathbf{LM}\circ\tau_{1}\right).
\]
\end{prop}

\begin{proof}
Recall the natural transformations $\upsilon:\varUpsilon\rightarrow\tau_{1}\circ\mathbf{LM}$
(introduced in Proposition \ref{prop:subfuncotorvarepsi}) and $\xi':\mathbf{LM}\circ\tau_{1}\rightarrow\tau_{1}\circ\mathbf{LM}$
(defined in Proposition \ref{prop:subfunclmtm}). The direct sum in
the category $\mathbf{Fct}\left(\mathfrak{U}\boldsymbol{\beta},\mathbb{K}\textrm{-}\mathfrak{Mod}\right)$
(induced by the direct sum in the category $\mathbb{K}\textrm{-}\mathfrak{Mod}$)
allows us to define a natural transformation:\textit{
\[
\upsilon\oplus\xi':\varUpsilon\oplus\left(\mathbf{LM}\circ\tau_{1}\right)\longrightarrow\left(\tau_{1}\circ\mathbf{LM}\right)\left(F\right).
\]
}This is a natural equivalence since for all natural numbers $n$,
we have an isomorphism of $\mathbb{K}$-modules according to Proposition
\ref{prop:splittingtaul1}: $\varUpsilon\left(F\right)\left(n\right)\oplus\left(\mathbf{LM}\circ\tau_{1}\right)\left(F\right)\left(n\right)\cong\left(\tau_{1}\circ\mathbf{LM}\right)\left(F\right)\left(n\right)$.
We conclude using Proposition \ref{prop:equivalencevarepsitau2}.
\end{proof}

\subsection{Splitting of the difference functor}

Recall the natural transformation $i_{1}:Id_{\mathbf{Fct}\left(\mathfrak{U}\boldsymbol{\beta},\mathbb{K}\textrm{-}\mathfrak{Mod}\right)}\rightarrow\tau_{1}$
of $\mathbf{Fct}\left(\mathfrak{U}\boldsymbol{\beta},\mathbb{K}\textrm{-}\mathfrak{Mod}\right)$.
Our aim is to study the cokernel of $i_{1}\circ\mathbf{LM}$. We recall
that for $F$ an object of $\mathbf{Fct}\left(\mathfrak{U}\boldsymbol{\beta},\mathbb{K}\textrm{-}\mathfrak{Mod}\right)$,
for all natural numbers $n$, $\left(i_{1}\mathbf{LM}\right)\left(F\right)_{n}=\mathbf{LM}\left(F\right)\left(\left[1,id_{1+n}\right]\right)$
(see Definition \ref{def:defix}).
\begin{rem}
\label{rem:evaluationi1lm}Explicitly for all elements $i$ of $\mathcal{I}_{\mathbb{K}\left[\mathbf{F}_{n}\right]}$,
for all elements $v$ of $F\left(n\right)$:
\[
\left(i_{1}\mathbf{LM}\right)\left(F\right)_{n}\left(i\underset{\mathbb{K}\left[\mathbf{F}_{n}\right]}{\varotimes}v\right)=\left(\iota_{\mathcal{I}_{\mathbb{K}\left[\mathbf{F}_{1}\right]}}*id_{\mathcal{I}_{\mathbb{K}\left[\mathbf{F}_{n}\right]}}\right)\left(i\right)\underset{\mathbb{K}\left[\mathbf{F}_{1+n}\right]}{\varotimes}F\left(id_{1}\natural\iota_{1}\natural id_{n}\right)\left(v\right).
\]
\end{rem}

\paragraph{The natural transformation $\mathbf{LM}\circ i_{1}$:}

Let us consider the exact sequence (\ref{eq:ESCaract}) in the category
of endofunctors of $\mathbf{Fct}\left(\mathfrak{U}\boldsymbol{\beta},\mathbb{K}\textrm{-}\mathfrak{Mod}\right)$
of Proposition \ref{prop:lemmecaract}:
\[
\xymatrix{0\ar@{->}[r] & \kappa_{1}\ar@{->}[r]^{\Omega_{1}} & Id\ar@{->}[r]^{i_{1}} & \tau_{1}\ar@{->}[r]^{\varDelta_{1}} & \delta_{1}\ar@{->}[r] & 0}
.
\]
Since the Long-Moody functor is exact (see Proposition \ref{prop:exactnessLM}),
we have the following exact sequence:
\begin{equation}
\xymatrix{0\ar@{->}[r] & \mathbf{LM}\circ\kappa_{1}\ar@{->}[rr]^{{\color{white}ooo}\mathbf{LM}\left(\Omega_{1}\right)} &  & \mathbf{LM}\ar@{->}[rr]^{\mathbf{LM}\left(i_{1}\right)} &  & \mathbf{LM}\circ\tau_{1}\ar@{->}[rr]^{\mathbf{LM}\left(\varDelta_{1}\right)} &  & \mathbf{LM}\circ\delta_{1}\ar@{->}[r] & 0}
.\label{eq:LESLMi1}
\end{equation}

\begin{rem}
\label{rem:evaluationlmi1}From the definition of $\mathbf{LM}$ (see
Theorem \ref{Thm:LMFunctor}), we deduce that for $F$ an object of
$\mathbf{Fct}\left(\mathfrak{U}\boldsymbol{\beta},\mathbb{K}\textrm{-}\mathfrak{Mod}\right)$,
for all natural numbers $n$, for all elements $i$ of $\mathcal{I}_{\mathbb{K}\left[\mathbf{F}_{n}\right]}$,
for all elements $v$ of $F\left(n\right)$:
\[
\mathbf{LM}\left(i_{1}\right)\left(F\right)_{n}\left(i\underset{\mathbb{K}\left[\mathbf{F}_{n}\right]}{\varotimes}v\right)=i\underset{\mathbb{K}\left[\mathbf{F}_{n}\right]}{\varotimes}F\left(\iota_{1}\natural id_{1}\natural id_{n}\right)\left(v\right).
\]
Recall the natural transformation $\xi':\mathbf{LM}\circ\tau_{1}\rightarrow\tau_{1}\circ\mathbf{LM}$
introduced in \ref{prop:subfunclmtm}.
\end{rem}

\begin{lem}
\label{lem:relationLMi1andi1LM}As natural transformations from $\mathbf{LM}$
to $\tau_{1}\circ\mathbf{LM}$, which are endofunctors of the category
$\mathbf{Fct}\left(\mathfrak{U}\boldsymbol{\beta},\mathbb{K}\textrm{-}\mathfrak{Mod}\right)$,
the following equality holds:
\[
\xi'\circ\left(\mathbf{LM}\left(i_{1}\right)\right)=i_{1}\mathbf{LM}.
\]
\end{lem}

\begin{proof}
Let $F$ be an object of $\mathbf{Fct}\left(\mathfrak{U}\boldsymbol{\beta},\mathbb{K}\textrm{-}\mathfrak{Mod}\right)$.
Let $n$ be a natural number. Let $i$ be an element of $\mathcal{I}_{\mathbb{K}\left[\mathbf{F}_{n}\right]}$
and let $v$ be an element of $F\left(n\right)$. Since $\left(b_{1,1}^{\mathfrak{\boldsymbol{\beta}}}\right)^{-1}\circ\left(\iota_{1}\natural id_{1}\right)=id_{1}\natural\iota_{1}$
by Definition \ref{def:defprebraided}, we deduce from Proposition
\ref{prop:subfunclmtm}, Remark \ref{rem:evaluationlmi1} and Remark
\ref{rem:evaluationi1lm}, that:
\[
\left(\xi'\circ\left(\mathbf{LM}\left(i_{1}\right)\right)\right)\left(F\right)_{n}\left(i\underset{\mathbb{K}\left[\mathbf{F}_{n}\right]}{\varotimes}v\right)=\left(id_{1}*i\right)\underset{\mathbb{K}\left[\mathbf{F}_{1+n}\right]}{\varotimes}F\left(id_{1}\natural\iota_{1}\natural id_{n}\right)\left(v\right)=\left(i_{1}\mathbf{LM}\right)\left(F\right)_{n}\left(i\underset{\mathbb{K}\left[\mathbf{F}_{n}\right]}{\varotimes}v\right).
\]
\end{proof}

\paragraph{Decomposition results:}

Lemma \ref{lem:relationLMi1andi1LM} leads to the following key results.
\begin{thm}
\label{thm:Splitting LM}\label{thm:commutationkappaLM}There is a
natural equivalence in the category $\mathbf{Fct}\left(\mathfrak{U}\boldsymbol{\beta},\mathbb{K}\textrm{-}\mathfrak{Mod}\right)$:
\[
\delta_{1}\circ\mathbf{LM}\cong\tau_{2}\oplus\left(\mathbf{LM}\circ\delta_{1}\right).
\]
Moreover, there is a natural isomorphism $\kappa_{1}\circ\mathbf{LM}\cong\mathbf{LM}\circ\kappa_{1}$.
\end{thm}

\begin{proof}
It follows from the definition of $i_{1}$ (see Proposition \ref{prop:lemmecaract})
and from Lemma \ref{lem:relationLMi1andi1LM} that the following diagram
is commutative and the row is an exact sequence:

\[
\xymatrix{0\ar@{->}[r] & \kappa_{1}\circ\mathbf{LM}\ar@{->}[rr]^{{\color{white}ooo}\Omega_{1}\mathbf{LM}} &  & \mathbf{LM}\ar@{->}[rr]^{i_{1}\mathbf{LM}}\ar@{=}[d] &  & \tau_{1}\circ\mathbf{LM}\ar@{->}[rr]^{\varDelta_{1}\mathbf{LM}} &  & \delta_{1}\circ\mathbf{LM}\ar@{->}[r] & 0\\
 &  &  & \mathbf{LM}\ar@{->}[rr]^{\mathbf{LM}\left(i_{1}\right)} &  & \mathbf{LM}\circ\tau_{1}.\ar@{^{(}->}[u]_{\textrm{by Lemma \ref{lem:relationLMi1andi1LM}}}^{\xi'}
}
\]
We denote by $i_{\mathbf{LM}\circ\tau_{1}}^{\oplus}$ the inclusion
morphism $\mathbf{LM}\circ\tau_{1}\hookrightarrow\tau_{2}\oplus\left(\mathbf{LM}\circ\tau_{1}\right)$.
The functor $\mathbf{LM}\circ\kappa_{1}$ is also the kernel of the
natural transformation $i_{\mathbf{LM}\circ\tau_{1}}^{\oplus}\circ\left(\mathbf{LM}\circ i_{1}\right)$,
as the inclusion morphism $i_{\mathbf{LM}\circ\tau_{1}}^{\oplus}:\mathbf{LM}\circ\tau_{1}\hookrightarrow\tau_{2}\oplus\left(\mathbf{LM}\circ\tau_{1}\right)$
is a monomorphism. Then, recalling the exact sequence (\ref{eq:LESLMi1}),
we obtain that the following diagram is commutative and that the two
rows are exact:

\[
\xymatrix{0\ar@{->}[r] & \kappa_{1}\circ\mathbf{LM}\ar@{->}[rr]^{{\color{white}ooo}\Omega_{1}\mathbf{LM}} &  & \mathbf{LM}\ar@{->}[rr]^{i_{1}\mathbf{LM}}\ar@{=}[d] &  & \tau_{1}\circ\mathbf{LM}\ar@{->}[rr]^{\varDelta_{1}\mathbf{LM}} &  & \delta_{1}\circ\mathbf{LM}\ar@{->}[r] & 0\\
0\ar@{->}[r] & \mathbf{LM}\circ\kappa_{1}\ar@{->}[rr]_{{\color{white}ooo}\mathbf{LM}\left(\Omega_{1}\right)} &  & \mathbf{LM}\ar@{->}[rr]_{i_{\mathbf{LM}\circ\tau_{1}}^{\oplus}\circ\left(\mathbf{LM}\left(i_{1}\right)\right)\,\,\,\,\,\,\,\,\,\,\,\,} &  & \tau_{2}\oplus\left(\mathbf{LM}\circ\tau_{1}\right)\ar@{->}[rr]_{id_{\tau_{2}}\oplus\left(\mathbf{LM}\left(\varDelta_{1}\right)\right)}\ar@{->}[u]_{\upsilon\oplus\xi'}^{\textrm{\ensuremath{\cong}\,\,by Proposition \ref{prop:splittingtranslation}}} &  & \tau_{2}\oplus\left(\mathbf{LM}\circ\delta_{1}\right)\ar@{->}[r] & 0
}
\]
A fortiori, by definition of $\delta_{1}$ (see Definition \ref{def:defix})
and the universal property of the cokernel, we deduce that:
\[
\tau_{2}\oplus\left(\mathbf{LM}\circ\delta_{1}\right)\cong\delta_{1}\circ\mathbf{LM}.
\]
Furthermore, by the unicity up to isomorphism of the kernel, we conclude
that $\kappa_{1}\circ\mathbf{LM}\cong\mathbf{LM}\circ\kappa_{1}$.
\end{proof}

\subsection{Increase of the polynomial degree}

The results formulated in Theorem \ref{thm:Splitting LM} allow us
to understand the effect of the Long-Moody functors on (very) strong
polynomial functors.
\begin{prop}
\label{prop:envoiedegnnn1+1}Let $F$ be a non-null object of $\mathbf{Fct}\left(\mathfrak{U}\boldsymbol{\beta},\mathbb{K}\textrm{-}\mathfrak{Mod}\right)$.
If the functor $F$ is strong polynomial of degree $d$, then:

\begin{enumerate}
\item the functor $\tau_{2}\left(F\right)$ belongs to $\mathcal{P}ol_{d}^{strong}\left(\mathfrak{U}\boldsymbol{\beta},\textrm{\ensuremath{\mathbb{K}}-}\mathfrak{Mod}\right)$;
\item the functor $\mathbf{LM}\left(F\right)$ belongs to $\mathcal{P}ol_{d+1}^{strong}\left(\mathfrak{U}\boldsymbol{\beta},\textrm{\ensuremath{\mathbb{K}}-}\mathfrak{Mod}\right)$.
\end{enumerate}
\end{prop}

\begin{proof}
We prove these two results by induction on the degree of polynomiality.
For the first result, it follows from the commutation property $5$
of Proposition \ref{prop:lemmecaract} for $\tau_{2}$. For the second
result, let us first consider $F$ a strong polynomial functor of
degree $0$. By Theorem \ref{thm:Splitting LM}, we obtain that $\delta_{1}\mathbf{LM}\left(F\right)\cong\tau_{2}\left(F\right)$.
Therefore $\mathbf{LM}\left(F\right)$ is a strong polynomial functor
of degree less than or equal to $1$. Now, assume that $F$ is a strong
polynomial functor of degree $n\geq0$. By Theorem \ref{thm:Splitting LM}:\textit{
$\delta_{1}\mathbf{LM}\left(F\right)\cong\mathbf{LM}\left(\delta_{1}F\right)\oplus\tau_{2}\left(F\right)$}.
By the inductive hypothesis and the result on $\tau_{2}$, we deduce
that $\mathbf{LM}\left(F\right)$ is a strong polynomial functor of
degree less than or equal to $n+1$.
\end{proof}
\begin{cor}
\label{cor:LMpoly}For all natural numbers $d$, the endofunctor $\mathbf{LM}$
of $\mathbf{Fct}\left(\mathfrak{U}\boldsymbol{\beta},\mathbb{K}\textrm{-}\mathfrak{Mod}\right)$
restricts to a functor:
\[
\mathbf{LM}:\mathcal{P}ol_{d}^{strong}\left(\mathfrak{U}\boldsymbol{\beta},\textrm{\ensuremath{\mathbb{K}}-}\mathfrak{Mod}\right)\longrightarrow\mathcal{P}ol_{d+1}^{strong}\left(\mathfrak{U}\boldsymbol{\beta},\textrm{\textrm{\ensuremath{\mathbb{K}}-}\ensuremath{\mathfrak{Mod}}}\right).
\]
\end{cor}

\begin{cor}
\label{cor:Main result}Let $d$ be a natural number and $F$ be an
object of $\mathcal{P}ol_{d}^{strong}\left(\mathfrak{U}\boldsymbol{\beta},\textrm{\ensuremath{\mathbb{K}}-}\mathfrak{Mod}\right)$
such that the strong polynomial degree of $\tau_{2}\left(F\right)$
is equal to $d$. Then, the functor $\mathbf{LM}\left(F\right)$ is
a strong polynomial functor of degree equal to $d+1$.
\end{cor}

\begin{thm}
\label{thm:Main result2}Let $d$ be a natural number and $F$ be
an object of $\mathcal{VP}ol_{d}\left(\mathfrak{U}\boldsymbol{\beta},\textrm{\ensuremath{\mathbb{K}}-}\mathfrak{Mod}\right)$
of degree equal to $d$. Then, the functor $\mathbf{LM}\left(F\right)$
is a very strong polynomial functor of degree equal to $d+1$.
\end{thm}

\begin{proof}
Using Lemma \ref{lem:hypotheseverystrongok}, it follows from Corollary
\ref{cor:Main result} that $\mathbf{LM}\left(F\right)$ is a strong
polynomial functor of degree equal to $n+1$.  Since the functor $\mathbf{LM}$
commutes with the evanescence functor $\kappa_{1}$ by Theorem \ref{thm:commutationkappaLM},
we deduce that $\left(\kappa_{1}\circ\mathbf{LM}\right)\left(F\right)\cong\left(\mathbf{LM}\circ\kappa_{1}\right)\left(F\right)=0$.
Moreover, using Theorem \ref{thm:Splitting LM}, we have:
\[
\left(\kappa_{1}\circ\left(\delta_{1}\circ\mathbf{LM}\right)\right)\left(F\right)\cong\left(\kappa_{1}\circ\tau_{2}\right)\left(F\right)\bigoplus\left(\kappa_{1}\circ\left(\mathbf{LM}\circ\delta_{1}\right)\right)\left(F\right).
\]
Therefore, the fact that $\tau_{2}$ commutes with the evanescence
functor $\kappa_{1}$ (see the commutation property $6$ of Proposition
\ref{prop:lemmecaract}) and Theorem \ref{thm:commutationkappaLM}
together imply that:
\[
\left(\kappa_{1}\circ\left(\delta_{1}\circ\mathbf{LM}\right)\right)\left(F\right)\cong\left(\tau_{2}\circ\kappa_{1}\right)\left(F\right)\bigoplus\left(\mathbf{LM}\circ\left(\kappa_{1}\circ\delta_{1}\right)\right)\left(F\right).
\]
The result then follows from the fact that $F$ is an object of $\mathcal{VP}ol_{n}\left(\mathfrak{U}\boldsymbol{\beta},\textrm{\ensuremath{\mathbb{K}}-}\mathfrak{Mod}\right)$
and $\tau_{2}$ is a reduced endofunctor of the category\textit{ }$\mathbf{Fct}\left(\mathfrak{U}\boldsymbol{\beta},\textrm{\ensuremath{\mathbb{K}}-}\mathfrak{Mod}\right)$.
\end{proof}
\begin{example}
By Proposition \ref{prop:degre0verystrong}, $\mathfrak{X}$ is a
very strong polynomial functor of degree $0$. Now applying the Long-Moody
functor $\mathbf{LM}_{1}$, we proved in Proposition \ref{prop:recoveringunredBurau}
that $t^{-1}\mathbf{LM}_{1}\left(t\mathfrak{X}\right)$ is naturally
equivalent to $\mathfrak{Bur}_{t^{2}}$, which is very strong polynomial
of degree $1$ by Proposition \ref{prop:BurTYMverystrong}.
\end{example}

\subsection{Other properties of the Long-Moody functors}

We have proven in the previous  section that a Long-Moody functor
sends (very) strong polynomial functors to (very) strong polynomial
functors. We can also prove that a (very) strong polynomial functor
in the essential image of a Long-Moody functor is necessarily the
image of another strong polynomial functor.
\begin{prop}
\label{prop:degncomesfromdegn+1}Let $d$ be a natural number. Let
$F$ be a strong polynomial functor of degree $d$ in the category
$\mathbf{Fct}\left(\mathfrak{U}\boldsymbol{\beta},\textrm{\ensuremath{\mathbb{K}}-}\mathfrak{Mod}\right)$.
Assume that there exists an object $G$ of the category $\mathbf{Fct}\left(\mathfrak{U}\boldsymbol{\beta},\mathbb{K}\textrm{-}\mathfrak{Mod}\right)$
such that $\mathbf{LM}\left(G\right)=F$. Then, the functor $G$ is
a strong polynomial functor of degree less than or equal to $d+1$
in the category $\mathbf{Fct}\left(\mathfrak{U}\boldsymbol{\beta},\textrm{\ensuremath{\mathbb{K}}-}\mathfrak{Mod}\right)$.
\end{prop}

\begin{proof}
It follows from Theorem \ref{thm:Splitting LM} that:
\[
\delta_{1}F\cong\tau_{2}\left(G\right)\oplus\left(\mathbf{LM}\circ\delta_{1}\right)\left(G\right).
\]
According to Corollary \ref{Corlem:directsummand}, the functor $\tau_{2}\left(G\right)$
is an object of the category $\mathcal{P}ol_{d-1}^{strong}\left(\mathfrak{U}\boldsymbol{\beta},\textrm{\ensuremath{\mathbb{K}}-}\mathfrak{Mod}\right)$,
and because of Lemma \ref{lem:varepsiFdecreasedegree} the functor
$G$ is an object of the category $\mathcal{P}ol_{d+1}^{strong}\left(\mathfrak{U}\boldsymbol{\beta},\textrm{\ensuremath{\mathbb{K}}-}\mathfrak{Mod}\right)$.
\end{proof}
\begin{prop}
\label{prop:nonessentialsurject}The Long-Moody functor $\mathbf{LM}:\mathbf{Fct}\left(\boldsymbol{\beta},\mathbb{K}\textrm{-}\mathfrak{Mod}\right)\longrightarrow\mathbf{Fct}\left(\boldsymbol{\beta},\mathbb{K}\textrm{-}\mathfrak{Mod}\right)$
is not essentially surjective.
\end{prop}

\begin{proof}
Let $l$ be a natural number. Let $E_{l}:\mathfrak{U}\boldsymbol{\beta}\longrightarrow\mathbb{K}\textrm{-}\mathfrak{Mod}$
be the functor which factorizes through the category $\mathbb{N}$,
such that $E_{l}\left(n\right)=\mathbb{K}^{\oplus n^{l}}$ for all
natural numbers $n$ and for all $\left[n'-n,\sigma\right]\in Hom_{\mathfrak{U}\boldsymbol{\beta}}\left(n,n'\right)$
(with $n$, $n'$ natural numbers such that $n'\geq n$), $E_{l}\left(\left[n'-n,\sigma\right]\right)=\iota_{\mathbb{C}\left[t^{\pm1}\right]^{\oplus n'^{l}-n^{l}}}\oplus id_{\mathbb{C}\left[t^{\pm1}\right]^{\oplus n^{l}}}$.
In particular, for all natural numbers $n,$ for every Artin generator
$\sigma_{i}$ of $\mathbf{B}_{n}$, $E_{l}\left(\sigma_{i}\right)=id_{\mathbb{K}^{\oplus n^{l}}}$.
It inductively follows from this definition and direct computations
that $E_{l}$ is a very strong polynomial functor of degree $l$.

Let us assume that $\mathbf{LM}$ is essentially surjective. Hence,
there exists an object $F$ of \textit{$\mathbf{Fct}\left(\boldsymbol{\beta},\mathbb{K}\textrm{-}\mathfrak{Mod}\right)$
}such that $\mathbf{LM}\left(F\right)\cong E_{l}$. Because of the
definition of $\mathbf{LM}\left(F\right)$ on morphisms (see Theorem
\ref{Thm:LMFunctor}), this implies that for all natural numbers $n$
and for all $\sigma\in\mathbf{B}_{n}$, $a_{n}\left(\sigma\right)=id_{n}$.
Also, if $\mathbf{LM}$ is essentially surjective, there exists an
object $T$ of the category $\mathbf{Fct}\left(\boldsymbol{\beta},\mathbb{K}\textrm{-}\mathfrak{Mod}\right)$
such that we can recover the Burau functor from $\mathbf{LM}\left(T\right)$,
ie something like $\alpha\mathbf{LM}\left(T\right)$ (see Notation
\ref{nota:yX}) with $\alpha\in\mathbb{K}$. We deduce from the definition
of $\mathbf{LM}\left(T\right)$ on objects and morphisms that for
all $n\geq1$, $T\left(n\right)=\mathbb{K}$ and for all generator
$\sigma_{i}$ of $\mathbf{B}_{n}$:
\[
\mathbf{LM}\left(T\right)\left(\sigma_{i}\right)=T\left(\sigma_{i}\right)\cdot Id_{n}.
\]
Then necessarily, for all $i\in\left\{ 1,\ldots,n\right\} $, $T\left(\sigma_{i}\right)=\delta$
such that $\delta^{2}=t$ and we consider $\delta^{-1}\mathbf{LM}\left(T\right)$.
We deduce that there exists a natural transformation $\omega:\delta^{-1}\mathbf{LM}\left(T\right)\overset{\cong}{\rightarrow}\mathfrak{Bur}_{t}$.
This contradicts the fact that for all $\sigma\in\mathbf{B}_{n}$,
$a_{n}\left(\sigma\right)=id_{n}$.
\end{proof}
\begin{rem}
The proof of Proposition \ref{prop:nonessentialsurject} shows in
particular that a Long-Moody functor $\mathbf{LM}$ is not essentially
surjective on very strong polynomial functors in any degree.
\end{rem}

In \cite[Section 4.7, Open Problem 7]{BirmanBrendlesurvey}, Birman
and Brendle ask ``whether all finite dimensional unitary matrix representations
of $\mathbf{B}_{n}$ arise in a manner which is related to the construction''
recalled in Theorem \ref{Thm:LMFunctor}. Since the Tong-Yang-Ma and
unreduced Burau representations recalled in Theorem \ref{thm:resulttym}
are unitary representations, the proof of Proposition \ref{prop:nonessentialsurject}
shows that any Long-Moody functor (and especially the one based on
the version of the construction of Theorem \ref{Thm:LMFunctor}) cannot
provide all the functors encoding unitary representations. Therefore,
we refine the problem asking whether all functors encoding families
of finite dimensional unitary representations of braid groups lie
in the image of a Long-Moody functor.
\begin{rem}
Another question is to ask whether we can directly obtain the reduced
Burau functor $\overline{\mathfrak{Bur}}_{t}$ by a Long-Moody functor.
Recall that for all natural numbers $n$, $\overline{\mathfrak{Bur}}_{t}\left(n\right)=\mathbb{C}\left[t^{\pm1}\right]^{\oplus n-1}$
and $\mathbf{LM}\left(F\right)\left(n\right)\cong\left(F\left(n+1\right)\right)^{\oplus n}$
for any Long-Moody functor $\mathbf{LM}$ and any object $F$ of $\mathbf{Fct}\left(\mathfrak{U}\boldsymbol{\beta},\mathbb{K}\textrm{-}\mathfrak{Mod}\right)$
(see Remark \ref{rem:basisfunctor}). Therefore, for dimensional considerations
on the objects, it is clear that we have to consider a modified version
of the Long-Moody construction. This modification would be to take
the tensor product with $\mathcal{I}_{\mathbf{F}_{n-1}}$ on $\mathbf{F}_{n-1}$,
the $\mathbb{K}$-module $F\left(n+1\right)$ being a $\mathbb{K}\left[\mathbf{F}_{n-1}\right]$-module
using a morphism $\mathbf{F}_{n-1}\rightarrow\left(\mathbf{F}_{n-1}\underset{a_{n}'}{\rtimes}\mathbf{B}_{n+1}\right)\rightarrow\mathbf{B}_{n+1}$
for all natural numbers $n$, where $a_{n}':\mathbf{B}_{n+1}\rightarrow Aut\left(\mathbf{F}_{n-1}\right)$
is a group morphism.
\end{rem}

\bibliographystyle{plain}
\bibliography{bibliographiethese}

\lyxaddress{\textsc{IRMA, Université de Strasbourg, 7 rue René Descartes, 67084
Strasbourg Cedex, France}\\
\textit{E-mail address: }\texttt{soulie@math.unistra.fr}}
\end{document}